\documentclass[a4paper]{amsart}

\usepackage[normalem]{ulem}
\usepackage{amscd}
\usepackage{verbatim}

\newcommand\numberingtheoremsectionyesno{yes}
\newcommand\numberingequationsectionyesno{yes}
\newcommand\pagesizeextendednormal{extended}
\newcommand\reportudemathyesno{no}
\newcommand\reportudemathnumber{SM-UDE-???}
\newcommand\reportudemathyear{2017}
\newcommand\reportudematheingang{\mydate}

\newcommand\mytitle{\Large Time-Harmonic Electro-Magnetic Scattering\\ 
in Exterior Weak Lipschitz Domains with Mixed Boundary Conditions}
\newcommand\mytitlerepude{}
\newcommand\myshorttitle{Time-Harmonic Maxwell Equations}
\newcommand\myauthorone{Frank Osterbrink}
\newcommand\myauthortwo{Dirk Pauly}
\newcommand\myauthors{\myauthorone\quad\&\quad\myauthortwo}
\newcommand\myaddressone{Fakult\"at f\"ur Mathematik,
Universit\"at Duisburg-Essen, Campus Essen, Germany}

\newcommand\myemailone{frank.osterbrink@uni-due.de}
\newcommand\myemailtwo{dirk.pauly@uni-due.de}
\newcommand\mykeywords{Maxwell equations, radiating solutions, exterior boundary value problems, polynomial decay, mixed boundary conditions, weighted Sobolev spaces, Hodge-Helmholtz decompositions}
\newcommand\mysubjclass{35Q60, 78A25, 78A30}
\newcommand\mydate{\today; Corresponding Author: Dirk Pauly}
\newcommand\mythanks{}

\usepackage[mathscr]{eucal}
\usepackage[english]{babel}
\usepackage{exscale,ifthen}
\usepackage{amsmath,amsfonts,amssymb,amsthm,amscd,amsbsy}
\usepackage{bbm,nicefrac,mathtools}
\usepackage{enumitem}
\usepackage{tikz-cd}
\usepackage{tikz,graphicx,color}
\usepackage{caption}[2008/08/24]
\usepackage{fancyhdr,array}
\usepackage{oldgerm}

\DeclareFontFamily{U}{mathx}{\hyphenchar\font45}
\DeclareFontShape{U}{mathx}{m}{n}{
      <5> <6> <7> <8> <9> <10>
      <10.95> <12> <14.4> <17.28> <20.74> <24.88>
      mathx10
      }{}
\DeclareSymbolFont{mathx}{U}{mathx}{m}{n}
\DeclareFontSubstitution{U}{mathx}{m}{n}
\DeclareMathAccent{\widecheck}{0}{mathx}{"71}
\DeclareMathAccent{\wideparen}{0}{mathx}{"75}

\newcommand{\preprintudemath}[5]{
\thispagestyle{empty}
\Large
\begin{center}SCHRIFTENREIHE DER FAKULT\"AT F\"UR MATHEMATIK\end{center}
\vspace*{5mm}
\begin{center}#1\end{center}
\vspace*{5mm}
\begin{center}by\end{center}
\begin{center}#2\end{center}
\vspace*{5mm}
\begin{center}#3\hspace{80mm}#4\end{center}
\newpage
\thispagestyle{empty}
\vspace*{210mm}
Received: #5
\newpage
\addtocounter{page}{-2}
\normalsize}

\ifthenelse{\equal{\pagesizeextendednormal}{extended}}
{\setlength{\textwidth}{16cm}\setlength{\textheight}{22.5cm}
\setlength{\oddsidemargin}{-0.2cm}
\setlength{\evensidemargin}{-0.2cm}}{}
\ifthenelse{\equal{\numberingequationsectionyesno}{yes}}
		   {\numberwithin{equation}{section}}{}

\makeatletter
\newcommand{\leqnomode}{\tagsleft@true}
\newcommand{\reqnomode}{\tagsleft@false}
\makeatother

\newcommand{\dsp}{\displaystyle}
\newcommand{\ovl}[1]{\overline{#1}}

\newcommand{\setb}[2]{\big\{#1\,\big|\,#2\big\}}
\newcommand{\setB}[2]{\Big\{#1\,\Big|\,#2\Big\}}

\newcolumntype{L}[1]{>{\raggedright\arraybackslash}p{#1}}
\newcolumntype{C}[1]{>{\centering\arraybackslash}p{#1}}
\newcolumntype{R}[1]{>{\raggedleft\arraybackslash}p{#1}} 

\ifthenelse{\equal{\numberingtheoremsectionyesno}{yes}}
{\newtheorem{lem}{Lemma}[section]}
{\newtheorem{lem}{Lemma}}
\newtheorem{defi}[lem]{Definition}
\newtheorem{theo}[lem]{Theorem}
\newtheorem{cor}[lem]{Corollary}
\newtheorem{rem}[lem]{Remark}

\newcommand{\al}{\alpha}
\newcommand{\da}{\delta}
\newcommand{\eps}{\varepsilon}
\newcommand{\ceta}{\check{\eta}}
\newcommand{\ga}{\gamma}
\newcommand{\ka}{\kappa}
\newcommand{\la}{\lambda}
\newcommand{\om}{\omega}

\newcommand{\Ga}{\mathsf{\Gamma}}
\newcommand{\La}{\mathrm{\Lambda}}
\newcommand{\Laz}{\La_{0}}
\newcommand{\tLaz}{\widetilde{\La}_{0}}\newcommand{\Om}{\Omega}
\newcommand{\Omb}{\ovl{\Om}}

\newcommand{\calA}{{\mathcal A}}

\newcommand{\calF}{{\mathcal F}}

\newcommand{\calL}{{\mathcal L}}
\newcommand{\calM}{{\mathcal M}}

\newcommand{\calP}{{\mathcal P}}

\newcommand{\calR}{{\mathcal R}}
\newcommand{\calS}{{\mathcal S}}

\newcommand{\sfC}{{\mathsf C}}
\newcommand{\sfD}{{\mathsf D}}

\newcommand{\sfH}{{\mathsf H}}

\newcommand{\sfL}{{\mathsf L}}

\newcommand{\sfR}{{\mathsf R}}

\newcommand{\sfX}{{\mathsf X}}

\newcommand{\aLR}{
	\font\eqfont=cmr10 scaled 1200
	\font\arrfont=cmsy10 scaled 1200
	\textfont0=\eqfont
	\textfont2=\arrfont
	\protect\Relbar\protect\joinrel\Rightarrow}
\renewcommand{\Longrightarrow}{\aLR}

\newcommand{\To}{\longrightarrow}

\newcommand{\restr}[2]{\left.#1 \right|_{#2}}
\newcommand{\map}[3]{#1:#2\;\longrightarrow\;#3}
\newcommand{\maps}[5]{#1:#2\;\longrightarrow\;#3,\;#4\longmapsto #5}
\newcommand{\Map}[5]	
    {\begin{array}{ccccc}	
    #1 & : & #2 & \longrightarrow & #3 \\[5pt]
    {}   & {}  & #4 & \longmapsto & #5	
    \end{array}}
\newcommand{\multif}[4]
    {\left\{\begin{array}{cl}
    \displaystyle#1 &\text{for }~#2 \\[6pt]
    \displaystyle#3 &\text{for }~#4
    \end{array}\right.}

\DeclareMathOperator{\supp}{supp}
\DeclareMathOperator{\dist}{dist}
\renewcommand{\Re}{\mathrm{Re}\,}
\renewcommand{\Im}{\mathrm{Im}\,}

\newcommand{\sol}{{\mathcal L}_{\om}}
\newcommand{\A}{\mathrm{A}}
\newcommand{\M}{\mathrm{M}}

\newcommand{\G}{\mathrm{G}}

\DeclareMathOperator{\p}{\partial}
\DeclareMathOperator{\pr}{\hspace*{-0.05cm}\p_{\it r}\hspace*{-0.7mm}}

\DeclareMathOperator{\Rot}{Rot}

\DeclareMathOperator{\rot}{rot}

\DeclareMathOperator{\divergenz}{div}
\renewcommand{\div}{\divergenz}

\newcommand{\sm}{\setminus}
\newcommand{\dpl}{\dotplus}

\newcommand{\dlan}{d\la^{n}}
\newcommand{\dlanmo}{d\la^{n-1}_{s}}
\newcommand{\rtil}{\tilde{r}}
\newcommand{\rhat}{\hat{r}}
\newcommand{\rcheck}{\check{r}}

\newcommand{\ttil}{\tilde{t}}
\newcommand{\that}{\hat{t}}

\newcommand{\Gat}{\Ga_{1}}
\newcommand{\Gan}{\Ga_{2}}
\newcommand{\Gatn}{\Ga_{i}}
\newcommand{\Gant}{\Ga_{j}}

\newcommand{\B}{\mathrm{B}}
\newcommand{\C}{\mathbb{C}}
\newcommand{\I}{\mathbb{I}}

\newcommand{\N}{\mathbb{N}}
\newcommand{\reals}{\mathbb{R}}
\newcommand{\U}{\mathrm{U}}
\newcommand{\cU}{\widecheck{\mathrm{U}}}
\newcommand{\Sp}{\mathrm{S}}
\newcommand{\dod}{\mathcal{D}}
\newcommand{\rg}{\mathcal{R}}
\renewcommand{\ker}{\mathcal{N}}
\newcommand{\V}{\mathrm{V}}
\newcommand{\X}{\mathrm{X}}	
\newcommand{\Y}{\mathrm{Y}}
\newcommand{\Z}{\mathbb{Z}}

\newcommand{\Gap}{\G(\rhat,\rtil)}

\newcommand{\Gaprcheck}{\G(1,\check{r})}
\newcommand{\Gaprtil}{\G(1,\rtil)}
\newcommand{\Gaprztil}{\G(r_{0},\rtil)}

\newcommand{\gk}[1]{\mathcal{N}_{\mathsf{gen}}(#1)}
\newcommand{\gs}{\sigma_{\mathsf{gen}}(\calM)}	

\newcommand{\rthree}{\reals^{3}}
\newcommand{\rn}{\reals^{n}}

\newcommand{\rttt}{\reals^{3\times3}}

\newcommand{\Omrtil}{\Om(\rtil)}

\newcommand{\Omda}{\Om(\da)}

\newcommand{\cgen}[3]{\overset{#1}{\sfC}{}^{\mathrm{#2}}_{\mathrm{#3}}}
\newcommand{\Cgen}[2]{\mathring{\sfC}{}^{\mathrm{#1}}_{\mathrm{#2}}}

\newcommand{\ci}{\cgen{}{\substack{\infty\\[-5pt]\phantom{\infty}}}{}}

\newcommand{\cirthree}{\ci(\rthree)}

\newcommand{\cirn}{\ci(\rn)}

\newcommand{\cic}{\Cgen{\substack{\infty\\[-5pt]\phantom{\infty}}}{}}

\newcommand{\cicom}{\cic(\Om)}

\newcommand{\cicrthree}{\cic(\rthree)}

\newcommand{\cicrn}{\cic(\rn)}

\newcommand{\cict}{\cgen{}{\substack{\infty\\[-5pt]\phantom{\infty}}}{\Gat}}
\newcommand{\cicn}{\cgen{}{\substack{\infty\\[-5pt]\phantom{\infty}}}{\Gan}}
\newcommand{\cictn}{\cgen{}{\substack{\infty\\[-5pt]\phantom{\infty}}}{\Gatn}}
\newcommand{\cicnt}{\cgen{}{\substack{\infty\\[-5pt]\phantom{\infty}}}{\Gant}}

\newcommand{\cictom}{\cict(\Om)}
\newcommand{\cicnom}{\cicn(\Om)}
\newcommand{\cictnom}{\cictn(\Om)}
\newcommand{\cicntom}{\cicnt(\Om)}

\newcommand{\lsymb}{\sfL}
\newcommand{\lgen}[3]{\overset{#1}{\lsymb}{}^{\mathrm{#2}}_{\mathrm{#3}}}

\newcommand{\ltloc}{\lgen{}{2}{loc}}

\newcommand{\lo}{\lgen{}{1}{}}
\newcommand{\lt}{\lgen{}{2}{}}
\newcommand{\li}{\lsymb^{\infty}}

\newcommand{\ltlocom}{\ltloc(\Om)}

\newcommand{\ltom}{\lt(\Om)}

\newcommand{\ltlocomb}{\ltloc(\Omb)}

\newcommand{\lorn}{\lo(\rn)}
\newcommand{\ltrn}{\lt(\rn)}

\newcommand{\ltmo}{\lgen{}{2}{-1}}

\newcommand{\lto}{\lgen{}{2}{1}}

\newcommand{\ltms}{\lgen{}{2}{-s}}
\newcommand{\ltsmtwo}{\lgen{}{2}{s-2}}
\newcommand{\ltsmo}{\lgen{}{2}{s-1}}
\newcommand{\lts}{\lgen{}{2}{s}}
\newcommand{\ltspo}{\lgen{}{2}{s+1}}

\newcommand{\lttmo}{\lgen{}{2}{t-1}}
\newcommand{\ltt}{\lgen{}{2}{t}}
\newcommand{\lttpo}{\lgen{}{2}{t+1}}

\newcommand{\ltmoom}{\ltmo(\Om)}

\newcommand{\ltsmoom}{\ltsmo(\Om)}
\newcommand{\ltsom}{\lts(\Om)}

\newcommand{\lttom}{\ltt(\Om)}
\newcommand{\lttpoom}{\lttpo(\Om)}

\newcommand{\ltmorn}{\ltmo(\rn)}

\newcommand{\ltmsrn}{\ltms(\rn)}
\newcommand{\ltsmtworn}{\ltsmtwo(\rn)}
\newcommand{\ltsmorn}{\ltsmo(\rn)}
\newcommand{\ltsrn}{\lts(\rn)}
\newcommand{\ltsporn}{\ltspo(\rn)}

\newcommand{\lttmorn}{\lttmo(\rn)}
\newcommand{\lttrn}{\ltt(\rn)}
\newcommand{\lttporn}{\lttpo(\rn)}

\newcommand{\ltthat}{\lgen{}{2}{\,\that}}

\newcommand{\ltttil}{\lgen{}{2}{\ttil}}
\newcommand{\lttpal}{\lgen{}{2}{t+|\al|}}
\newcommand{\ltomthrees}{\lgen{}{2}{1-3s}}
\newcommand{\ltsmka}{\lgen{}{2}{s-\ka}}
\newcommand{\ltsmm}{\lgen{}{2}{s-m}}
\newcommand{\ltbig}[1]{\lgen{}{2}{>#1}}

\newcommand{\ltLa}{\lgen{}{2}{\La}}

\newcommand{\ltthatom}{\ltthat(\Om)}

\newcommand{\ltttilom}{\ltttil(\Om)}
\newcommand{\lttpalom}{\lttpal(\Om)}

\newcommand{\ltsmkaom}{\ltsmka(\Om)}
\newcommand{\ltsmmom}{\ltsmm(\Om)}
\newcommand{\ltbigom}[1]{\ltbig{#1}(\Om)}

\newcommand{\ltLaom}{\ltLa(\Om)}

\newcommand{\ltomthreesrn}{\ltomthrees(\rn)}

\newcommand{\ltthatGaprztil}{\ltthat(\Gaprztil)}

\newcommand{\ltomda}{\lt(\Omda)}

\newcommand{\ltomrtil}{\lt(\Omrtil)}

\newcommand{\hsymb}{\sfH}
\newcommand{\hgen}[3]{\overset{#1}{\hsymb}{}^{\mathrm{#2}}_{\mathrm{#3}}}

\newcommand{\holoc}{\hgen{}{1}{loc}}

\newcommand{\ho}{\hgen{}{1}{}}

\newcommand{\hosmtwo}{\hgen{}{1}{s-2}}

\newcommand{\hosmtworn}{\hosmtwo(\rn)}

\newcommand{\hot}{\hgen{}{1}{t}}

\newcommand{\hottn}{\hgen{}{1}{t,\Gatn}}

\newcommand{\hmt}{\hgen{}{m}{t}}

\newcommand{\hmtn}{\hgen{}{m}{t,\Gan}}
\newcommand{\hmttn}{\hgen{}{m}{t,\Gatn}}

\newcommand{\hotom}{\hot(\Om)}

\newcommand{\hottnom}{\hottn(\Om)}

\newcommand{\hmtom}{\hmt(\Om)}

\newcommand{\hmttnom}{\hmttn(\Om)}

\newcommand{\Hsymb}{\text{\sf\bfseries H}}
\newcommand{\Hgen}[3]{\overset{#1}{\Hsymb}{}^{\mathrm{#2}}_{\mathrm{#3}}}

\newcommand{\Htwoloc}{\Hgen{}{2}{loc}}

\newcommand{\Htwolocrn}{\Htwoloc(\rn)}

\newcommand{\Ho}{\Hgen{}{1}{}}

\newcommand{\Hozt}{\Hgen{}{1}{\Gat}}
\newcommand{\Hozn}{\Hgen{}{1}{\Gan}}
\newcommand{\Hoztn}{\Hgen{}{1}{\Gatn}}

\newcommand{\Htwo}{\Hgen{}{2}{}}

\newcommand{\Hs}{\Hgen{}{s}{}}

\newcommand{\Hoztom}{\Hozt(\Om)}
\newcommand{\Hoznom}{\Hozn(\Om)}
\newcommand{\Hoztnom}{\Hoztn(\Om)}

\newcommand{\Htworn}{\Htwo(\rn)}

\newcommand{\Hos}{\Hgen{}{1}{s}}

\newcommand{\Htwos}{\Hgen{}{2}{s}}

\newcommand{\Htwosmo}{\Hgen{}{2}{s-1}}

\newcommand{\Htwosrn}{\Htwos(\rn)}

\newcommand{\Hot}{\Hgen{}{1}{t}}

\newcommand{\Hottn}{\Hgen{}{1}{t,\Gatn}}

\newcommand{\Htwot}{\Hgen{}{2}{t}}

\newcommand{\Htwotmo}{\Hgen{}{2}{t-1}}

\newcommand{\Hmt}{\Hgen{}{m}{t}}

\newcommand{\Hmttn}{\Hgen{}{m}{t,\Gatn}}

\newcommand{\Hotom}{\Hot(\Om)}

\newcommand{\Hottnom}{\Hottn(\Om)}

\newcommand{\Hmtom}{\Hmt(\Om)}

\newcommand{\Hmttnom}{\Hmttn(\Om)}

\newcommand{\Hotrn}{\Hot(\rn)}

\newcommand{\Htwotrn}{\Htwot(\rn)}

\newcommand{\Htwotmorn}{\Htwotmo(\rn)}

\newcommand{\Hothat}{\Hgen{}{1}{\that}}

\newcommand{\Htwoms}{\Hgen{}{2}{-s}}
\newcommand{\Htwothat}{\Hgen{}{2}{\that}}

\newcommand{\rsymb}{\sfR}
\newcommand{\rgen}[3]{\overset{#1}{\rsymb}{}^{\mathrm{#2}}_{\mathrm{#3}}}

\newcommand{\rloct}{\rgen{}{}{loc,\Gat}}
\newcommand{\rlocn}{\rgen{}{}{loc,\Gan}}

\newcommand{\rvox}{\rgen{}{}{vox}}

\newcommand{\rvoxom}{\rvox(\Om)}

\newcommand{\rz}{\rgen{}{}{0}}

\newcommand{\rztn}{\rgen{}{}{\Gatn}}

\newcommand{\rztnom}{\rztn(\Om)}

\newcommand{\zrzt}{{}_{0}\rgen{}{}{\Gat}}
\newcommand{\zrzn}{{}_{0}\rgen{}{}{\Gan}}

\newcommand{\zrztom}{\zrzt(\Om)}
\newcommand{\zrznom}{\zrzn(\Om)}

\newcommand{\rs}{\rgen{}{}{s}}

\newcommand{\zrs}{{}_{0}\rgen{}{}{s}}

\newcommand{\rt}{\rgen{}{}{t}}

\newcommand{\rtt}{\rgen{}{}{t,\Gat}}

\newcommand{\rttn}{\rgen{}{}{t,\Gatn}}

\newcommand{\rtom}{\rt(\Om)}

\newcommand{\rttnom}{\rttn(\Om)}

\newcommand{\zrt}{{}_{0}\rgen{}{}{t}}

\newcommand{\zrttn}{{}_{0}\rgen{}{}{t,\Gatn}}

\newcommand{\zrtpotn}{{}_{0}\rgen{}{}{t+1,\Gatn}}

\newcommand{\zrtom}{\zrt(\Om)}

\newcommand{\zrttnom}{\zrttn(\Om)}

\newcommand{\zrtpotnom}{\zrtpotn(\Om)}

\newcommand{\rthat}{\rgen{}{}{\that}}

\newcommand{\rthatom}{\rthat(\Om)}

\newcommand{\Rsymb}{\text{\sf\bfseries R}}
\newcommand{\Rgen}[3]{\overset{#1}{\Rsymb}{}^{\mathrm{#2}}_{\mathrm{#3}}}

\newcommand{\R}{\Rgen{}{}{}}

\newcommand{\Rzt}{\Rgen{}{}{\Gat}}
\newcommand{\Rzn}{\Rgen{}{}{\Gan}}

\newcommand{\Rom}{\R(\Om)}

\newcommand{\Rztom}{\Rzt(\Om)}
\newcommand{\Rznom}{\Rzn(\Om)}

\newcommand{\Rs}{\Rgen{}{}{s}}

\newcommand{\Rst}{\Rgen{}{}{s,\Gat}}

\newcommand{\Rsmo}{\Rgen{}{}{s-1}}

\newcommand{\Rsom}{\Rs(\Om)}

\newcommand{\Rstom}{\Rst(\Om)}

\newcommand{\Rsmoom}{\Rsmo(\Om)}

\newcommand{\Rt}{\Rgen{}{}{t}}

\newcommand{\Rtt}{\Rgen{}{}{t,\Gat}}
\newcommand{\Rtn}{\Rgen{}{}{t,\Gan}}
\newcommand{\Rttn}{\Rgen{}{}{t,\Gatn}}

\newcommand{\Rtom}{\Rt(\Om)}

\newcommand{\Rttom}{\Rtt(\Om)}
\newcommand{\Rtnom}{\Rtn(\Om)}
\newcommand{\Rttnom}{\Rttn(\Om)}

\newcommand{\Rthat}{\Rgen{}{}{\that}}
\newcommand{\Rthatt}{\Rgen{}{}{\that,\Gat}}
\newcommand{\Rthatn}{\Rgen{}{}{\that,\Gan}}

\newcommand{\Rttil}{\Rgen{}{}{\ttil}}
\newcommand{\Rttilt}{\Rgen{}{}{\ttil,\Gat}}
\newcommand{\Rttiln}{\Rgen{}{}{\ttil,\Gan}}

\newcommand{\Rsm}[1]{\Rgen{}{}{<#1}}

\newcommand{\Rsmt}[1]{\Rgen{}{}{<#1,\Gat}}
\newcommand{\Rsmn}[1]{\Rgen{}{}{<#1,\Gan}}

\newcommand{\Rbig}[1]{\Rgen{}{}{>#1}}

\newcommand{\Rthatom}{\Rthat(\Om)}
\newcommand{\Rthattom}{\Rthatt(\Om)}
\newcommand{\Rthatnom}{\Rthatn(\Om)}

\newcommand{\Rttilom}{\Rttil(\Om)}
\newcommand{\Rttiltom}{\Rttilt(\Om)}
\newcommand{\Rttilnom}{\Rttiln(\Om)}

\newcommand{\Rsmom}[1]{\Rsm{#1}(\Om)}

\newcommand{\Rsmtom}[1]{\Rsmt{#1}(\Om)}
\newcommand{\Rsmnom}[1]{\Rsmn{#1}(\Om)}

\newcommand{\Rbigom}[1]{\Rbig{#1}(\Om)}

\newcommand{\dsymb}{\sfD}
\newcommand{\dgen}[3]{\overset{#1}{\dsymb}{}^{\mathrm{#2}}_{\mathrm{#3}}}

\newcommand{\dz}{\dgen{}{}{0}}

\newcommand{\zdzt}{{}_{0}\dgen{}{}{\Gat}}
\newcommand{\zdzn}{{}_{0}\dgen{}{}{\Gan}}

\newcommand{\zdztom}{\zdzt(\Om)}
\newcommand{\zdznom}{\zdzn(\Om)}

\newcommand{\zds}{{}_{0}\dgen{}{}{s}}

\newcommand{\dt}{\dgen{}{}{t}}

\newcommand{\dtt}{\dgen{}{}{t,\Gat}}

\newcommand{\dttn}{\dgen{}{}{t,\Gatn}}

\newcommand{\dtom}{\dt(\Om)}

\newcommand{\dttom}{\dtt(\Om)}

\newcommand{\dttnom}{\dttn(\Om)}

\newcommand{\zdt}{{}_{0}\dgen{}{}{t}}

\newcommand{\zdtt}{{}_{0}\dgen{}{}{t,\Gat}}
\newcommand{\zdtn}{{}_{0}\dgen{}{}{t,\Gan}}
\newcommand{\zdttn}{{}_{0}\dgen{}{}{t,\Gatn}}

\newcommand{\zdtpotn}{{}_{0}\dgen{}{}{t+1,\Gatn}}

\newcommand{\zdttom}{\zdtt(\Om)}
\newcommand{\zdtnom}{\zdtn(\Om)}
\newcommand{\zdttnom}{\zdttn(\Om)}

\newcommand{\zdtpotnom}{\zdtpotn(\Om)}

\newcommand{\zdthatt}{{}_{0}\dgen{}{}{\that,\Gat}} 
\newcommand{\zdthatn}{{}_{0}\dgen{}{}{\that,\Gan}}

\newcommand{\zdthattom}{\zdthatt(\Om)} 
\newcommand{\zdthatnom}{\zdthatn(\Om)}

\newcommand{\Dsymb}{\text{\sf\bfseries D}}
\newcommand{\Dgen}[3]{\overset{#1}{\Dsymb}{}^{\mathrm{#2}}_{\mathrm{#3}}}

\newcommand{\D}{\Dgen{}{}{}}

\newcommand{\Dzn}{\Dgen{}{}{\Gan}}

\newcommand{\Dznt}{\Dgen{}{}{\Gant}}

\newcommand{\Dznom}{\Dzn(\Om)}

\newcommand{\Dzntom}{\Dznt(\Om)}

\newcommand{\Dsn}{\Dgen{}{}{s,\Gan}}

\newcommand{\Dsnom}{\Dsn(\Om)}

\newcommand{\Dt}{\Dgen{}{}{t}}

\newcommand{\Dttn}{\Dgen{}{}{t,\Gatn}}

\newcommand{\Dtom}{\Dt(\Om)}

\newcommand{\Dttnom}{\Dttn(\Om)}

\newcommand{\norm}[1]{\left\|\,#1\,\right\|}

\newcommand{\normltws}[1]{\norm{#1}_{\lt}}

\newcommand{\normltom}[1]{\norm{#1}_{\ltom}}

\newcommand{\normltrn}[1]{\norm{#1}_{\ltrn}}

\newcommand{\normltsmtwows}[1]{\norm{#1}_{\ltsmtwo}}
\newcommand{\normltsmows}[1]{\norm{#1}_{\ltsmo}}
\newcommand{\normltsws}[1]{\norm{#1}_{\lts}}

\newcommand{\normlttws}[1]{\norm{#1}_{\ltt}}

\newcommand{\normltmoom}[1]{\norm{#1}_{\ltmoom}}

\newcommand{\normltsmoom}[1]{\norm{#1}_{\ltsmoom}}
\newcommand{\normltsom}[1]{\norm{#1}_{\ltsom}}

\newcommand{\normlttom}[1]{\norm{#1}_{\lttom}}
\newcommand{\normlttpoom}[1]{\norm{#1}_{\lttpoom}}

\newcommand{\normltmsrn}[1]{\norm{#1}_{\ltmsrn}}
\newcommand{\normltsmtworn}[1]{\norm{#1}_{\ltsmtworn}}
\newcommand{\normltsmorn}[1]{\norm{#1}_{\ltsmorn}}
\newcommand{\normltsrn}[1]{\norm{#1}_{\ltsrn}}
\newcommand{\normltsporn}[1]{\norm{#1}_{\ltsporn}}

\newcommand{\normlttmorn}[1]{\norm{#1}_{\lttmorn}}
\newcommand{\normlttrn}[1]{\norm{#1}_{\lttrn}}
\newcommand{\normlttporn}[1]{\norm{#1}_{\lttporn}}

\newcommand{\normltthatws}[1]{\norm{#1}_{\ltthat}}

\newcommand{\normltthatom}[1]{\norm{#1}_{\ltthatom}}

\newcommand{\normlttpalom}[1]{\norm{#1}_{\lttpalom}}

\newcommand{\normltsmkaom}[1]{\norm{#1}_{\ltsmkaom}}
\newcommand{\normltsmmom}[1]{\norm{#1}_{\ltsmmom}}
\newcommand{\normltLaom}[1]{\norm{#1}_{\ltLaom}}

\newcommand{\normltomthreesrn}[1]{\norm{#1}_{\ltomthreesrn}}

\newcommand{\normltthatGaprztil}[1]{\norm{#1}_{\ltthatGaprztil}}

\newcommand{\normltomda}[1]{\norm{#1}_{\ltomda}}

\newcommand{\normhosmtwows}[1]{\norm{#1}_{\hosmtwo}}

\newcommand{\normhosmtworn}[1]{\norm{#1}_{\hosmtworn}}

\newcommand{\normhmtom}[1]{\norm{#1}_{\hmtom}}

\newcommand{\normHtworn}[1]{\norm{#1}_{\Htworn}}

\newcommand{\normHosws}[1]{\norm{#1}_{\Hos}}

\newcommand{\normHtwosmows}[1]{\norm{#1}_{\Htwosmo}}

\newcommand{\normHtwosrn}[1]{\norm{#1}_{\Htwosrn}}

\newcommand{\normHmtom}[1]{\norm{#1}_{\Hmtom}}

\newcommand{\normHotrn}[1]{\norm{#1}_{\Hotrn}}

\newcommand{\normHtwotrn}[1]{\norm{#1}_{\Htwotrn}}

\newcommand{\normrsws}[1]{\norm{#1}_{\rs}}

\newcommand{\normrtom}[1]{\norm{#1}_{\rtom}}

\newcommand{\normrthatom}[1]{\norm{#1}_{\rthatom}}

\newcommand{\normRom}[1]{\norm{#1}_{\Rom}}

\newcommand{\normRsmoom}[1]{\norm{#1}_{\Rsmoom}}

\newcommand{\normRtom}[1]{\norm{#1}_{\Rtom}}

\newcommand{\normRthatom}[1]{\norm{#1}_{\Rthatom}}

\newcommand{\normdtom}[1]{\norm{#1}_{\dtom}}

\newcommand{\normDtom}[1]{\norm{#1}_{\Dtom}}

\newcommand{\scp}[2]{\langle\,#1\,,#2\,\rangle}

\newcommand{\scpltom}[2]{\scp{#1}{#2}_{\ltom}}
\newcommand{\scpltrn}[2]{\scp{#1}{#2}_{\ltrn}}
\newcommand{\scpltLaom}[2]{\scp{#1}{#2}_{\ltLaom}}

\newcommand{\scpltGaprztil}[2]{\scp{#1}{#2}_{\lt(\Gaprztil)}}
\newcommand{\scpltthatGaprztil}[2]{\scp{#1}{#2}_{\ltthat(\Gaprztil)}}

\newcommand{\scpltomrtil}[2]{\scp{#1}{#2}_{\ltomrtil}}

\newcommand{\ptwomat}[4]{\begin{pmatrix}#1&#2\\#3&#4\end{pmatrix}}

\title[\sc\myshorttitle]{\Large\sf\mytitle}
\author{\myauthorone}
\author{\myauthortwo}
\address{\myaddressone}
\email[\myauthorone]{\myemailone}
\email[\myauthortwo]{\myemailtwo}
\keywords{\mykeywords}
\subjclass{\mysubjclass}
\date{\mydate}
\thanks{\mythanks}

\newcommand{\ssolu}{w}
\newcommand{\bssolu}{\ovl{\ssolu}}
\newcommand{\tssolu}{\ssolu_e}
\newcommand{\btssolu}{\ovl{\ssolu}_e}

\newcommand{\srhs}{g}

\newcommand{\tsrhs}{\srhs_e}

\newcommand{\solu}{u}
\newcommand{\hsolu}{\hat{\solu}}
\newcommand{\tsolu}{\tilde{\solu}}
\newcommand{\csolu}{\check{\solu}}

\newcommand{\rhs}{f}
\newcommand{\hrhs}{\hat{\rhs}}
\newcommand{\trhs}{\tilde{\rhs}}
\newcommand{\crhs}{\check{\rhs}}
\begin{document}
%
%
%
\ifthenelse{\equal{\reportudemathyesno}{yes}}
{\preprintudemath{\mytitlerepude}{\myauthors}{\reportudemathnumber}{\reportudemathyear}{\reportudematheingang}}{}
%
%
\begin{abstract}
This paper treats the time-harmonic electro-magnetic scattering or radiation 
problem governed by Maxwell's equations, i.e., 
\begin{alignat*}{4}
   -\rot H+i\om\eps E&=F\;\;\;&&\text{in}\;\Om,
   \qquad &\quad 
   E\times\nu &=0\;\;\;&&\text{on}\;\Ga_1,\\
   \rot E+i\om\mu H&=G\;\;\;&&\text{in}\;\Om,
   \qquad &\quad 
   H\times\nu &=0\;\;\;&&\text{on}\;\Ga_2,
\end{alignat*}
where $\om\in\C\sm(0)$ and $\Om\subset\rthree$ is an exterior weak Lipschitz 
domain with boundary $\Ga:=\p\Om$ divided into two disjoint parts $\Ga_1$ and 
$\Ga_2$. We will present a solution theory using the framework of polynomially 
weighted Sobolev spaces for the rotation and divergence. For the physically 
interesting case $\om\in\reals\sm(0)$ we will show a Fredholm alternative 
type result to hold using the principle of limiting absorption introduced by 
Eidus in the 1960's. The necessary a-priori-estimate and polynomial decay of 
eigenfunctions for the Maxwell equations will be obtained by transferring well 
known results for the Helmholtz equation using a suitable decomposition of the 
fields $E$ and $H$. The crucial point for existence is a local version of Weck's 
selection theorem, also called Maxwell compactness property.
\end{abstract}
\maketitle
\tableofcontents
%
%
\reqnomode
%
%
\section{Introduction}\label{sec:introduction}
%
%
The equations that describe the behavior of electro-magnetic vector fields in some space-time domain
$I\times\Om\subset\reals\times\rthree$, first completely formulated by J.\,C. Maxwell in 1864, are
\begin{alignat*}{3}
	-\rot H+\p_{t}D&=J,
	\qquad\quad & 
	\rot E+\p_{t}B&=0,
	\qquad\quad &
	&\text{ in }I\times\Om,\\
	\div D&=\rho,
	\qquad\quad & 
	\div B&=0,
	\qquad\quad &
	&\text{ in }I\times\Om,
\end{alignat*}
where $E$,\,$H$ are the electric resp.\;magnetic field, $D$,\,$B$ represent the 
displacement current and magnetic induction and $J$,\,$\rho$ describe the current 
density resp.\;the charge density. Excluding, e.g., ferromagnetic 
resp. ferroelectric materials, the parameters linking $E$ and $H$ with $D$ and 
$B$ are often assumed to be of the linear form $D=\eps E$ and $B=\mu H$, where 
$\eps$ and $\mu$ are matrix-valued functions describing the permittivity and 
permeability of the medium filling $\Om$. Here we are especially interested in 
the case of an exterior domain $\Om\subset\rthree$, i.e., a connected open subset 
with compact complement. Applying the divergence to the first two equations we 
see that the latter two equations are implicitly included in the first two and 
may be omitted. Hence, neglecting the static case, Maxwell's equations reduce to 
\begin{alignat*}{3}
	-\rot H+\p_{t}\hspace*{-0.04cm}\big(\hspace*{0.02cm}\eps E\hspace*{0.02cm}\big)&=F,
	\qquad\quad & 
	\rot E+\p_{t}\hspace*{-0.04cm}\big(\hspace*{0.02cm}\mu H\hspace*{0.02cm}\big)&=G,
	\qquad\quad &
	&\text{ in }I\times\Om,
\end{alignat*}
with arbitrary right hand sides $F$,\,$G$. Among the wide range of phenomena 
described by these equations one important case is the discussion 
of\,\emph{``time-harmonic''} electro-magnetic fields where all fields vary 
sinusoidally in time with a single frequency $\om\in\C\sm(0)$, i.e.,
\begin{align*}
	E(t,x)=e^{i\om t}E(x), \qquad H(t,x)=e^{i\om t}H(x), 
	\qquad
	G(t,x)=e^{i\om t}G(x), \qquad F(t,x)=e^{i\om t}F(x)	\,.
\end{align*}
Substituting this ansatz into the equations (\,or using Fourier transformation in 
time\,) and assuming that $\eps$ and $\mu$ are time-independent we are lead to 
what is called \emph{``time-harmonic Maxwell's equations''}:
\begin{align}\label{equ:int_max-sys}
	\rot E+i\om\mu H=G,\qquad\quad-\rot H+i\om\eps E=F,\qquad\quad\text{in }\Om\,.
\end{align}
This system equipped with suitable boundary conditions describes, e.g., the 
scattering of time-harmonic electro-magnetic waves which is of high interest in 
many applications like geophysics, medicine, electrical engineering, biology and 
many others.\\[12pt]
First existence results concerning boundary value problems for the time-harmonic
Maxwell system in bounded and exterior domains have been given by M\"uller 
\cite{muller_behavior_1954},\,\cite{muller_randwertprobleme_1952}. He studied 
isotropic and homogeneous media and used integral equation methods.  
Using alternating differential forms, Weyl \cite{weyl_naturlichen_1952} 
investigated these equations on Riemannian manifolds of arbitrary dimension, while 
Werner \cite{werner_randwertprobleme_1965} was able to transfer M\"uller's 
results to the case of inhomogeneous but isotropic media. However, for general 
inhomogeneous anisotropic media and arbitrary exterior domains, boundary integral 
methods are less useful since they heavily depend on the explicit knowledge of 
the fundamental solution and strong assumptions on boundary regularity. That is 
why Hilbert space methods are a promising alternative. Unfortunately, 
Maxwell's equations are non elliptic, hence it is in general not possible to 
estimate all first derivatives of a solution. In \cite{leis_zur_1968} Leis could 
overcome this problem by transforming the boundary value problem for Maxwell's 
system into a boundary value problem for the Helmholtz equation, assuming that 
the medium filling $\Om$, is inhomogeneous and anisotropic within a bounded 
subset of $\Om$. Nevertheless, he still needed boundary regularity 
to gain equivalence of both problems. But also for nonsmooth boundaries Hilbert 
space methods are expedient. In fact, as shown by Leis
\cite{leis_aussenraumaufgaben_1974}, it is sufficient that $\Om$ satisfies a 
certain selection theorem, later called 
\emph{Weck's selection theorem} or \emph{Maxwell compactness property}, 
which holds for a class of boundaries much larger than those accessible by the 
detour over $\ho$ (\,cf. Weck \cite{weck_maxwells_1974}, Costabel 
\cite{costabel_remark_1990} and Picard, Weck, Witsch 
\cite{picard_time-harmonic_2001}\,). See \cite{leis_initial_2013} for a detailed 
monograph and \cite{bauer_maxwell_2016} for the most recent result and an 
overview. The most recent result regarding a solution theory is due to Pauly 
\cite{pauly_polynomial_2012} (\,see also 
\cite{pauly_niederfrequenzasymptotik_2003}\,) and in its structure comparable to 
the results of Picard \cite{picard_wellenausbreitung_1982} and Picard, Weck \& 
Witsch \cite{picard_time-harmonic_2001}.
While all these results above have been obtained for full boundary conditions, in 
the present paper we study the case of mixed boundary conditions. More precisely, 
we are interested in solving the system \eqref{equ:int_max-sys} for 
$\om\in\C\sm(0)$ in an exterior domain $\Om\subset\rthree$, where we assume 
that $\Ga:=\partial\hspace*{0.01cm}\Om$ is decomposed into two relatively open 
subsets $\Ga_1$ and its complement $\Ga_2:=\Ga\sm\ovl{\Ga}_{1}$ and impose 
homogeneous boundary 
conditions, which in classical terms can be written as 
\begin{align}
	\nu\times E=0\text{ on }\Ga_{1},
	\qquad\qquad
	\nu\times H=0\text{ on }\Ga_{2},
    \qquad\qquad
    (\,\nu:\text{ outward unit normal}\,).
\end{align}
Conveniently, we can apply the same methods as in \cite{pauly_low_2006} 
(\,see also Picard, Weck \& Witsch \cite{picard_time-harmonic_2001}, Weck \& 
Witsch \cite{weck_generalized1_1997},\,\cite{weck_complete_1992}\,) to 
construct a solution. Indeed, most of the proofs carry over practically verbatim. 
For $\om\in\C\sm\reals$ the solution theory is obtained by standard Hilbert space 
methods as $\om$ belongs to the resolvent set of the Maxwell operator. In the 
case of $\om\in\reals\sm(0)$, i.e., $\om$ is in the continuous spectrum of the 
Maxwell operator, we use the limiting absorption principle introduced by Eidus 
\cite{eidus_principle_1965} and approximate solutions to $\om\in\reals\sm(0)$ 
by solutions corresponding to $\om\in\C\sm\reals$. This will be sufficient to 
show a generalized Fredholm alternative (\,cf. our main result, Theorem 
\ref{theo:bvp_fred-alt}\,) to hold. The essential ingredients needed 
for the limit process are\\[-8pt]
\begin{itemize}[itemsep=4pt]
	\item the polynomial decay of eigensolutions, 
	\item an a-priori-estimate for solutions corresponding to nonreal 	
		  frequencies, 
	\item a Helmholtz-type decomposition,
	\item and \emph{Weck's local selection theorem (WLST)}, that is,  
		  \begin{align*}
    		\Rztom\cap\eps^{-1}\Dznom
    		\xhookrightarrow{\hspace{0.4cm}}
    		\ltlocomb\text{ is compact.}\\[-8pt]
		  \end{align*}
\end{itemize}
While the first two are obtained by transferring well known results for the 
scalar Helmholtz equation to the time-harmonic Maxwell equations using a suitable 
decomposition of the fields $E$ and $H$, Lemma 4.1, the last one is an assumption 
on the quality of the boundary. As we will see, WLST is an immediate consequence 
of \emph{Weck's selection theorem (WST)}, i.e.,
\begin{align*}
    \Rzt(\Theta)\cap\eps^{-1}\Dzn(\Theta)
    \xhookrightarrow{\hspace{0.4cm}}
    \lt(\Theta)\text{ is compact,}
\end{align*}
which holds in bounded weak Lipschitz domains $\Theta\subset\rthree$, but fails 
in unbounded such as exterior domains (\,cf. Bauer, Pauly, Schomburg 
\cite{bauer_maxwell_2016} and the references therein\,). For strong Lipschitz-
domains see Jochmann \cite{jochmann_compactness_1997} and Fernandes, Gilardis 
\cite{fernandes_magnetostatic_1997}.
%
%
\section{Preliminaries and Notations}
%
%
Let $\Z$, $\N$, $\reals$ and $\C$ be the usual sets of integers, natural, real 
and complex numbers, respectively. Furthermore, let $i$ be the imaginary unit, 
$\Re z$, $\Im z$ and $\ovl{z}$ real part, imaginary part and complex conjugate 
of $z\in\C$, as well as 
\begin{align*}
	\hspace*{-0.7cm}\reals_{+}:=\setb{s\in\reals}{s>0},\quad
	\C_{+}:=\setb{z\in\C}{\Im z\geq 0},
	\quad\text{and,}\quad
	\I:=\setb{(2m+1)/2}{m\in\Z\sm(0)}\,.	
\end{align*}
For $x\in\rn$ with $x\neq 0$ we set $r(x):=|\,x\,|$ and $\xi(x):=x/|\,x\,|$ 
\big(\;$|\,\cdot\,|$\,: Euclidean norm in $\rn$\,\big). Moreover, 
$\U(\rtil)$\;resp.\;$\B(\rtil)$ indicate the open resp.\;closed ball of radius 
$\rtil$ in $\rn$ centered in the origin and we define
\begin{align*}
	\Sp(\rtil):=\B(\rtil)\sm\U(\rtil),
	\qquad
	\cU(\rtil):=\rthree\sm\B(\rtil),
	\qquad
	\G(\rtil,\rhat):=\cU(\rtil)\cap\U(\rhat),
\end{align*}
with $\rhat>\rtil$. If $\map{f}{\X}{\Y}$ is a function mapping $\X$ to $\Y$, 
the restriction of $f$ to a subset $\U\subset\X$ will be\\[1pt] 
marked with $\dsp\restr{f}{U}$ and $\dod(f)$, $\ker(f)$, $\rg(f)$, and $\supp f$ 
denote domain of definition, kernel, range, and support of $f$, respectively. 
Given two functions $\map{f,g}{\reals^{n}}{\C^{k}}$ we write
\begin{align*}
	f=\mathcal{O}(g)\;\text{ for }\;r\To\infty
	\qquad\text{if and only if}\qquad
	\exists\;c>0\;\;\exists\;\rtil>0\;\;\forall\;x\in\cU(\rtil):\quad 
	|\hspace*{0.2mm}f(x)\hspace*{0.2mm}|
	\leq c\cdot g(x)\,.	
\end{align*}
For $\X,\Y$ subspaces of a normed vector space $\V$, $\X+\Y$,
$\X\dpl\Y$ and $\X\oplus\Y$ indicate the sum, the direct sum and the orthogonal 
sum of $\X$ and $\Y$, where in the last case we presume the existence of a scalar 
product $\scp{\,\cdot\,}{\,\cdot\,}_{\V}$ on $\V$. Moreover, 
$\scp{\,\cdot\,}{\,\cdot\,}_{\X\times\Y}$ resp. $\norm{\,\cdot\,}_{\X\times\Y}$ 
denote the natural scalar product resp.\;induced norm on $\X\times\Y$ and if 
$\X=\Y$, we simply use the index $\X$ instead of $\X\times\X$. \\[-6pt]
\noindent
\subsection{General Assumptions and Weighted Sobolev Spaces} 
Unless stated otherwise, from now on and throughout this paper it is assumed 
that $\Om\subset\rthree$ is an exterior weak Lipschitz domain with weak Lipschitz 
interface in the sense of 
\cite[Definition 2.3, Definition 2.5]{bauer_maxwell_2016}, 
which in principle means that $\Ga=\partial\hspace*{0.01cm}\Om$ is a 
Lipschitz-manifold and $\Ga_{1}$\;resp.\;$\Ga_{2}$ are 
Lipschitz-submanifolds of\;$\Ga$. For later purpose we fix $r_{0}>0$ such that 
$\rthree\sm\Om\Subset\U(r_{0})$ and define for arbitrary $\rtil\geq r_{0}$
\begin{align*}
	\Omrtil:=\Om\cap\U(\rtil)\,.
\end{align*}
With $r_{k}:=2^{k}r_{0}$, $k\in\N$ and 
$\tilde{\eta}\in\dsp\ci(\reals)$ such that
\begin{align}\label{equ:prel_cut-off}
	0\leq\tilde{\eta}\leq1,
	\qquad\quad
	\supp\tilde{\eta}\subset(-\infty,2-\da),
	\qquad\quad
	\restr{\tilde{\eta}}{(-\infty,1+\da)}=1,
\end{align}
for some $0<\da<1$, we define functions 
$\eta,\ceta,\eta_{k},\ceta_{k}\in\dsp\cirthree$ by 
\begin{align*}
	\eta(x):=\tilde{\eta}(r(x)/r_{0})\,,
	\quad\;\;
	\ceta(x):=1-\eta(x)\,,
	\quad\;\;
	\eta_{k}(x):=\tilde{\eta}\big(r(x)/r_{k}\big)\,,
	\quad\;\text{resp.}\quad\;
	\ceta_{k}(x):=1-\eta_{k}(x)\,,
\end{align*}
meaning
\begin{alignat*}{7}
	&\supp\eta\subset \B(r_{1})
	\quad &&\text{with}\quad &
	&\eta=1\;\text{on}\;\U(r_{0})\,,
	\qquad &&\qquad &
	&\supp\eta_{k}\subset\U(r_{k+1})
	\quad &&\text{with}\quad &
	&\eta_{k}=1\;\text{on}\;\U(r_{k})\,,\\[-6pt]
	&&&&&\qquad &&\text{resp.}\qquad &&&&&&\\[-6pt]
	&\supp\ceta\subset \cU(r_{0})
	\quad &&\text{with}\quad &
	&\ceta=1\;\text{on}\;\cU(r_{1})\,,
	\qquad &&\qquad &
	&\supp\ceta_{k}\subset \cU(r_{k})
	\quad &&\text{with}\quad &
	&\ceta_{k}=1\;\text{on}\;\cU(r_{k+1})\,.
\end{alignat*}
These functions will later be utilized for particular cut-off procedures.\\[12pt]
Next we introduce our notations for Lebesgue and Sobolev spaces needed in the 
following discussion. Note that we will not indicate whether the elements of 
these spaces are scalar functions or vector fields. This will be always clear 
from the context. The example\footnote{Although the right hand 
sides $0$ and $r^{-2}$ are $\lt(\cU(1))$-functions, we have 
$E=\xi/r\notin\lt(\cU(1))$, but $E\in\ltmo(\cU(1))$} 
\begin{align*}
	E:=\nabla\ln(r)\in\holoc(\cU(1))\,,
	\quad
	\restr{\nu\times E}{\Sp(1)}=0\,,
	\quad
	\rot E=0\in\lt(\cU(1))\,,
	\quad
	\div E=r^{-2}\in\lt(\cU(1))\,,
\end{align*}
shows that a standard $\lt$-setting is not appropriate for exterior domains. Even 
for square-integrable right hand sides we cannot expect to find square-integrable 
solutions. Indeed it turns out that 
we have to work in weighted Lebesgue and Sobolev spaces to develop a solution 
theory.
With $\rho:=(\,1+r^{2}\,)^{1/2}$ we introduce for an arbitrary domain
$\Om\subset\rthree$, $t\in\reals$, and $m\in\N$
\begin{align*}
	\lttom & :=\setB{\ssolu\in\ltlocom}{\rho^{t}\ssolu\in\ltom},\\
  \Hmtom &:=\setB{\ssolu\in\lttom}
   {\forall\;|\alpha|\leq m:\;\p^{\alpha}\hspace*{-0.5mm}\ssolu\in\lttom},\\
  \hmtom &:=\setB{\ssolu\in\lttom}
   {\forall\;|\alpha|\leq m:\;\p^{\alpha}\hspace*{-0.5mm}\ssolu\in\lttpalom},
\end{align*}
\begin{alignat*}{2}
	\Rtom &:=\setB{E\in\lttom}{\rot E\in\lttom}\,,\qquad\quad 
		  &\rtom &:=\setB{E\in\lttom}{\rot E\in\lttpoom},\\
	\Dtom &:=\setB{H\in\lttom}{\div H\in\lttom}\,,\qquad\quad 
		  &\dtom &:=\setB{H\in\lttom}{\div H\in\lttpoom},
\end{alignat*}
where $\alpha=(\alpha_{1},\alpha_{2},\alpha_{3})\in\N^{3}$ is a multi-index and 
$\p^{\alpha}\hspace*{-0.5mm}w:=\p_{1}^{\alpha_{1}}\p_{2}^{\alpha_{2}}
\p_{3}^{\alpha_{3}}\hspace*{-0.5mm}w$, $\rot E$, and $\div H$ are the usual 
distributional or weak derivatives. Equipped with the induced norms
\begin{align*}
	\normlttom{\ssolu}^{2}&:=\normltom{\rho^{t}\ssolu}^{2}\,,\\
	\normHmtom{\ssolu}^{2}&:=\sum_{|\alpha|\leq m}
	 \normlttom{\p^{\alpha}\hspace*{-0.5mm}\ssolu}^{2}\,,\\
	\normhmtom{\ssolu}^{2}&:=\sum_{|\alpha|\leq m}
	 \normlttpalom{\p^{\alpha}\hspace*{-0.5mm}\ssolu}^{2}\,,\\[-20pt]
\end{align*}
\begin{alignat*}{2}
	\normRtom{E}^{2}&:=\normlttom{E}^{2}+\normlttom{\rot E}^{2}\,,
	\qquad\quad 
    &\normrtom{E}^{2}&:=\normlttom{E}^{2}+\normlttpoom{\rot E}^{2}\,,\\
	\normDtom{H}^{2}&:=\normlttom{H}^{2}+\normlttom{\div H}^{2}\,,
	\qquad\quad 
    &\normdtom{H}^{2}&:=\normlttom{H}^{2}+\normlttpoom{\div H}^{2}\,,
\end{alignat*}
they become Hilbert spaces. As usual, the subscript 
``$\mathrm{loc}$'' resp. ``$\mathrm{vox}$'' indicates local 
square-integrability resp.\;bounded support. Please note, that the bold spaces 
with weight $t=0$ correspond to the classical Lebesque and Sobolev spaces and for 
bounded domains {``non-weighted''} and weighted spaces even coincide:
\begin{align*}
	\Om\subset\rthree\text{\;bounded}
	\;\;\Longrightarrow\;\;
	\forall\;t\in\reals:\;\,
	\left\{\begin{aligned}
		\dsp\;\,\Hotom&=\hotom=\hsymb^{1}_{0}(\Om)=\hsymb^{1}(\Om)\\
	    \dsp\;\,\Rtom&=\rtom=\rz(\Om)=\hsymb(\rot,\Om)\\
	    \dsp\;\,\Dtom&=\dtom=\dz(\Om)=\hsymb(\div,\Om)
	       \end{aligned}\right.\\[-9pt]	
\end{align*}
Besides the usual set $\cicom$ of test fields (\,resp.\;test functions\,) we 
introduce   
\begin{align*}
	\cictnom:=\setB{\restr{\varphi}{\Om}}{\varphi\in\cicrthree\;\text{ and }
			  \dist(\supp \varphi, \mathsf{\Ga}_{i})>0\,}\,,\quad \mathrm{i}=1,2
\end{align*}  
to formulate boundary conditions in the weak sense: 
\begin{alignat}{3}\label{equ:prel_weak-bc}
    \Hmttnom &:=\ovl{\dsp\cictnom}^{\,\normHmtom{\cdot}}\,,\quad &
    \Rttnom &:=\ovl{\dsp\cictnom}^{\,\normRtom{\cdot}}\,, \quad &
    \Dttnom &:=\ovl{\dsp \cictnom}^{\,\normDtom{\cdot}}\,,\notag\\[-6pt]
    &&&&&\\[-6pt]
    \hmttnom &:=\ovl{\dsp\cictnom}^{\,\normhmtom{\cdot}}\,,\quad &
    \rttnom &:=\ovl{\dsp\cictnom}^{\,\normrtom{\cdot}}\,, \quad &
    \dttnom &:=\ovl{\dsp \cictnom}^{\,\normdtom{\cdot}}\,.\notag
\end{alignat}
These spaces indeed generalize vanishing scalar, tangential 
and normal Dirichlet boundary conditions even and in particular to boundaries 
for which the notion of a normal vector may not make any sense. 
Moreover, $0$ at the lower left corner denotes vanishing rotation 
resp.\;divergence, e.g.,
\begin{align*}
	\zrtom:=\setb{E\in\rtom}{\rot E=0}\,,\qquad
	\zdttom:=\setb{H\in\dttom}{\div H=0}\,,\qquad\ldots\qquad\,,
\end{align*}    
and if $t=0$ in any of the definitions given above, we will skip the weight, e.g.,
\begin{align*}
	\Hgen{}{m}{}(\Om)=\Hgen{}{m}{0}(\Om)\,,
	\qquad
	\rgen{}{}{\mathsf{\Ga_{1}}}(\Om)=\rgen{}{}{\mathsf{0,\Ga_{1}}}(\Om)\,,
	\qquad 
	\Dgen{}{}{\mathsf{\Ga_{1}}}(\Om)=\Dgen{}{}{\mathsf{0,\Ga_{1}}}(\Om)\,,
	\qquad\ldots\qquad\,.
\end{align*}
Finally we set
\begin{align*}
	\sfX_{\mathrm{<s}}:=\bigcap_{\mathrm{t<s}}\sfX_{\mathrm{t}}
	\qquad\text{and}\qquad
	\sfX_{\mathrm{>s}}:=\bigcup_{\mathrm{t>s}}\sfX_{\mathrm{t}}
	\qquad
	(\,s\in\reals\,)\,,
\end{align*}
for $\X_{\mathrm{t}}$ being any of the spaces above.
If $\Om=\rthree$ we omit the space reference, e.g.,
\begin{align*}
	\Hmt:=\Hmt(\rthree)\,,
	\qquad
	\rtt:=\rtt(\rthree)\,,
	\qquad
	\Dt:=\Dt(\rthree)\,,
	\qquad
	\hmtn:=\hmtn(\rthree)\,,
	\qquad\ldots\qquad.
\end{align*}  
\noindent
The material parameters $\eps$ and $\mu$ are assumed to be $\ka$-admissible 
in the following sense.

\begin{defi}\label{def:bvp_admissible}
	Let $\ka\geq 0$. We call a transformation $\ga$ $\ka$-admissible, if  
	\vspace*{4pt}
	\begin{itemize}[itemsep=6pt]
		\item $\map{\ga}{\Om}{\rttt}$ is an $\li$-matrix field,
		\item $\ga$ is symmetric, i.e.,
			  \begin{align*}
			  		\forall\;E,H\in\ltom:\quad
			  		\scpltom{E}{\ga H}=\scpltom{\ga E}{H}\,,
			  \end{align*}
		\item $\ga$ is uniformly positive definite, i.e.,
			  \begin{align*}
			  	\exists\;c>0\;\;\forall\;E\in\ltom:\quad\scpltom{E}{\ga E}
			  	\geq c\cdot\normltom{E}^2\,,
			  \end{align*}
		\item $\ga$ is asymptotically a multiple of the identity, i.e.,
			  \begin{align*}
			  	\ga=\ga_{0}\cdot\mathbbm{1}+\hat{\ga}\text{ with }\ga_{0}\in
			  	\reals_{+}\text{ and }
			  	\hat{\ga}=\mathcal{O}\big(r^{-\ka}\big)\text{ as }
			  	r\To\infty\,.\\[-8pt]
			  \end{align*}
	\end{itemize}
\end{defi}
\noindent
Then $\eps$,\,$\mu$ are pointwise invertible and $\eps^{-1}$, $\mu^{-1}$ defined 
by
\begin{align*}
	\eps^{-1}(x):=\big(\,\eps(x)\,\big)^{-1}
	\qquad\text{and}\qquad
	\mu^{-1}(x):=\big(\,\mu(x)\,\big)^{-1},\qquad\;x\in\Om\,,
\end{align*}
are also $\ka$-admissible. Moreover,\\[-11pt] 
\begin{align*}
	\scp{\,\cdot\,}{\,\cdot\,}_{\eps}:=\scpltom{\,\eps\,\cdot\,}{\,\cdot\,}
	\quad\;\text{and}\quad\;
	\quad\;\scp{\,\cdot\,}{\,\cdot\,}_{\mu}:=\scpltom{\,\mu\,\cdot\,}{\,\cdot\,}
\end{align*}
define scalar products on $\ltom$ inducing norms equivalent to 
the standard ones. Consequently, 
\begin{align*}
	\lgen{}{2}{\eps}(\Om)
	 :=\big(\,\ltom,\scp{\,\cdot\,}{\,\cdot\,}_{\eps}\,\big)\,,
	\quad\;\lgen{}{2}{\mu}(\Om)
	 :=\big(\,\ltom,\scp{\,\cdot\,}{\,\cdot\,}_{\mu}\,\big)
	\quad\;\text{and}\quad\;
	\ltLaom:=\lgen{}{2}{\eps}(\Om)\times\lgen{}{2}{\mu}(\Om)
\end{align*}
are Hilbert spaces
and we will write
\begin{align*}
	\norm{\cdot}_{\eps},\;\norm{\cdot}_{\mu}\;\norm{\cdot}_{\La}\,,
	\qquad
	\oplus_{\eps},\;\oplus_{\mu},\;\oplus_{\La}
	\qquad\text{and}\qquad
	\perp_{\eps},\;\perp_{\mu},\;\dsp\perp_{\La}
\end{align*}
to indicate the norm, the orthogonal sum and the orthogonal complement in these 
spaces. For further simplification and to shorten notation we also introduce for 
$\eps=\eps_{0}\cdot\mathbbm{1}+\hat{\eps}$ and 
$\mu=\mu_{0}\cdot\mathbbm{1}+\hat{\mu}$ the formal matrix operators
\begin{alignat*}{6}
	&\hspace*{0.6cm}\La:=\ptwomat{\eps}{0}{0}{\mu} 
	 \qquad &&\qquad & 
	&\hspace*{0.6cm}\La^{-1}:=\ptwomat{\eps^{-1}}{0}{0}{\mu^{-1}}
	 \qquad &&\qquad & 
	&\hspace*{0.6cm}\hat{\La}:=\ptwomat{\hat{\eps}}{0}{0}{\hat{\mu}}&&\\[-12pt]
	&&&,&&&&,&&&&\;\,,\\
	&\La\,(E,H)=(\eps E,\mu H) 
	 \qquad &&\qquad & 
	&\La^{-1}\,(E,H)=(\eps^{-1} E, \mu^{-1} H)
	 \qquad &&\qquad & 
	&\hat{\La}\,(E,H)=(\hat{\eps}E,\hat{\mu}H)&&
\end{alignat*}
\begin{alignat*}{4}
	&\hspace*{0.6cm}\Rot:=\ptwomat{0}{-\rot}{\rot}{0}
	\qquad &&\qquad &
	&\M:=i\Lambda^{-1}\Rot
	=\ptwomat{0}{-i\eps^{-1}\rot}{i\mu^{-1}\rot}{0}&&\\[-12pt]
	&&&,&&&&\;\,\;,\\
	&\Rot\,(E,H)=(-\rot H,\rot E)
	\qquad &&\qquad &
	&\hspace*{0.3cm}\M\,(E,H)=(-i\eps^{-1}\rot H,i\mu^{-1}\rot E)&&
\end{alignat*}
\begin{alignat*}{6}
	&\hspace*{0.6cm}\Xi:=\ptwomat{0}{-\xi\times}{\xi\times}{0}
    \qquad &&\qquad &
    &\hspace*{0.6cm}\Laz:=\ptwomat{\eps_{0}}{0}{0}{\mu_{0}}
    \qquad &&\qquad &
    &\hspace*{0.6cm}\tLaz:=\ptwomat{\mu_{0}}{0}{0}{\eps_{0}}&&\\[-12pt]
    &&&,&&&&,&&&&\;\,.\\
    &\Xi\,(E,H)=(-\xi\times H,\xi\times E)
    \qquad &&\qquad &
    &\Laz\,(E,H)=(\eps_{0} E,\mu_{0} H)
    \qquad && \qquad &
    &\tLaz\,(E,H)=(\mu_{0} E,\eps_{0} H)&&\\[-9pt]
\end{alignat*}
Recall $\xi(x)=x/r(x)$.\\[12pt]
We end this section with a Lemma, showing that the spaces defined in
\eqref{equ:prel_weak-bc} indeed generalize vanishing scalar, tangential 
and normal boundary conditions.
\begin{lem}\label{lem:prel_relationships}
	For $t\in\reals$ and $i\in(1,2)$ the following inclusions hold:\\[-8pt]
	\begin{enumerate}[itemsep=4pt,label=$(\alph*)$]
		\item $\;\hmttnom\subset\Hmttnom,
		         \quad
		         \rttnom\subset\Rttnom,
		         \quad
		         \dttnom\subset\Dttnom$
		\item $\;\nabla\Hottnom\subset\zrttnom,
		         \quad
		         \nabla\hottnom\subset\zrtpotnom$
		\item $\;\rot \Rttnom\subset\zdttnom,
		         \quad
		         \rot \rttnom\subset\zdtpotnom$\\[-8pt]
	\end{enumerate}
	Additionally we have for $i,j\in(1,2)$, $i\neq j$:
    \begin{align*}
		\Hottnom &=\boldsymbol{\mathcal{H}}^{1}_{\mathrm{t,\Ga_{i}}}(\Om)
	    		 :=\setB{\,\ssolu\in\Hotom\,}{\;\forall\;\Phi\in\cicntom:\;\;
		         \scpltom{\ssolu}{\div\Phi}
		         =-\scpltom{\nabla \ssolu}{\Phi}\;}\,,\\
	    \Rttnom &=\boldsymbol{\mathcal{R}}_{\mathrm{t,\Ga_{i}}}(\Om)
	    		:=\setB{\,E\in\Rtom\,}{\;\forall\;\Phi\in\cicntom:\;\;
		        \scpltom{E}{\rot \Phi}=\scpltom{\rot E}{\Phi}\;}\,,\\
	    \Dttnom &=\boldsymbol{\mathcal{D}}_{\mathrm{t,\Ga_{i}}}(\Om)
	    		:=\setB{\,H\in\Dtom\,}{\;\forall\;\phi\in\cicntom:\;\;
		        \scpltom{H}{\nabla\phi}=-\scpltom{\div H}{\phi}\;}\,,
    \end{align*}
    and
    \begin{align*}
	    \hottnom &=\mathcal{H}^{1}_{\mathrm{t,\Ga_{i}}}(\Om)
	    		 :=\setB{\,\ssolu\in\hotom\,}{\;\forall\;\Phi\in\cicntom:\;\;
	      		 \scpltom{\ssolu}{\div \Phi}
	      		 =-\scpltom{\nabla \ssolu}{\Phi}\;}\,,\\
        \rttnom &=\mathcal{R}_{\mathrm{t,\Ga_{i}}}(\Om)
        		:=\setB{\,E\in\rtom\,}{\;\forall\;\Phi\in\cicntom:\;\;
		        \scpltom{E}{\rot \Phi}=\scpltom{\rot E}{\Phi}\;}\,,\\
	    \dttnom &=\mathcal{D}_{\mathrm{t,\Ga_{i}}}(\Om)
	    		:=\setB{\,H\in\dtom\,}{\;\forall\;\phi\in\cicntom:\;\;
	    		\scpltom{H}{\nabla \phi}=-\scpltom{\div H}{\phi}\;}\,,
    \end{align*}
    where $\rm(\;$by continuity of the $\lt$-scalar product$ \rm\;)$ 
    we may also replace $\cicntom$ by
    \begin{align*}
    	\Hgen{}{1}{\mathrm{s,\Ga_{j}}}(\Om)\,,\;
    	\Rgen{}{}{\mathrm{s,\Ga_{j}}}(\Om)\,,\;
    	\Dgen{}{}{\mathrm{s,\Ga_{j}}}(\Om)\;
    	\qquad\text{resp.}\qquad 
    	\hgen{}{1}{\mathrm{s,\Ga_{j}}}(\Om)\,,\; 
    	\rgen{}{}{\mathrm{s,\Ga_{j}}}(\Om)\,,\;
    	\dgen{}{}{\mathrm{s,\Ga_{j}}}(\Om)\,,	
    \end{align*}
    with 
	$s+t\geq 0$ resp.\;$s+t\geq -1$.\\[-6pt]
\end{lem}
\begin{proof} 
	As representatives of the arguments we show
	\begin{align*}
		\mathrm{(i)}\;\rot\Rtnom\subset\zdtnom
		\qquad\text{and}\qquad
		\mathrm{(ii)}\;\Rttom=\boldsymbol{\mathcal{R}}_{\mathrm{t,\Gat}}(\Om)\,.
	\end{align*}
	For $E\in\rot\Rtnom$ there exists a sequence 
	$(\mathcal{E}_{n})_{n\in\N}\subset\dsp\cicnom$ such that 
	$\rot\mathcal{E}_{n}\To E$ in $\lttom$. Then 
	\begin{align*}
		\forall\phi\in\cicom:\quad \scpltom{E}{\nabla\phi}
		=\lim_{n\rightarrow\infty}\scpltom{\rot\mathcal{E}_n}{\nabla\phi}
		=-\lim_{n\rightarrow\infty}
		 \scpltom{\div\big(\rot\mathcal{E}_n\big)}{\phi}=0\,,	
	\end{align*}
	hence $E$ has vanishing divergence and $(E_{n})_{n\in\N}$ defined by 
	$E_{n}:=\rot \mathcal{E}_{n}$ satisfies
	\begin{align*}
		(E_{n})_{n\in\N}\subset\cicnom,\quad
		E_{n}\xrightarrow{\;\;\lttom\;\;} E
		\quad\text{ and}\quad
		\div E_{n}=\div\big(\rot\mathcal{E}_{n}\big)
		=0\xrightarrow{\;\;\lttpoom\;\;}0\,.
	\end{align*}
	Thus $E\in\zdtnom$ and we have shown $\mathrm{(i)}$. Let us show 
	$\mathrm{(ii)}$. We clearly have 
	$\Rttom\subset\boldsymbol{\mathcal{R}}_{\mathrm{t,\Gat}}(\Om).$ 
	For the other direction, let 
	$E\in\dsp\boldsymbol{\mathcal{R}}_{\mathrm{t,\Gat}}(\Om)$ and $\delta>0$. 
	Using the cut-off function from above we define $(E_{k})_{k\in\N}$ by 
	$E_{k}:=\eta_{k}E$. Then 
	$E_{k}\in\boldsymbol{\mathcal{R}}_{\tilde{\Ga}_\mathrm{1}}
	\big(\Om(2r_{k})\big)$, $\tilde{\Ga}_{\mathrm{1}}:=\Gat\cup\Sp(2r_{k})$, 
	since for $\Phi\in\dsp\cicn\big(\Om(2r_{k})\big)$ it holds by 
	$\eta_k\Phi\in\cicnom$
	\begin{align*}
		\scp{E_{k}}{\rot\Phi}_{\lt(\Om(2r_{k}))}
		&=\scpltom{\eta_{k}E}{\rot\Phi}\\
		&=\scpltom{E}{\rot(\eta_{k}\Phi)}-\scpltom{E}{\nabla\eta_{k}\times\Phi}\\
		&=\scp{\eta_{k}\rot E+\nabla\eta_{k}\times E}{\Phi}_{\lt(\Om(2r_{k}))}
		 =\scp{\rot E_{k}}{\Phi}_{\lt(\Om(2r_{k}))}\,.	
	\end{align*}
	Now observe, that by means of monotone convergence we have
	\footnote{Here and hereafter, $c>0$ 
	denotes some generic constant only depending on the indicated quantities.}
	\begin{align*}
		\hspace*{-0.75cm}\normRtom{E-E_{k}}
			=\normRtom{\ceta_{k}E}
			\leq c\cdot\Big(\norm{E}_{\Rt(\cU(r_{k}))}
			  +\frac{1}{2^k}\cdot\normlttom{E}\Big)
			\xrightarrow{\hspace*{0.6cm}}0\,,	
	\end{align*}
	hence we can choose $\hat{k}>0$ such that 
	$\normRtom{E-E_{\hat{k}}}<\delta/2\,.$
	Moreover $\Om(2r_{\hat{k}})=\Om\cap\U(2r_{\hat{k}})$ is a bounded weak 
	Lipschitz domain and therefore 
	(\,cf. \cite[Section 3.3]{bauer_maxwell_2016}\,)
	$\boldsymbol{\mathcal{R}}_{\tilde{\Ga}_\mathrm{1}}
	 \big(\hspace*{0.02cm}\Om(2r_{\hat{k}})\big)
	 =\Rsymb_{\tilde{\Ga}_\mathrm{1}}
	 \big(\hspace*{0.02cm}\Om(2r_{\hat{k}})\big),$
	yielding the existence of some
	$\Psi\in\dsp\sfC^{\substack{\infty\\[-5pt]\phantom{\infty}}}
	 _{\mathrm{\tilde{\Ga}_{1}}}\big(\hspace*{0.02cm}\Om(2r_{\hat{k}})\big)$ 
	such that 
	\begin{align*}
		\norm{E_{\hat{k}}-\Psi}_{\Rt(\hspace*{0.01cm}\Om(2r_{\hat{k}}))}
		\leq c\cdot\norm{E_{\hat{k}}-\Psi}_{\R(\Om(2r_{\hat{k}}))}<\da/2\,.
	\end{align*}
	Extending $\Psi$ by zero to $\Om$ we obtain (\,by abuse of notation\,) 
	$\Psi\in\dsp\cictom$ with 
	\begin{align*}
		\normRtom{E-\Psi}
			&\leq\normRtom{E-E_{\hat{k}}}
			 +\norm{E_{\hat{k}}-\Psi}_{\Rt(\hspace*{0.01cm}\Om(2r_{\hat{k}}))}
		 	 <\frac{\da}{2}+\frac{\da}{2}=\delta\,,\\[-10pt]
	\end{align*}
	which completes the proof.
\end{proof}
\subsection{Some Functional Analysis} 
Let $\mathsf{H}_{1}$ and $\mathsf{H}_{2}$ be Hilbert spaces. With 
$\mathrm{L}(\hspace*{0.4mm}\mathsf{H}_{1},\mathsf{H}_{2})$ and 
$\mathrm{B}(\hspace*{0.5mm}\mathsf{H}_{1},\mathsf{H}_{2})$ 
we introduce the sets of linear resp.\;bounded linear operators mapping 
$\mathsf{H}_{1}$\hspace*{0.04cm}to $\mathsf{H}_{2}$. For 
$\map{\A}{\dod(\A)\subset\mathsf{H}_{1}}{\mathsf{H}_{2}}$ linear, closed, 
and densely defined, the adjoint 
$\map{\A^{\ast}}{\dod(\A^{\ast})\subset\mathsf{H}_{2}}{\mathsf{H}_{1}}$ is
characterized by 
\begin{align*}
	\scp{\A\,x}{y}_{\mathsf{H}_2}=\scp{x}{\A^{\ast}\,y}_{\mathsf{H}_1}
	\qquad\forall\;x\in\dod(\A),\;y\in\dod(\A^{\ast})\,.
\end{align*}
By the projection theorem we have the following Helmholtz type decompositions
\begin{align*}
	\mathsf{H}_{1}=\ovl{\rg(\dsp\A^{\ast})}\oplus\ker(\A),
	\qquad\text{and}\qquad 
	\mathsf{H}_{2}=\ovl{\rg(\A)}\oplus\ker(\A^{\ast})\,,	
\end{align*}
which we use to define the corresponding reduced operators 
$\calA:=\restr{\A}{\ker(\A)^{\perp}}$, 
$\calA^{\ast}:=\restr{\A^{\ast}}{\ker(\A^{\ast})^{\perp}}$, i.e.,
\begin{alignat*}{3}
	&\map{\calA}{\dod(\calA)\subset\ovl{\rg(\dsp\A^{\ast})}}{\ovl{\rg(\A)}}
	\;\;\qquad &&\qquad &
	&\map{\calA^{\ast}}{\dod(\calA^{\ast})\subset\ovl{\rg(\A)}}
	     {\ovl{\rg(\dsp\A^{\ast})}}\\[-8pt]
	&\;\qquad &&\text{resp.}\;\;\;\qquad &&\\[-8pt]
	&\hspace*{0.5cm}\dod(\calA)=\dod(\A)\cap\ovl{\rg(\dsp\A^{\ast})}
	\qquad &&\qquad & 
	&\hspace*{0.5cm}\dod(\calA^{\ast})=\dod(\A^{\ast})\cap\ovl{\rg(\A)}\,.
\end{alignat*}
These operators are also closed, densely defined and indeed adjoint to each 
other. Moreover, by definition $\calA$ and $\calA^{\ast}$ are injective and 
therefore the inverse operators
\begin{align*}
	\map{\calA^{-1}}{\rg(\A)}{\dod(\calA)}
	\;\qquad\text{and}\qquad
	\map{\big(\calA^{\ast})^{-1}}{\rg(\dsp\A^{\ast})}{\dod(\calA^{\ast})}
\end{align*}
exist. The pair $(\calA,\calA^{\ast})$ satisfies the following result 
of the so called \emph{Functional Analysis Toolbox}, see 
e.g.\;\cite[Section 2]{pauly_solution_2016}, from which we 
will derive some Poincar{\'e} type estimates for the time-harmonic Maxwell 
operator $(\,\calM-\om\,)$ (\,cf. Remark \ref{rem:bvp_toolbox-real} and Remark 
\ref{rem:bvp_toolbox-komplex}\,).
\begin{lem}\label{lem:prel_toolbox}
	The following assertions are equivalent:\\[-9pt]
	\begin{enumerate}[itemsep=5pt]
		\item[$(1)$] $\dsp\exists\;c_{\A}\in(0,\infty)\;\,$
					 $\forall\;x\in \dod(\calA):$
					 $\quad\norm{x}_{\mathsf{H}_{1}}
					  \leq c_{\A}\norm{\A\,x}_{\mathsf{H}_{2}}$.
		\item[$(1^{\ast})$] $\dsp\exists\;c_{\A^{\ast}}\in(0,\infty)\;\,$
							$\forall\;y\in \dod(\calA^{\ast}):$
						    $\quad\norm{y}_{\mathsf{H}_{2}}
						     \leq c_{\A^{\ast}}
						     \norm{\A^{\ast}\,y}_{\mathsf{H}_{1}}$.
		\item[$(2)$] $\rg(\A)=\rg(\calA)$ is closed in $\mathsf{H}_{2}$.
		\item[$(2^{\ast})$] $\rg(\A^{\ast})=\rg(\calA^{\ast})$ ist closed in 
		                    $\mathsf{H}_{1}$.
		\item[$(3)$] $\dsp\map{\calA^{-1}}{\rg(\A)}{\dod(\calA)}$ is continuous.
		\item[$(3^{\ast})$] $\dsp\map{\big(\calA^{\ast}\big)^{-1}}
							{\rg(\A^{\ast})}{\dod(\calA^{\ast})}$ is continuous.
	\end{enumerate}	
\end{lem}
\noindent
Note that for the ``best'' constants $c_{\A}$ and $c_{\A^{\ast}}$ it holds
\begin{align*}
	\norm{\calA^{-1}}_{\rg(\A),\rg(\A^{\ast})}
	=c_{\A}
	=c_{\A^{\ast}}
	=\norm{(\calA^{\ast})^{-1}}_{\rg(\A^{\ast}),\rg(\A)}\,.		
\end{align*}
%
%
\section{Solution Theory for Time-Harmonic Maxwell Equations}\label{sec:th-bvp} 
%
%
\noindent
As mentioned above we shall treat the time-harmonic Maxwell equations with mixed 
boundary conditions
\begin{alignat}{4}
	-\rot H+i\om\eps E&=F\;\,&&\;\text{in}\;\;\,\Om,
	\qquad\qquad & 
	E\times\nu &=0\;&&\;\text{on}\;\;\,\Gat,\notag\\[-6pt]~\\[-6pt]
	\rot E+i\om\mu H&=G\;\;&&\;\text{in}\;\;\,\Om,
	\qquad\qquad & 
	H\times\nu&=0\;\,&&\;\text{on}\;\;\,\Gan,\notag
\end{alignat}
in an exterior weak Lipschitz domain $\Om\subset\rthree$ and for frequencies $
\om\in\C\sm(0)$. Moreover, we suppose that the material parameters $\eps$ and $
\mu$ are $\ka$-admissible with $\ka\geq 0$. Using the abbreviations from above and 
rewriting 
\[\solu:=(E,H)\qquad,\qquad\rhs:=i\La^{-1}(-F,G)\,,\] 
the weak formulation of these boundary value problem reads:\\[-10pt]
\begin{align}\label{equ:bvp_weak-max-sys}
	\textit{For }
	\rhs\in\ltloc(\Omb)
	\textit{ find }
	\solu\in\rloct(\Omb)\times\rlocn(\Omb)
	\textit{ such that }
	(\M-\om\,)\,\solu=\rhs.
\end{align}
We shall solve this problem using polynomially weighted Hilbert spaces. In doing 
so we avoid additional assumptions on boundary regularity for $\Om$, since only 
a compactness result comparable to Rellich's selection theorem is needed. More 
precisely, we will show that $\Om$ satisfies 
\emph{"Weck's (local) selection theorem"}, 
also called\;\emph{"(local) Maxwell compactness property"}, which in 
fact is also an assumption on the quality of the boundary and in some sense 
supersedes assumptions on boundary regularity.
\begin{defi}\label{def:bvp_WST}
    Let $\ga$ be $\ka$-admissible with 
    $\ka\geq 0$ and let $\Om\subset\rthree$ be open. $\Om$ satisfies\;{"Weck's local selection theorem"} 
    (WLST) (\,or has the\;{"local Maxwell compactness property"}\,), 
    if the embedding
    \begin{align}       
    	\Rztom\cap\ga^{-1}\Dznom\xhookrightarrow{\hspace{0.4cm}}\ltlocomb
    \end{align}
    is compact. $\Om$ satisfies {"Weck's selection theorem"} (WST) (\,or has
    the {"Maxwell compactness property"}\,) if the embedding
    \begin{align}
	 	\Rzt(\Om)\cap\ga^{-1}\Dzn(\Om)
	 	\xhookrightarrow{\hspace{0.4cm}}\lt(\Om)
    \end{align}
    is compact.
\end{defi}
\begin{rem}\label{rem:bvp_WST}
	Note that Weck's (local) selection theorem is essentially independent of
	$\ga$\;meaning that a domain $\Om\subset\rthree$ satisfies WST resp.\;WLST, 
	if and only if the imbedding 
	\begin{align*}
		 \Rztom\cap\Dznom\xhookrightarrow{\hspace{0.4cm}}\ltom
		 \qquad\text{resp.}\qquad
		 \Rztom\cap\Dznom\xhookrightarrow{\hspace{0.4cm}}\ltlocomb
	\end{align*}
	is compact. The proof is practically identical with the one of\, 
	{\rm \cite[Lemma 2]{picard_elementary_1984}} 
	\big(\,see also \cite{weck_maxwells_1974},\,\cite{weber_local_1980}\;\big).
\end{rem}
\begin{lem}\label{lem:bvp_WST}
	Let $\ga$ be $\ka$-admissible with $\ka\geq 0$
	and let $\Om\subset\rthree$ be an exterior domain. Then the following 
    statements are equivalent:
	\vspace*{3pt}
	\begin{enumerate}[label=\rm$(\alph*)$,leftmargin=30pt,itemsep=6pt]
		\item $\Om$ satisfies WLST.
		\item For all $\rtil>r_{0}$ the imbedding
              \begin{align*}
                  \Rgen{}{}{\tilde{\Ga}_{1}}(\Omrtil)\cap\ga^{-1}&\Dzn(\Omrtil)
                  \xhookrightarrow{\hspace{0.4cm}}\lt(\Omrtil)
              \end{align*}
              with $\tilde{\Ga}_{1}:=\Ga_{1}\cup\Sp(\rtil)$ is compact, i.e., $\Omrtil$ satisfies WST.
		\item For all $\rtil>r_{0}$ the imbedding
              \begin{align*}
                  \Rgen{}{}{{\Ga}_{1}}(\Omrtil)\cap\ga^{-1}&
                  \Dgen{}{}{\tilde{\Ga}_{2}}(\Omrtil)
                  \xhookrightarrow{\hspace{0.4cm}}\lt(\Omrtil)
              \end{align*}
              with $\tilde{\Ga}_{2}:=\Ga_{2}\cup\Sp(\rtil)$ is compact, i.e., $\Omrtil$ satisfies WST.
        \item For all $s,t\in\reals$ with $t<s$ the imbedding
		      \begin{align*}
		          \Rstom\cap\ga^{-1}&\Dsnom
		          \xhookrightarrow{\hspace{0.4cm}}\lttom
		      \end{align*}
		      is compact.
    \end{enumerate}
\end{lem}
\begin{proof}~\\[-6pt]
	\begin{enumerate}[leftmargin=30pt,itemsep=6pt,itemindent=-15pt]
		\item[]\underline{$(a)\,\Rightarrow\,(b)$:}
			   Let $\rtil>r_{0}$. By Remark \ref{rem:bvp_WST} it is 
			   sufficient to show the compactness of 
			   \begin{align*}
			   		\Rgen{}{}{\tilde{\Ga}_{1}}(\Omrtil)\cap\Dzn(\Omrtil)
                  \xhookrightarrow{\hspace{0.4cm}}\lt(\Omrtil)\,.
			   \end{align*} 
		       Therefore let $(E_{n})_{n\in\N}\subset\Rgen{}{}{\tilde{\Ga}_{1}}
		       (\Omrtil)\cap\Dzn(\Omrtil)$ be bounded, choose $r_{0}<\rhat<\rtil$ 
		       and a cut-off function $\chi\in\cicrthree$ with 
		       $\supp\chi\subset\U(\rtil)$ and $\restr{\chi}{\B(\rhat)}=1$. 
		       Then, for every $n\in\N$ we have 
			   \begin{align*}
			   		\hspace*{1cm}E_{n}=\check{E}_{n}+\hat{E}_{n}
			   		:=\chi E_{n}	+(1-\chi) E_{n}\,,
			   		\qquad\supp\check{E}_{n}\subset\Om(\rtil)\,,
			   		\qquad\supp\hat{E}_{n}\subset\G(\rhat,\rtil)\,,
			   \end{align*} splitting $(E_{n})_{n\in\N}$ 
			   into $(\check{E}_{n})_{n\in\N}$ and $(\hat{E}_{n})_{n\in\N}$. 
			   Extending $\check{E}_{n}$ resp.\;$\hat{E}_{n}$ by zero, we obtain 
			   (\,by abuse of notation\,) sequences
			   \begin{align*}
			   		(\check{E}_{n})_{n\in\N}\subset\Rztom\cap\Dznom
			   		\qquad\text{and}\qquad
			   		(\hat{E}_{n})_{n\in\N}\subset
			   		\Rgen{}{}{\Sp(\rtil)}(\U(\rtil))\cap\D(\U(\rtil))
			   \end{align*}
			   which are bounded in the respective spaces. Thus, using 
			   Weck's local selection theorem and Remark 
			   \ref{rem:bvp_WST}, we can choose a subsequence 
			   $(\check{E}_{\pi(n)})_{n\in\N}$ of $(\check{E}_{n})_{n\in\N}$ 
			   converging in $\ltlocomb$. The corresponding subsequence 
			   $(\hat{E}_{\pi(n)})_{n\in\N}$ is of course also bounded in 
			   $\Rgen{}{}{\Sp(\rtil)}(\U(\rtil))\cap\D(\U(\rtil))$ and 
			   by \cite[Theorem 2.2]{weber_regularity_1981}, even in 
			   $\Ho(\U(\rtil))$, hence (\,Rellich's selection theorem\,) has a 
			   subsequence $(\hat{E}_{\tilde{\pi}(n)})_{n\in\N}$ converging
			   in $\lt(\U(\rtil))$. Thus
			   \begin{align*}
			   	  &\norm{E_{\tilde{\pi}(n)}-E_{\tilde{\pi}(m)}}_{\lt(\Omrtil)}\\
			   	  &\qquad\leq c\cdot\Big(\,\norm{\chi\big(E_{\tilde{\pi}(n)} 
			   	   -E_{\tilde{\pi}(m)}\big)}_{\lt(\Omrtil)}
			   	   +\norm{(1-\chi)\big(E_{\tilde{\pi}(n)}
			   	   -E_{\tilde{\pi}(m)}\big)}_{\lt(\Omrtil)}\Big)\\
			   	  &\qquad\qquad\leq c\cdot\Big(\,\big\|\,
			   	   \check{E}_{\tilde{\pi}(n)}
			   	   -\check{E}_{\tilde{\pi}(m)}\,\big\|_{\lt(\Omrtil)}
			   	   +\big\|\,\hat{E}_{\tilde{\pi}(n)}
			   	   -\hat{E}_{\tilde{\pi}(m)}\,\big\|_{\lt(\U(\rtil))}\,\Big)
			   	   \xrightarrow{\;m,n\rightarrow\infty\;}0\,,
			   \end{align*}
 			   meaning that $(E_{\tilde{\pi}(n)})_{n\in\N}\subset(E_{n})_{n\in\N}$ is 
 			   a Cauchy sequence in 
 			   $\lt(\Omrtil)$.\\[-6pt]
		\item[]\underline{$(b)\,\Rightarrow\,(d)$:}
			   Let $s,t\in\reals$ with $s>t$ and 
			   $(E_{n})_{n\in\N}\subset\Rstom\cap\ga^{-1}\Dsnom$ be bounded. 
			   Then there exists a subsequence 
			   $(E_{\pi(n)})_{n\in\N}\subset(E_{n})_{n\in\N}$ converging 
			   weakly in $\Rstom\cap\ga^{-1}\Dsnom$ to some vector field 
			   $E\in\Rstom\cap\ga^{-1}\Dsnom$. We now construct a subsequence
			   $(E_{\tilde{\pi}(n)})_{n\in\N}$ of $(E_{\pi(n)})_{n\in\N}$ 
			   converging in $\ltlocomb$ to the same limit $E$. Therefore, 
			   observe that
               \begin{align*}
               		(E_{\pi(n),1})_{n\in\N}
               		\quad\text{with}\quad 
               		E_{\pi(n),1}:=\eta_{1}E_{\pi(n)}
               \end{align*}
               is bounded in 
               $\Rgen{}{}{\tilde{\Ga}_{1}}(\Om(r_{2}))
                \cap\ga^{-1}\Dzn(\Om(r_{2}))$,  
               $\tilde{\Ga}_{1}:=\Ga_{1}\hspace*{0.3mm}\cup\,\Sp(r_{2})$ 
               such that by assumption there exists a subsequence 
               $(E_{\pi_{1}(n),1})_{n\in\N}$ converging in $\lt(\Om(r_2))$. 
               Then $(E_{\pi_{1}(n)})_{n\in\N}\subset(E_{\pi(n)})_{n\in\N}$ is 
               converging in $\lt(\Om(r_{1}))$ and as $(E_{\pi_{1}(n)})_{n\in\N}$ 
               is also weakly convergent in $\lt(\Om(r_{1}))$, we have 
               \begin{align*}
               		E_{\pi_{1}(n)}\longrightarrow E
               		\quad\text{in}\quad\lt(\Om(r_1))\,.
               \end{align*}
               Multiplying $(E_{\pi_{1}(n)})_{n\in\N}$ with $\eta_{2}$ we obtain 
               a sequence $(E_{\pi_1(n),2})_{n\in\N}$, 
               $E_{\pi_{1}(n),2}:=\eta_{2}E_{\pi_{1}(n)}$ bounded 
               in $\Rgen{}{}{\tilde{\Ga}_{1}}(\Om(r_{3}))
               \hspace*{0.5mm}\cap\hspace*{0.5mm}
               \ga^{-1}\Dzn(\Om(r_{3}))$, 
               $\tilde{\Ga}_{1}:=\Ga_{1}\cup\,\Sp(r_{3})$ and as before we 
               construct a subsequence $(E_{\pi_{2}(n),2})_{n\in\N}$ converging 
               in $\lt(\Om(r_{3}))$, giving a subsequence 
               $(E_{\pi_{2}(n)})_{n\in\N}\subset(E_{\pi_{1}(n)})_{n\in\N}$ with
               \begin{align*}
               		E_{\pi_{2}(n)}\longrightarrow E
               		\quad\text{in}\quad\lt(\Om(r_{2}))\,.
               \end{align*} 
               Continuing like this, we successively construct subsequences 
               $(E_{\pi_{k}(n)})_{n\in\N}$ with $E_{\pi_{k}(n)}\longrightarrow E$ 
               in $\lt(\Om(r_{k}))$ and switching to the diagonal sequence we 
               indeed end up with a sequence $(E_{\tilde{\pi}(n)})_{n\in\N}$, 
               $\tilde{\pi}(n):=\pi_{n}(n)$, with 
               $E_{\tilde{\pi}(n)}\longrightarrow E$ in $\ltlocomb$. 
 			   Now Lemma \ref{lem:app-tech_weighted-est} implies for arbitrary 
 			   $\theta>0$  
			   \begin{align*}
			   		\normlttom{E_{\tilde{\pi}(n)}-E}
			   			\leq c\cdot\normltomda{E_{\tilde{\pi}(n)}-E}+\theta\,,
			   \end{align*}
			   with $c,\da\in(0,\infty)$ independent of $E_{\tilde{\pi}(n)}$. 
			   Hence
			   \begin{align*}
			   		\limsup_{n\rightarrow\infty}
			   		\normlttom{E_{\tilde{\pi}(n)}-E}\leq\theta\,,
			   \end{align*} 
			   and we obtain $E_{\tilde{\pi}(n)}\To E$ in $\lttom$.\\[-6pt]
		\item[]\underline{$(d)\,\Rightarrow\,(a)$:}
			   For $(E_{n})_{n\in\N}$ bounded in $\Rztom\cap\ga^{-1}\Dznom$,
			   assertion (c) implies the existence of a subsequence 
			   $(E_{\pi(n)})_{n\in\N}$ converging in $\ltmoom$ to some 
			   $E\in\ltmoom$. Then $E\in\ltlocomb$ and as
			   \begin{align*}
			   		\forall\;\rtil>0:\quad	
			   		\norm{E_{\pi(n)}-E}_{\lt(\Om(\rtil))}
			   		\leq (1+\rtil)^{1/2}\cdot\normltmoom{E_{\pi(n)}-E}\,,	
			   \end{align*}
 			   we obtain $(E_{\pi(n)})_{n\in\N}\To E$ in $\ltlocomb$.\\[-6pt]
	\end{enumerate}
	Similar arguments to those corresponding to $(b)$ show the assertion for 
	$(c)$.
\end{proof}
\noindent
As shown by Bauer, Pauly, and Schomburg \cite[Theorem 4.7]{bauer_maxwell_2016}, 
bounded weak Lipschitz domains satisfy Weck's selection theorem and by Lemma 
\ref{lem:bvp_WST} $(a)$ this directly implies the following.
\begin{theo}\label{theo:bvp_WST}
	Exterior weak Lipschitz domains satisfy\;Weck's local selection theorem.
\end{theo}
\noindent
Returning to our initial question,
a first step to a solution theory for 
\eqref{equ:bvp_weak-max-sys} is the following observation.

\begin{theo}
	The \emph{Maxwell operator} 
	\begin{align*}
		\maps{\calM}{\Rztom\times\Rznom\subset\ltLaom}{\ltLaom}
		     {\solu}{\M\,\solu}\,,
	\end{align*}
	is self-adjoint and reduced by the closure of its range 
	\begin{align*}
		\ovl{\rg(\calM)}
		  =\eps^{-1}\ovl{\rot\Rznom}\times\mu^{-1}\ovl{\rot\Rztom}\,.
	\end{align*}
\end{theo}
\noindent
We note that here, in the case of an exterior domain $\Om$, the respective 
ranges are not closed.
\begin{proof}
	The proof is straightforward using Lemma \ref{lem:prel_relationships}, i.e., 
	the equivalence of the definition of weak and strong boundary conditions.
\end{proof}
\noindent
Thus $\sigma(\calM)\subset\reals$, meaning that every $\om\in\C\sm\reals$ is 
contained in the resolvent set of $\calM$ and hence for given $\rhs\in\ltLaom$ we 
obtain a unique solution of \eqref{equ:bvp_weak-max-sys} by 
$\solu:=\big(\,\calM-\om\,\big)^{-1}\hspace*{0.02cm}\rhs\in\Rztom\times\Rznom$.
Moreover, using the resolvent estimate 
$\norm{(\,\calM-\om\,)^{-1}}\leq|\,\Im\om\,|^{-1}$ and the differential 
equation, we get
\begin{align*}
	\normRom{u}
		\leq c\cdot\Big(\,\normltLaom{\solu}
		      +\normltLaom{\rhs}+|\,\om\,|\normltLaom{\solu}\,\Big)
		\leq c\cdot\frac{\,1+|\,\om\,|\,}{|\,\Im \om\,|}	
		     \cdot\normltLaom{\rhs}\,.
\end{align*}

\begin{theo}
	For $\om\in\C\sm\reals$ the solution operator
	\begin{align*}
		\map{\sol:=\big(\,\calM-\om\,\big)^{-1}}{\ltLaom}
	     	{\Rztom\times\Rznom}	
	\end{align*}
	is continuous with $\dsp\norm{\sol}_{\ltLaom,\Rom}
	\leq c\cdot\frac{\,1+|\,\om\,|\,}{|\,\Im\om\,|}$, 
	where $c$ is independent of $\om$ and $\rhs$.
\end{theo}
\begin{rem}\label{rem:bvp_toolbox-komplex}
	Let $\om\in\C\sm\reals$. By Lemma \ref{lem:prel_toolbox} the following 
	statements are equivalent to the boundedness of $\sol$:
	\begin{enumerate}[label=$\circ$, leftmargin=0.75cm]
		\item (Friedrichs/Poincar{\'e} type estimate) There exists $c>0$ such that
			  \begin{align*}
			  	\normRom{\solu}
			  	\leq c\cdot\normltLaom{\big(\,\M-\om\,\big)\hspace*{0.02cm}\solu}
			  	\quad\;\;\,
			  	\forall\;\solu\in\Rztom\times\Rznom\,.
			  \end{align*}
		\item (Closed range) The range
			  \begin{align*}
			  	\rg(\,\calM-\om\,)
			  	 =\big(\,\calM-\om\,\big)\,\big(\,\Rztom\times\Rznom\,\big)
			  \end{align*}
			  is closed in $\ltLaom$.\\
	\end{enumerate}
\end{rem}
\noindent
The case $\om\in\reals\sm(0)$ is much more challenging, since we want to 
solve in the continuous spectrum of the Maxwell operator. Clearly this cannot be 
done for every $\rhs\in\ltLaom$, since otherwise we would have 
$\rg(\,\calM-\om\,)=\ltLaom$ and therefore $(\,\calM-\om\,)^{-1}$ would 
be continuous (\,cf. Lemma \ref{lem:prel_toolbox}\,) or in other words 
$\om\not\in\sigma(\calM)$. Thus we have to restrict ourselves to certain 
subspaces of $\ltLa$ or generalize our solution concept. Actually, we 
will do both and show existence as well as uniqueness of weaker, so called 
\emph{``radiating solutions''}, by switching to data 
$\rhs\in\ltsom$ for some $s>1/2$.
\begin{defi}\label{def:bvp_rad-sol}
	Let $\om\in\reals\sm(0)$ and $f\in\ltlocom$. We call 
	$\solu$ (radiating) solution of 
	\eqref{equ:bvp_weak-max-sys}, if
	\begin{align*}
		\solu\in\Rsmtom{-\frac{1}{2}}\times\Rsmnom{-\frac{1}{2}}	
	\end{align*}
	and
	\begin{align}
		\big(\,\M&-\om\,\big)\hspace*{0.03cm}\solu=\rhs\,,
		\label{equ:bvp_rad-sol-cond1}\\[3pt]
		\big(\,\Laz+\sqrt{\eps_{0}\mu_{0}}\;\Xi\;\big)\hspace*{0.03cm}&\solu
		\in\ltbigom{-\frac{1}{2}}\,.
		\label{equ:bvp_rad-sol-cond2}
	\end{align}
\end{defi}
\begin{rem}
	Since
	\begin{align*}
		\big(\,\Laz+\sqrt{\eps_{0}\mu_{0}}\;\Xi\;\big)\hspace*{0.03cm}\solu
		 =\Laz\,\bigg(\,E-\sqrt{\frac{\mu_0}{\eps_{0}}}\,\xi\times H\,,\,
		  H+\sqrt{\frac{\eps_{0}}{\mu_{0}}}\,\xi\times E\,\bigg)\,,	
	\end{align*}
	the last condition is just the classical Silver-M\"uller radiation condition 
	which describes the behavior of the electro-magnetic field at infinity and is 
	needed to distinguish outgoing from incoming waves 
	(interchanging signs would yield incoming waves).
\end{rem}
\noindent
In order to construct such a radiating solution $\solu$ we use 
the \emph{``limiting absorption principle''} introduced by Eidus and 
approximate $\solu$ by solutions $(\solu_{n})_{n\in\N}$ associated with 
frequencies $(\om_{n})_{n\in\N}\subset\C\sm\reals$ converging to 
$\om\in\reals\sm(0)$. This leads to statement $(\mathrm{4})$ of our main 
result Theorem \ref{theo:bvp_fred-alt}, where the following abbreviations are used: 
\begin{alignat*}{2}
	\gk{\,\calM-\om\,}
		&:=\setb{\solu}{\solu\text{ is a radiating solution of }
		 (\,\M-\om\,)\hspace*{0.03cm}\solu=0\,}
		&\quad
		&\textit{(gen. kernel of $\calM-\om$)}\,,\\
	\gs &:=\setb{\,\om\in\C\sm(0)\hspace*{0.03cm}}{\,\gk{\calM-\om}\neq(0)\,} 
		&\quad\qquad
		&\textit{(gen. point spectrum of $\calM$)}\,.\\[-4pt]
\end{alignat*}
\begin{theo}[\,Fredholm alternative\,]\label{theo:bvp_fred-alt}
	Let $\Om\subset\rthree$ be an exterior weak Lipschitz domain with boundary 
	$\Ga$ and weak Lipschitz boundary parts $\Ga_{1}$ and 
	$\Ga_{2}=\Ga\sm\ovl{\Ga}_{1}$. Furthermore let $\om\in\reals\sm(0)$ and 
	$\eps$,\,$\mu$ be $\ka$-admissible with $\ka>1$. Then:
	\begin{enumerate}[label=$(\arabic*)$]
		\item $\dsp\gk{\,\calM-\om\,}\subset
			   \bigcap_{t\in\reals}\Big(\,\Rttom\cap\eps^{-1}\zdtnom\,\Big)
			   \,\times\,
			   \Big(\,\Rtnom\cap\mu^{-1}\zdttom\,\Big)$\footnote{We even have
			   \begin{align*}
			   		\dsp\gk{\,\calM-\om\,}
			   		\subset\bigcap_{t\in\reals}
			   		\Big(\,\Rttom\cap\eps^{-1}\rot\Rtnom\,\Big)
			   		\,\times\,
			   		\Big(\,\Rtnom\cap\mu^{-1}\rot\Rttom\,\Big)\,.
			   \end{align*}}\,.\\[2pt]
		\item $\dim\gk{\,\calM-\om\,}<\infty$\,.\\[3pt]
		\item $\gs\subset\reals\sm(0)$ and $\gs$ has no accumulation point in 
			  $\reals\sm(0)$\,.\\[3pt]
		\item For all $\rhs\in\dsp\ltbigom{\frac{1}{2}}$ 
			  there exists a radiating solution $\solu$ of 
			  \eqref{equ:bvp_weak-max-sys}, if and only if
			  \begin{align}\label{equ:bvp_fred-alt-cond1}
			  		\forall\;v\in\gk{\,\calM-\om\,}:
			  		\qquad\scpltLaom{\rhs}{v}=0\,.	
			  \end{align}
			  Moreover, we can choose $\solu$ such that
			  \begin{align}\label{equ:bvp_fred-alt-cond2}
			  		\forall\;v\in\gk{\,\calM-\om\,}:
			  		\qquad\scpltLaom{\solu}{v}=0\,.
			  \end{align}
			  Then $\solu$ is uniquely determined.\\[2pt]
		\item For all $s$,${-t}>1/2$ the solution operator
			  \begin{align*}
			  	\hspace*{0.75cm}
			  	\map{\sol}{\ltsom\cap\,\gk{\,\calM-\om\,}^{\perp_{\La}}}
			  	{\big(\,\Rttom\times\Rtnom\,\big)\cap\,
			  	 \gk{\,\calM-\om\,}^{\perp_{\La}}}
			  \end{align*}
			  defined by $(\mathrm{4})$ is continuous.\\[-2pt]
	\end{enumerate}
\end{theo}
\begin{rem}\label{rem:bvp_toolbox-real}
	Under the conditions of Theorem \ref{theo:bvp_fred-alt} the following 
	statements are equivalent to the boundedness of $\sol$ 
	\big(\,cf. Lemma \ref{lem:prel_toolbox} and Remark \ref{rem:bvp_toolbox-komplex}\,\big):\\[-10pt]
	\begin{enumerate}[label=$\circ$, leftmargin=0.75cm]
		\item (Friedrichs/Poincar{\'e} type estimate) For all $s$,$-t>1/2$ there 
			  exists $c>0$ such that
			  \begin{align*}
			  	\normRtom{\solu}
			  	\leq c\cdot\normltsom{\big(\,\M-\om\,\big)\hspace*{0.02cm}\solu}
			  \end{align*}
			  holds for all $\solu\in\big(\,\Rttom\times\Rtnom\,\big)\cap\,\gk{\,\calM-\om\,}^{\perp_{\La}}$
			  satisfying the radiation condition.
		\item (Closed range) For all $s$,$-t>1/2$ the range
			  \begin{align*}
			  	\rg(\,\calM-\om\,)=
			  	\big(\,\calM-\om\,\big)\,\big(\,\Rttom\times\Rtnom\,\big)
			  \end{align*}
			  is closed in $\ltsom$.\\
	\end{enumerate}
\end{rem}
\noindent
By the same indirect arguments as in \cite[Corollary 3.9]{pauly_low_2006} 
\big(\,see also \cite[Section 4.9]{pauly_niederfrequenzasymptotik_2003}\,\big), 
we get even stronger estimates for the solution operator $\sol$.\\[-6pt]
\begin{cor}
	Let $\Om\subset\rthree$ be an exterior weak Lipschitz domain with boundary 
	$\Ga$ and weak Lipschitz boundary parts $\Ga_{1}$ and 
	$\Ga_{2}:=\Ga\sm\ovl{\Ga}_{1}$. Furthermore let $s$,\,$-t>1/2$, 
	$\eps$,\,$\mu$ be $\ka$-admissible with $\ka>1$ and 
	$\dsp K\Subset\C_{+}\sm(0)$ with $\ovl{K}\cap\gs=\emptyset$. Then:\\[-9pt]
	\begin{enumerate}[label=$(\arabic*)$, leftmargin=0.85cm]
		\item There exist constants $c>0$ and $\that>-1/2$ such that for 
			  all $\om\in\ovl{K}$ and $\rhs\in\ltsom$\\[-8pt]
			  \begin{align*}
			  	\normRtom{\sol\rhs}
			   	 +\normltthatom{\big(\,\Laz
			   	 +\sqrt{\eps_{0}\mu_{0}}\;\Xi\;\big)\hspace*{0.03cm}\sol\rhs}
			   	 \leq c\cdot\normltsom{\rhs}	
			  \end{align*}
			  holds, implying that $\map{\sol}{\ltsom}{\Rttom\times\Rtnom}$
			  is equicontinuous w.r.t.\;$\om\in\ovl{K}$.\\[-5pt]
		\item The mapping
			  \begin{align*}
			  	 \Map{\calL}{\ovl{K}}
			  	     {\mathrm{B}\big(\,\ltsom\,,\Rttom\times\Rtnom\,\big)}
			  	     {\om}{\sol}
			  \end{align*}
			  is uniformly continuous. 
	\end{enumerate}
\end{cor}
%
%
\section{Polynomial Decay and A-Priori Estimate}
%
%
As stated before, we will construct a solution $\solu$ in the case of 
$\om\in\reals\sm(0)$ by solving \eqref{equ:bvp_weak-max-sys} for 
$\om_{n}=\om+i\sigma_{n}\in\C_{+}\hspace*{-0.07cm}\sm\reals$ and sending 
$\sigma_{n}\To 0$ (\,using  $(\om_{n})_{n\in\N}\in\C_{-}\sm\reals$ instead will 
lead to {``incoming''} solutions\,). The essential ingredients 
to generate convergence are the polynomial decay of eigensolutions, an 
a-priori-estimate for solutions corresponding to nonreal frequencies and Weck's 
local selection theorem. While the latter one is already satisfied 
(\,cf. Theorem \ref{theo:bvp_WST}\,), we obtain the 
first two in the spirit of \cite{weck_generalized1_1997} using the following 
decomposition Lemma introduced in \cite{pauly_niederfrequenzasymptotik_2003} 
\big(\,see also \cite{pauly_low_2006},\,\cite{pauly_polynomial_2012}\,\big).
\begin{lem}\label{lem:pol-est_decomp}
	Let $\om\in K\Subset\C\sm(0)$, $\eps$,\,$\mu$ be $\ka$-admissible with 
	$\ka\geq0$ and $s,t\in\reals$ such that $0\leq s\in\reals\sm\mathbb{I}$ 
	and $t\leq s\leq t+\ka$. Moreover, assume that $\solu\in\Rtom$ satisfies 
	$(\,\M-\om\,)\,\solu=\rhs\in\ltsom$. Then
	\begin{align*}
		\rhs_{1}:=\big(\,\mathrm{C}_{\Rot,\ceta}-i\om\ceta\hat{\La}\,\big)
		\hspace*{0.02cm}\solu-i\ceta\La\rhs\in\lts
	\end{align*}
	and by decomposing 
	\begin{align*}
		\rhs_{1}=\rhs_{\sfR}+\rhs_{\sfD}+\rhs_{\calS}
		 \in\zrs\dpl\zds\dpl\calS_{s}\\[-10pt]	
	\end{align*}
	according to {\rm\cite[Theorem 4]{weck_generalized_1994}} it holds
	\begin{align*}
		\rhs_{2}:=\rhs_{\sfD}+\frac{i}{\om}\,\tLaz\hspace*{-0.11cm}
		 \parbox[t]{0.35cm}{\vspace*{-0.4cm}${}^{-1}$}
		 \Rot\hspace*{0.01cm}\rhs_{\calS}\in\zds\,.
	\end{align*}
	Additionally, $\solu$ may be decomposed into
	\begin{align*}
		\solu=\eta\solu+\solu_{1}+\solu_{2}+\solu_{3}\,,	
	\end{align*}
	where
	\begin{enumerate}[label=$(\arabic*)$]
		\item $\eta\solu\in\rvoxom$ and for all $\that\in\reals$
			  \begin{align*}
			  		\normrthatom{\eta\solu}	
			  		\leq c\cdot\Big(\normltsom{\rhs}
			  		 +\normltsmkaom{\solu}\Big)\,,
			  \end{align*}
		\item $\dsp\solu_{1}
		  	   :=-\frac{i}{\om}\Laz^{-1}\big(\,\rhs_{\sfR}+\rhs_{\calS}\,\big)
		  	   \in\rs$ and
		  	   \begin{align*}
		  	   		\normrsws{\solu_{1}}\leq c\cdot\normltsws{\rhs_{1}}\,,	
		  	   \end{align*}
		\item $\dsp\solu_{2}:={\dsp\calF}^{\hspace*{0.03cm}-1}
			   \big(\,\rho^{-2}\big(\,1-ir\,\Xi\,\big)
			   \,\calF(\rhs_{2})\,\big)
			   \in\Hos\cap\zds$ and
			   \begin{align*}
					\normHosws{\solu_{2}}\leq c\cdot\normltsws{\rhs_{2}}\,,
			   \end{align*}
		\item $\solu_{3}:=\tsolu-\solu_{2}
			   \in\Htwot\cap\zdt$ and
			   for all $\that\leq t$
			   \begin{align*}
			   		\norm{\solu_{3}}	_{\Htwothat}
			   		\leq c\cdot\Big(\norm{\solu_{3}}_{\ltthat}
			   		+\norm{\solu_{2}}_{\Hothat}\Big)\,,
			   \end{align*}
			   where 
			   $\tsolu:=i\om^{-1}\Laz^{-1}
			    \big(\hspace*{0.03cm}\Rot\ceta\solu-\rhs_{\sfD}\,\big)
			    \in\Hot\cap\zdt$\\[-8pt]
	\end{enumerate}
	with constants $c\in(0,\infty)$ independent of $\solu$,\,$\rhs$\;or\;\,$\om$. 
	These fields solve the following equations:
	\begin{align*}
		\big(\hspace*{0.03cm}\Rot+\,i\om&\Laz\,\big)\,\ceta \solu=\rhs_{1}\,,
		\qquad 
		\big(\hspace*{0.03cm}\Rot+\,i\om\Laz\,\big)\hspace*{0.03cm}
		\tsolu=\rhs_{2}\,, 
		\qquad 
		\big(\hspace*{0.03cm}\Rot+\,i\om\Laz\,\big)\hspace*{0.03cm}\solu_{3}
		 =\big(\,1-\om\Laz\,\big)\,\solu_{2}\,,\\
		&\hspace*{0.6cm}\big(\,\Delta+\om^{2}\eps_0\mu_0\,\big)
		 \hspace*{0.03cm}\solu_{3}
		 =\big(\,1-i\om\tLaz\,\big)\,\rhs_{2}
		 -\big(\,1+\om^{2}\eps_0\mu_0\,\big)\hspace*{0.03cm}\solu_{2}\,.
	\end{align*}
	Moreover the following estimates hold for all\,\;$\that\leq t$ and uniformly 
	w.r.t.\;$\la\in K$,\;$\solu$\;and\hspace*{0.13cm}$\rhs$:\\[-10pt]
 	\begin{enumerate}[label=$\circ$, leftmargin=0.95cm,itemsep=4pt]
 		\item $\dsp\normltsws{\rhs_{2}}
 			   \leq c\cdot\normltsws{\rhs_{1}}
 			   \leq c\cdot\Big(\normltsom{\rhs}+\normltsmkaom{\solu}\Big)$	
 		\item $\dsp\normRthatom{\solu}\leq c\cdot\Big(\normltsom{\rhs}
 			   +\normltsmkaom{\solu}+\normltthatws{\solu_{3}}\Big)$
 		\item $\dsp\normltsws{\big(\,\Delta+\om^{2}\eps_{0}\mu_{0}\,\big)
 			   \hspace*{0.03cm}\solu_{3}}
 			   \leq c\cdot\Big(\normltsom{\rhs}+\normltsmkaom{u}\Big)$
 		\item $\dsp\normltthatws{\big(\hspace*{0.02cm}\Rot-\,i\la\hspace*{0.02cm}
 			   \sqrt{\eps_{0}\mu_{0}}\;\Xi\;\big)\hspace*{0.03cm}\solu}
 			   \leq c\cdot\Big(\normltsom{\rhs}+\normltsmkaom{\solu}
 			   +\normltthatws{\big(\hspace*{0.02cm}\Rot-\,i\la\hspace*{0.02cm}
 			   \sqrt{\eps_{0}\mu_{0}}\;\Xi\;\big)
 			   \hspace*{0.03cm}\solu_{3}}\Big)$
 	\end{enumerate}
    Here $\calS_{s}$ is a finite dimensional subspace of\;\,$\cicrthree$, $\calF$ 
    the Fourier transformation and 
    \begin{align*}
    	\mathrm{C}_{A,B}:=AB-BA
    \end{align*}
	the commutator of $A$ and $B$.
\end{lem}
\noindent
Basically, this lemma allows us to split $\solu$ into two parts. One part 
(\,consisting of $\eta\solu$,\;$\solu_{1}$ and $\solu_{2}$\,) has better 
integrability properties and the other part (\,consisting of $\solu_{3}$\,) is 
more regular and satisfies a Helmholtz equation in the whole of $\rthree$. Thus 
we can use well known results from the theory for Helmholtz equation (\,cf.
\;Appendix, Section \ref{sec:pol-est_hh-theory}\,) to establish corresponding 
results for Maxwell's equations. We start with the polynomial decay of 
solutions, especially of eigensolutions, which will lead to assertions 
$(\mathrm{1})$ - $(\mathrm{3})$ of our main theorem. Moreover, this will also 
show, that the solution $\solu$ we are going to construct, can be chosen to be 
perpendicular to the generalized kernel of the time-harmonic Maxwell operator. 
As in the proof of \cite[Theorem 4.2]{pauly_polynomial_2012} we obtain 
(\,see also Appendix, Section \ref{sec:proofs-max}\,) the following.\\[-10pt]
\begin{lem}[\,Polynomial decay of solutions\,]
\label{lem:pol-est_polynomial-decay}
	Let $J\subset\reals\sm(0)$ be some interval, $\om\in J$, $\eps$,\,$\mu$ be
	$\ka$-admissible with $\ka>1$ and $s\in\reals\sm\mathbb{I}$ with $s>1/2$. If
	\begin{align*}
		u\in\Rbigom{-\frac{1}{2}}
		\;\,\text{ satisfies }\;\,
		\big(\,\M-\om\,\big)\hspace*{0.03cm}\solu=:\rhs\in\ltsom\,,	
	\end{align*}
	then
	\begin{align*}
		u\in\Rsmoom
		\qquad\text{and}\qquad
		\normRsmoom{\solu}\leq c\cdot\Big(\normltsom{f}+\normltomda{u}\Big)\,,	
	\end{align*}
	with $c,\da\in(0,\infty)$ independent of\,\;$\om$,\;$\solu$ and 
	$\rhs$.\\[-9pt]
\end{lem}
\noindent
In short: If a solution $\solu$ satisfies $\solu\in\Rtom$ for some 
$t\hspace*{-0.01cm}>\hspace*{-0.02cm}-1/2$ and the right hand side
$\rhs=\big(\,\M-\om\,\big)\hspace*{0.03cm}\solu$ has better integrability 
properties, meaning $f\in\ltsom$ for some $s>1/2$, then also $\solu$ is
better integrable, i.e., $\solu\in\Rsmoom$. Especially, if 
\begin{align*}
	\solu\in\Rbigom{-\frac{1}{2}}
	\qquad\;\text{and}\qquad\; 
	f\in\ltsom\;\;\;\,\forall\;s\in\reals\,,
\end{align*}
then $\solu\in\Rsom$ for all $s\in\reals$, which is 
called \emph{``polynomial decay''}.\\[-10pt] 
\begin{cor}\label{cor:pol-est_polynomial-decay-es}
	Let $\om\in\reals\sm(0)$ and assume $\eps$,\,$\mu$ to be $\ka$-admissible 
	with $\ka>1$ and
	\begin{align*}
		\solu\in\Rsmtom{-\frac{1}{2}}\times\Rsmnom{-\frac{1}{2}}	
	\end{align*}
	to be a radiating solution (\,cf. Definition \ref{def:bvp_rad-sol}\,) of 
	$\big(\,\calM-\om\,\big)\hspace*{0.03cm}\solu=0$. Then:
	\begin{align*}
		\solu\in\bigcap_{t\in\reals}\big(\,\Rttom\times\Rtnom\,\big)\,.\\[-9pt]
	\end{align*}
\end{cor}
\begin{proof}
	According to Lemma \ref{lem:pol-est_polynomial-decay}, it suffices to show 
	$\solu\in\Rtom$ for some $t>-1/2$. Therefore, remember that $\solu$ is a 
	radiating solution, the radiation condition 
	\eqref{equ:bvp_rad-sol-cond2} holds and there exists $\that>-1/2$ such that 
	\begin{align}\label{equ:pol-est_pol-dec-es-rc}
		\big(\,\Laz+\sqrt{\eps_{0}\mu_{0}}\;\Xi\;\big)\hspace*{0.03cm}&\solu
		\in\ltthatom\,.
	\end{align}
	On the other hand we have\\[-8pt]
	\begin{align*}
		&\normltthatGaprztil{\big(\hspace*{0.02cm}\Laz+\sqrt{\eps_{0}\mu_{0}}
		 \;\Xi\,\big)\hspace*{0.02cm}\solu}^{2}\\
		&\qquad=\normltthatGaprztil{\Laz\hspace*{0.02cm}\solu}^{2}
		 +2\,\sqrt{\eps_{0}\mu_{0}}\;\Re\scpltthatGaprztil{\Xi\,\solu}
		 {\Laz\hspace*{0.02cm}\solu}
		 +\eps_{0}\mu_{0}\normltthatGaprztil{\Xi\,\solu}^{2}
	\end{align*}
	and using Lemma \ref{lem:app-tech_part-Int} 
	(\,cf. Appendix, Section \ref{sec:tech-tools}\hspace*{0.02cm}) with
	\begin{align*}
		\phi(s):=(1+s^{2})^{\that},
		\qquad
		\Phi:=\phi\circ r,
		\qquad
		\psi(\sigma)=\int_{\max\{r_{0},\sigma\}}^{\rtil}\phi(\tau)\,d\tau,
		\qquad
		\Psi=\psi\circ r\,,	
	\end{align*}
	as well as the differential equation, we conclude
	\begin{align*}
		\Re\scpltthatGaprztil{\Xi\,\solu}{\Laz\hspace*{0.02cm}\solu}
		&=\Re\scpltGaprztil{\Phi\,\Xi\,\solu}{\Laz\hspace*{0.02cm}\solu}\\
		&=\Re\Big(\,\scpltomrtil{\Psi\Rot\hspace*{0.02cm}\solu}
		{\Laz\hspace*{0.02cm}\solu}+\scpltomrtil{\Psi\hspace*{0.02cm}\solu}
		{\tLaz\hspace*{0.02cm}\Rot\hspace*{0.02cm}\solu}\,\Big)\\
		&=\Re\Big(\,\scpltomrtil{-i\om\Psi\La\hspace*{0.02cm}\solu}
		{\Laz\hspace*{0.02cm}\solu}+\scpltomrtil{\Psi\hspace*{0.02cm}\solu}
		{-i\om\tLaz\hspace*{0.02cm}\La\hspace*{0.02cm}\solu}\,\Big)\\
		&=\Re\,\underbrace{i\om\hspace*{0.02cm}
		\scpltomrtil{\Psi\hspace*{0.02cm}\La\hspace*{0.02cm}\solu}
		{\big(\,\tLaz-\Laz\,\big)\hspace*{0.02cm}\solu}}_{\in i\reals}=0\,,
	\end{align*}
	hence 
	\begin{align*}
		\normltthatGaprztil{\solu}
		\leq c\cdot\normltthatGaprztil{\big(\hspace*{0.02cm}
		\Laz+\sqrt{\eps_{0}\mu_{0}}\;\Xi\,\big)\hspace*{0.02cm}\solu}
	\end{align*}
	with $c\in(0,\infty)$ independent of $\rtil$. Now the monotone convergence 
	theorem and \eqref{equ:pol-est_pol-dec-es-rc} show\\[-6pt]
	\begin{align*}
	    \norm{\solu}_{\ltthat(\cU(r_{0}))}
	     \leq c\cdot\norm{\big(\,\Laz+\sqrt{\eps_0\mu_0}\;\Xi\,\big)
	     \hspace*{0.02cm}\solu}_{\ltthat(\cU(r_{0}))}<\infty\,,
	\end{align*}
	which already implies $\solu\in\ltthatom$ and completes the proof.
\end{proof}
\noindent
The next step is an a-priori estimate for solutions corresponding to nonreal 
frequencies, which will later guarantee that our solution satisfies the radiation 
condition \eqref{equ:bvp_rad-sol-cond2} and has the proper integrability. 
The proof of it is practically identical with the proof of 
\cite[Lemma 6.3]{pauly_polynomial_2012} 
(\,cf. Appendix, Section \ref{sec:proofs-max}\,).\\[-10pt]
\begin{lem}[A-priori estimate for Maxwell's equations]
\label{lem:pol-est_a-priori}
	Let $J\Subset\reals\sm(0)$ be some interval, $-t$,\,$s>1/2$ and 
	$\eps$,\,$\mu$ be $\ka$-admissible with $\ka>1$. Then there exist constants 
	$c,\da\in(0,\infty)$ and some $\that>-1/2$, such that for all 
	$\om\in\C_{+}$ with $\om^{2}=\la^{2}+i\la\sigma$, 
	$\la\in J$, $\sigma\in\big(\,0,\hspace*{0.01cm}
	{\sqrt{\eps_{0}\mu_{0}}\hspace*{0.02cm}}^{-1}\,\big]$ and 
	$\rhs\in\ltsom$
	\begin{align*}
		\normRtom{\sol\rhs}+\normltthatom{\big(\hspace*{0.02cm}\Laz
		+\sqrt{\eps_{0}\mu_{0}}\;\Xi\,\big)
		\hspace*{0.02cm}\sol\rhs}
		\leq c\cdot\Big(\normltsom{\rhs}+\normltomda{\sol\rhs}\Big)\,.\\[-9pt]
	\end{align*}
\end{lem}
%
%
\section{Proof of the Main Result}
\label{sec:main-result}
%
%
\noindent 
Before we start with the proof of Theorem \ref{theo:bvp_fred-alt} we provide some 
Helmholtz type decompositions, which will be useful in the following. These are 
immediate consequences of the projection theorem and Lemma 
\ref{lem:prel_relationships}.\\[-10pt]
\begin{lem}\label{lem:main-result_helm-decomp}
	It holds\\[-8pt]
	\begin{alignat*}{2}
		\lgen{}{2}{\eps}(\Om)
			&=\ovl{\dsp\nabla\Hoztom}\oplus_{\eps}\eps^{-1}\zdznom\,,
			\quad\quad&
		\lgen{}{2}{\mu}(\Om)
			&=\ovl{\dsp\nabla\Hoznom}\oplus_{\mu}\mu^{-1}\zdztom\,,\\
		\Rztom &=\ovl{\dsp\nabla\Hoztom}\oplus_{\eps}
		         \Big(\Rztom\cap\eps^{-1}\zdznom\Big)\,,
		\quad\quad&
		\Rznom &=\ovl{\dsp\nabla\Hoznom}\oplus_{\mu}
		         \Big(\Rznom\cap\mu^{-1}\zdztom\Big)\,,
	\end{alignat*}
	where the closures are taken in $\ltom$.\\[-9pt]
\end{lem}
\begin{proof}
	Let $\ga\in\{\eps,\mu\}$ and $i,j\in\{1,2\}$ with $i\neq j$. The linear 
	operator 
	\begin{align*}
		\map{\nabla_{i}}{\Hoztnom\subset\ltom}{\lgen{}{2}{\ga}(\Om)}	
	\end{align*}
	is densely defined and closed with adjoint 
	(\,cf. Lemma \ref{lem:prel_relationships}\,)
	\begin{align*}
		\map{-\div_{j}\ga}{\ga^{-1}\Dzntom\subset\lgen{}{2}{\ga}(\Om)}{\ltom}\,.
	\end{align*}
	The projection theorem yields 
	\begin{align*}
		\lgen{}{2}{\ga}(\Om)
		  =\ovl{\rg(\nabla_{i})}\oplus_{\ga}\ker(\div_{j}\ga)\,.
	\end{align*}
	The remaining assertion follows by $\nabla\Hoztnom\subset\rztnom$.
\end{proof}
\begin{proof}[\rm\textbf{Proof of Theorem \ref{theo:bvp_fred-alt}}]
Let $\om\in\reals\sm(0)$ and $\eps,\mu$ be $\ka$-admissible for 
some $\ka>1$.\\[8pt]
$(\mathrm{1})$: The assertion follows by Corollary 
\ref{cor:pol-est_polynomial-decay-es}
and the differential equation 
\begin{align*}
	\big(\,\M-\om\,\big)\,\solu=0
	\;\Longleftrightarrow\;
	\solu=i\om^{-1}\Lambda^{-1}\Rot\solu	\,,
\end{align*}
using the fact that (\,cf. Lemma \ref{lem:prel_relationships}\,)
\begin{align*}
	\rot\Rttom\subset\zdttom
	\qquad\text{resp.}\qquad
	\rot\Rtnom\subset\zdtnom\,.
\end{align*}
$(\mathrm{2})$: Let us assume that $\dim\gk{\,\calM-\om\,}=\infty$. Using (1) there exists a 
$\ltLa$-orthonormal sequence $(\solu_{n})_{n\in\N}\subset\gk{\,\calM-\om\,}$ 
converging weakly in $\ltom$ to 0. By the differential equation this sequence is 
bounded in $\big(\,\Rztom\cap\eps^{-1}\zdznom\,\big)
\times\big(\,\Rznom\cap\mu^{-1}\zdztom\,\big)$. Hence, due to Weck's local 
selection theorem, we can choose a subsequence, $(\solu_{\pi(n)})_{n\in\N}$ 
converging to 0 in $\ltloc(\Omb)$ (\,$(\solu_{\pi(n)})_{n\in\N}$ also converges 
weakly on every bounded subset\,). Now let $1<s\in\reals\sm\mathbb{I}$. Then 
Lemma \ref{lem:pol-est_polynomial-decay} guarantees the existence of 
$c,\da\in(0,\infty)$ independent of $(\solu_{\pi(n)})_{n\in\N}$ such that
\begin{align*}
	1=\normltLaom{\solu_{\pi(n)}}
	 \leq c\cdot\normRsmoom{\solu_{\pi(n)}}
	 \leq c\cdot\norm{\solu_{\pi(n)}}_{\ltomda}
	 \xrightarrow{\,\;n\rightarrow\infty\;\,}0	
\end{align*}
holds; a contradiction.\\[8pt]
$(\mathrm{3})$: $\calM$ is a selfadjoint operator, hence we clearly have 
$\gs\subset\reals\sm(0)$. Now assume $\widetilde{\om}\in\reals\sm(0)$ is an 
accumulation point of $\gs$. Then we can choose a sequence 
$(\om_{n})_{n\in\N}\subset\reals\sm(0)$ with $\om_{n}\neq\om_{m}$ for 
$n\neq m$, $\om_{n}\To\widetilde{\om}$ and a corresponding sequence 
$(\solu_{n})_{n\in\N}$ with $\solu_{n}\in\gk{\,\calM-\om_{n}\,}\sm(0)$. As 
$\calM$ is selfadjoint, eigenvectors associated to different eigenvalues are 
orthogonal provided they are well enough integrable (\,which is given by 
$(\mathrm{1})$\,) and thus by normalizing $(\solu_{n})_{n\in\N}$ we end up with 
an $\ltLa$-orthonormal sequence. Continuing as in $(\mathrm{2})$, we again obtain 
a contradiction.\\[8pt] 
$(\mathrm{4})$: First of all, if a solution $\solu$ satisfies 
\eqref{equ:bvp_fred-alt-cond2}, it is uniquely determined as for the 
homogeneous problem $\solu\in\gk{\,\calM-\om\,}$ together with (1) and 
\eqref{equ:bvp_fred-alt-cond2} implies $\solu=0$. Moreover, using Lemma 
\ref{lem:prel_relationships} and $(\mathrm{1})$, we obtain
\begin{align*}
	\scpltLaom{\rhs}{v}
	 =\scpltLaom{\big(\,\M-\om\,)\,\solu}{v}
	 =\scpltLaom{u}{\big(\,\M-\om\,)\,v}=0
	 \qquad\;\forall\;v\in\gk{\,\calM-\om\,}\,,		
\end{align*}
meaning \eqref{equ:bvp_fred-alt-cond1} is necessary. In order to show, that 
\eqref{equ:bvp_fred-alt-cond1} is also sufficient, we use Eidus' principle of 
limiting absorption. Therefore let $s>1/2$ and $\rhs\in\ltsom$ satisfy 
\eqref{equ:bvp_fred-alt-cond1}. We take a sequence 
$(\sigma_{n})_{n\in\N}\subset\reals_{+}$ with $\sigma_{n}\To 0$ and construct 
a sequence of frequencies 
\begin{align*}
	(\om_{n})_{n\in\N},\qquad\om_{n}
	 :=\sqrt{\om^{2}+i\sigma_{n}\om}\in\C_{+}\sm\reals\,,	
\end{align*}
converging to $\om$. Since $\calM$ is a selfadjoint operator we obtain 
(\,cf. Section \ref{sec:th-bvp}\,) a corresponding sequence of solutions 
$(\solu_{n})_{n\in\N}$, 
$\solu_{n}:=\mathcal{L}_{\om_{n}}\rhs\in\Rztom\times\Rznom$ 
satisfying $\big(\,\M-\om_{n}\,\big)\,\solu_{n}=\rhs$. Now our aim is to show 
that this sequence or at least a subsequence is converging to a solution 
$\solu$. By Lemma \ref{lem:main-result_helm-decomp} we decompose 
\begin{align*}
	\solu_{n}=\hsolu_{n}+\tsolu_{n}
	\qquad &\text{and}\qquad
	\rhs=\hrhs+\trhs\,,
\end{align*}
with
\begin{align}\label{equ:proof_two-equ}
\begin{split}
	\hsolu_{n},\hrhs&\in\dsp\ovl{\dsp\nabla\Hoztom}\times\ovl{\dsp\nabla\Hoznom}
	\subset\zrztom\times\zrznom\,,\\
	\tsolu_{n},\trhs &\in\Big(\,\Rztom\cap\eps^{-1}\zdznom\,\Big)
	\times\Big(\,\Rznom\cap\mu^{-1}\zdztom\,\Big)\,.
\end{split}	
\end{align}
Inserting these (orthogonal) decompositions in the differential equation we end 
up with two equations
\begin{align*}
	-\om_{n}\hsolu_{n}=\hrhs
	\qquad &\text{and}\qquad
	\big(\,\M-\om_{n}\,\big)\,\tsolu_{n}=\trhs\,,	
\end{align*}
noting that the first one is trivial and implies $\lt$-convergence of 
$(\hsolu_{n})_{n\in\N}$. For dealing with the second equation we need the 
following additional assumption on $(\solu_{n})_{n\in\N}$, which we will prove 
in the end:\\[-9pt] 
\begin{align}\label{equ:proof_proof-cond}
	\forall\;t<-1/2\;\;\,\exists\;c\in(0,\infty)\;\;\,\forall\;n\in\N:
	\quad
	\normlttom{\solu_{n}}\leq c
\end{align}
Let $\that<-1/2$ and $c\in(0,\infty)$ such that \eqref{equ:proof_proof-cond} 
holds. Then, by construction and \eqref{equ:proof_two-equ}$_{2}$, the sequence 
$(\tsolu_{n})_{n\in\N}$ is bounded in 
$\big(\,\Rthattom\cap\eps^{-1}\zdthatnom\,\big)
\times\big(\,\Rthatnom\cap\mu^{-1}\zdthattom\,\big)$. Hence 
(\,Theorem \ref{theo:bvp_WST} and Lemma \ref{lem:bvp_WST}\,), 
$(\tsolu_{n})_{n\in\N}$ has a subsequence $(\tsolu_{\pi(n)})_{n\in\N}$ 
converging in $\ltttilom$ for some $\ttil<\that$ and by 
the equation even in $\Rttiltom\times\Rttilnom$. 
Consequently, the entire sequence $(\solu_{\pi(n)})_{n\in\N}$ converges in 
$\Rttilom$ to some $\solu$ satisfying
\begin{align*}
	\solu\in\Rttiltom\times\Rttilnom
	\qquad\text{and}\qquad
	\big(\,\M-\om\,\big)\,\solu=\rhs\,.
\end{align*}
Additionally, with Corollary \ref{cor:pol-est_polynomial-decay-es} and Lemma 
\ref{lem:prel_relationships} we obtain for $n\in\N$ and arbitrary 
$v\in\gk{\,\calM-\om\,}$
\begin{align*}
	0=\scpltLaom{\rhs}{v}
	 &=\scpltLaom{\big(\,\M-\om_{\pi(n)}\,\big)\,\solu_{\pi(n)}}{v}\\
	 &=\scpltLaom{\solu_{\pi(n)}}{\big(\,\M-\ovl{\om}_{\pi(n)}\,\big)\,v}
	  =\big(\,\om-\om_{\pi(n)}\,\big)\cdot\scpltLaom{\solu_{\pi(n)}}{v}\,.
\end{align*}
Hence $\scpltLaom{\solu_{\pi(n)}}{v}=0$ and, as $\scpltLaom{\cdot}{v}$ is 
continuous on $\ltttilom\times\ltttilom$ by (1), we obtain 
\begin{align*}
	\scpltLaom{\solu}{v}
	 =\lim_{n\rightarrow\infty}\scpltLaom{\solu_{\pi(n)}}{v}=0\,.
\end{align*}
Thus, up to now, we have constructed a vector field 
$\solu\in\gk{\,\calM-\om\,}^{\perp_{\La}}$, 
which has the right boundary conditions and satisfies the differential equation. 
But for being a radiating solution, it still remains to show, that 
$\solu\in\Rsmom{-\frac{1}{2}}$ and enjoys the radiation 
condition \eqref{equ:bvp_rad-sol-cond2}. For that let $t<-1/2$. Then, by Lemma 
\ref{lem:pol-est_a-priori}, there exist $c,\da\in(0,\infty)$ and some 
$\check{t}>-1/2$, such that for $n\in\N$ large enough we obtain uniformly in
$\sigma_{\pi(n)}$, $\solu_{\pi(n)}$, $f$ and $\rtil>0$:
\begin{align*}
	\norm{\solu_{\pi(n)}}_{\Rt(\Omrtil)}	
	 +\norm{\big(\hspace*{0.02cm}\Laz+\sqrt{\eps_{0}\mu_{0}}\;\Xi\,\big)
	 \hspace*{0.02cm}\solu_{\pi(n)}}_{\lsymb^{2}_{\mathrm{\check{t}}}(\Omrtil)}
	 \leq c\cdot\Big(\normltsom{\rhs}+\normltomda{\solu_{\pi(n)}}\Big)\,.
\end{align*}
Sending $n\To\infty$ and afterwards $\rtil\To\infty$ (\,monotone convergence\,) 
we obtain 
\begin{align}\label{equ:main-res_a-priori-final}
	\normRtom{\solu}	
	 +\norm{\big(\hspace*{0.02cm}\Laz+\sqrt{\eps_{0}\mu_{0}}\;\Xi\,\big)
	 \hspace*{0.02cm}\solu}_{\lsymb^{2}_{\mathrm{\check{t}}}(\Om)}
	 \leq c\cdot\Big(\normltsom{\rhs}+\normltomda{\solu}\Big)<\infty\,,
\end{align}
yielding 
\begin{align*}
	u\in\Rsmom{-\frac{1}{2}}
	\qquad\text{and}\qquad
	\big(\hspace*{0.02cm}\Laz+\sqrt{\eps_{0}\mu_{0}}\;\Xi\,\big)
	\hspace*{0.02cm}\solu\in\ltbigom{-\frac{1}{2}}\,.
\end{align*}
This completes the proof of existence, if we can show 
\eqref{equ:proof_proof-cond}. To this end we assume it to be wrong, i.e., 
there exists $t<-1/2$ and a sequence 
$(\solu_{n})_{n\in\N}\subset\Rttom\times\Rtnom$, 
$\solu_{n}:=\mathcal{L}_{\om_{n}}f$ with 
$\normlttom{\solu_{n}}\To\infty$ for $n\To\infty$. Defining 
\begin{align*}
	\csolu_{n}:=\normlttom{\solu_{n}}^{-1}\cdot\solu_{n}
	\qquad\text{and}\qquad	
	\crhs_{n}:=\normlttom{\solu_{n}}^{-1}\cdot\rhs\,,
\end{align*}
we have
\begin{align*}
	\normlttom{\csolu_{n}}=1\,,
	\qquad
	\crhs\To 0\;\text{ in }\,\ltsom
	\qquad\text{and}\qquad \big(\,\M-\om_{n}\,\big)\,\csolu_{n}=\crhs_{n}\,.
\end{align*}
Then, repeating the arguments from above, we obtain some $\check{t}<t$ and 
a subsequence $(\csolu_{\pi(n)})_{n\in\N}$ converging in 
$\lsymb^{2}_{\mathrm{\check{t}}}(\Om)$ to some 
$\csolu\in\gk{\,\calM-\om\,}\cap\gk{\,\calM-\om\,}^{\perp_{\La}}$, hence 
$\csolu=0$. But Lemma \ref{lem:pol-est_a-priori} ensures the existence of 
$c,\da\in(0,\infty)$ (\,independent of $\sigma_{\pi(n)}$, $\csolu_{\pi(n)}$ 
and $\crhs_{\pi(n)}$\,) such that
\begin{align*}
	1=\normlttom{\csolu_{\pi(n)}}
	 \leq c\cdot\Big(\,\big\|\,\crhs_{\pi(n)}\,\big\|_{\ltsom}
	 +\normltomda{\csolu_{\pi(n)}}\Big)
	 \xrightarrow{\;n\rightarrow\infty\;}0
\end{align*}
holds; a contradiction.\\[8pt]
$(\mathrm{5})$: Let $-t,s>1/2$. By $(\mathrm{4})$ the solution operator 
\begin{align*}
	\map{\sol}
	 {\underbrace{\ltsom\cap\gk{\,\calM-\om\,}^{\perp_{\La}}}
	  _{=:\dod(\sol)}}
	 {\underbrace{\big(\,\Rttom\times\Rtnom\,\big)
	  \cap\gk{\,\calM-\om\,}^{\perp_{\La}}}_{=:\calR(\sol)}}	
\end{align*}
is well defined. Furthermore, due to the polynomial decay of eigensolutions, 
$\dod(\sol)$ is closed in $\ltsom$. Thus, the assertion follows from the closed 
graph theorem, if we can show that $\sol$ is closed. Therefore, take 
$(\rhs_{n})_{n\in\N}\subset\dod(\sol)$ with
\begin{align*}
	\rhs_{n}\To\rhs	\;\text{ in }\,\ltsom
	\qquad\text{and}\qquad
	\solu_{n}:=\sol\rhs_{n}\To u\;\text{ in }\,\Rttom\times\Rtnom\,.
\end{align*}
Then clearly $\rhs\in\dod(\sol)$, $u\in\calR(\sol)$ and as 
$\big(\,\M-\om\,\big)\,\solu_{n}=\rhs_{n}$, we obtain 
$\big(\,\M-\om\,\big)\,\solu=\rhs$. Now estimate 
\eqref{equ:main-res_a-priori-final} (\,along with monotone convergence\,) 
shows as before 
\begin{align*}
	u\in\Rsmom{-\frac{1}{2}}
	\qquad\text{and}\qquad
	\big(\,\Laz+\sqrt{\eps_{0}\mu_{0}}\;\Xi\;\big)\hspace*{0.03cm}
	&\solu\in\ltbigom{-\frac{1}{2}}\,,
\end{align*}
meaning $\solu$ is a radiating solution, i.e., $\solu=\sol\rhs$, which 
completes the proof.\\[-10pt]
\end{proof}
\begin{rem}
	During the discussion at AANMPDE10 (\,10th Workshop on Analysis and Advanced 
	Numerical Methods for Partial Differential Equations\,), M. Waurick and 
	S. Trostorff pointed out, that it is sufficient to use weakly convergent 
	subsequences for the construction of the (\,radiating\,) solution. This is 
	in fact true (\,the radiation condition and regularity properties follow 
	from Lemma \ref{lem:pol-est_a-priori} by the boundedness of the sequence 
	and the weak lower semicontinuity of the norms\,), but it should be noted, 
	that Weck's local selection theorem is still needed to prove
	\eqref{equ:proof_proof-cond}, since here norm convergence is indispensable in 
	order to generate a contradiction. Anyway, we thank both for the vivid 
	discussion and constructive criticism.
\end{rem}
%
%
%
%
\bibliographystyle{plain} 
\bibliography{time-harmonic-biblio}
%
%
%
%
\appendix 
%
%
\section{Technical Tools}
\label{sec:tech-tools}
%
%
\begin{lem}\label{lem:app-tech_weighted-est}
	Let $\Om\subset\rthree$ be an arbitrary exterior domain and 
	$s,\,t,\,\theta\in\reals$ with $t<s$ 
	and $\theta>0$. Then there exist constants $c,\da\in(0,\infty)$ such that
	\begin{align*}
		\normlttom{w}\leq c\cdot\normltomda{w}+\theta\cdot\normltsom{w}	
	\end{align*}
	holds for all $w\in\ltsom$.
\end{lem}
\begin{proof}
Let $\rthree\sm\Om\subset\U(r_{0})$. For $\rtil\geq r_{0}$ we obtain
\begin{align*}
	\normlttom{w}^{2}
		=\norm{w}^{2}_{\ltt(\Omrtil)}+\norm{w}^{2}_{\ltt(\cU(\rtil))}	
		&\leq \big(\hspace*{0.02cm}1+\rtil^{2}\hspace*{0.02cm}\big)^{\max\{0,t\}}
		 \cdot\norm{w}^{2}_{\lt(\Omrtil)}
		 +\big(\hspace*{0.02cm}1+\rtil^{2}\hspace*{0.02cm}\big)^{t-s}
		 \cdot\norm{w}^{2}_{\lts(\cU(\rtil))}\\
		&\leq \big(\hspace*{0.02cm}1+\rtil^{2}\hspace*{0.02cm}\big)^{\max\{0,t\}}
		 \cdot\norm{w}^{2}_{\lt(\Omrtil)}
	     +\big(\hspace*{0.02cm}1+\rtil^{2}\hspace*{0.02cm}\big)^{t-s}
		 \cdot\normltsom{w}.
\end{align*}
Since $t<s$ we can choose $\rtil$ such that 
$\big(\hspace*{0.02cm}1+\rtil^{2}\hspace*{0.02cm}\big)^{t-s}\leq \theta^{2}$,
which completes the proof.
\end{proof}
\begin{lem}\label{lem:app-tech_liminf}
For $\rtil>0$ and $f\in\lo(\rn)$ it holds
\begin{align*}
	\liminf_{r\rightarrow\infty}\;
	r\hspace*{-0.04cm}\int_{\Sp(r)}|\,f\hspace*{0.5mm}|\;\dlanmo=0\,.
\end{align*}
\end{lem}
\begin{proof}
	Otherwise there exists $\rhat>0$ and $c>0$ such that
	\begin{align*}
		\int_{\Sp(r)}|\,f\hspace*{0.5mm}|\;\dlanmo\geq\frac{c}{r}
		\qquad\forall\;r\geq\rhat 
	\end{align*}
	and using Fubini's theorem we obtain
	\begin{align*}
		\norm{f}_{\lo(\rn)}^{2}\geq\int_{\cU(\rhat)}|\,f\hspace*{0.5mm}|\;\dlan
		=\int_{\rhat}^{\infty}\int_{\Sp(r)}|\,f\hspace*{0.5mm}|\;\dlanmo\;dr
		\geq c\cdot\int_{\rhat}^{\infty}\frac{1}{r}\;dr=\infty\,,
	\end{align*}
	a contradiction.
\end{proof}
\begin{lem}\label{lem:app-tech_part-Int}
	Let $\Om\subset\rthree$ be an exterior weak Lipschitz domain with boundary 
	$\Ga$ and weak Lipschitz boundary parts $\Ga_{1}$ and 
	$\Ga_{2}=\Ga\sm\ovl{\Ga}_{1}$. Furthermore, let 
	$\rhat,\rtil\in\reals_{+}$ with 
	$\rtil>\rhat$ and\;\,$\rthree\sm\Om\subset\U(\rhat)$ as well as 
	$\phi\in\sfC^{0}\big(\hspace*{0.03cm}[\hspace*{0.03cm}\rhat,
	\rtil\hspace*{0.03cm}]
	\hspace*{0.05cm},\hspace*{-0.01cm}\C\hspace*{0.03cm}\big)$. If 
	$\solu\in\Rttom\times\Rtnom$ for some $t\in\reals$, it holds
	\begin{align}\label{equ:app-tech_part-Int}
         \scp{\Phi\hspace*{0.03cm}\Xi\hspace*{0.03cm}\solu}{\Laz\solu}_{\lt(\G(\rhat,\rtil))}
        =\scp{\Psi\hspace*{0.03cm}\Rot\solu}{\Laz\solu}_{\lt(\Omrtil)}
		+\scp{\Psi\solu}{\Rot\Laz\solu}_{\lt(\Omrtil)}\,,
	\end{align}
	where $\Phi:=\phi\circ r$, $\Psi:=\psi\circ r$, and 
	\begin{align*}
		\maps{\psi}{[0,\rtil]}{\C}{\sigma}
		{\int_{\max\{\rhat,\sigma\}}^{\rtil}\phi(\tau)\,d\tau}.
	\end{align*}
\end{lem}
\begin{proof}
	As $\cictom$ respectively $\cicnom$ is dense in $\Rttom$ respectively $\Rtnom$ by definition 
	it is enough to show equation \eqref{equ:app-tech_part-Int} for 
	$\solu=(\solu_{1},\solu_{2})\in\cictom\times\cicnom\subset\cicrthree$. 
	Observing that the support of products of $\solu_{1}$ and $\solu_{2}$
	is compactly supported in some $\Theta\subset\ovl{\Theta}\subset\Om$, 
	we may choose a cut-off function $\varphi\in\cicom\subset\cicrthree$
	with $\varphi|_{\Theta}=1$ and replace $\solu$ by $\varphi\solu=:v=:(E,H)$.
	Without loss of generality we assume $\rthree\sm\Theta\subset\U(\rhat)$.
	Using Gauss's divergence theorem we compute
	\begin{align*}
	 	\scp{\Phi\hspace*{0.03cm}\Xi\hspace*{0.03cm}\solu}{\Laz\solu}
	 	 _{\lt(\G(\rhat,\rtil))}
	 	&=\int_{\rhat}^{\rtil}\phi(r)\,\scp{\Xi\hspace*{0.03cm}\solu}{\Laz\solu}
	   	 _{\lt(\Sp(r))}\;dr
	 	=\int_{\rhat}^{\rtil}\phi(r)\,\scp{\Xi\hspace*{0.03cm}v}{\Laz v}
	   	 _{\lt(\Sp(r))}\;dr\\
	 	&=\int_{\rhat}^{\rtil}\phi(r)\,\Big(\,\mu_{0}\scp{\xi\times E}{H}
	   	 _{\lt(\Sp(r))}-\eps_{0}\scp{\xi\times H}{E}_{\lt(\Sp(r))}\,\Big)\,dr\\
	 	&=\int_{\rhat}^{\rtil}\phi(r)\int_{\Sp(r)}\Big(\,\mu_{0}\,
	 	  \xi\cdot\big(E\times \ovl{H}\hspace*{0.03cm}\big)-\eps_{0}\,
	 	  \xi\cdot\big(H\times \ovl{E}\hspace*{0.03cm}\big)\,\Big)\,d\lambda_{s}^{2}\,dr\\
	 	&=\int_{\rhat}^{\rtil}\phi(r)\int_{\U(r)}\Big(\,\mu_{0}
	 	  \div\big(E\times \ovl{H}\hspace*{0.03cm}\big)-\eps_{0}
	 	  \div\big(H\times \ovl{E}\hspace*{0.03cm}\big)\,\Big)\,d\lambda^{3}\,dr\,.
	\end{align*}
	Note that
	\begin{align*}
		\mu_{0}\div\big(E\times\ovl{H}\hspace*{0.03cm}\big)
		 -\eps_{0}\div\big(H\times \ovl{E}\hspace*{0.03cm}\big)
		&=\mu_{0}\Big(\,\ovl{H}\rot E-E\rot\ovl{H}\,\Big)
		 -\eps_{0}\Big(\,\ovl{E}\rot H-H\rot\ovl{E}\,\Big)\\
		&=\Big(\big(\mu_{0}
		 \ovl{H}\big)\rot E
		 -\big(\eps_{0}\ovl{E}\big)\rot H\,\Big)
		 +\Big(\,H\rot\big(\eps_{0}\ovl{E}\big)
		 -E\rot\big(\mu_{0}\ovl{H}\big)\Big)\\
		&=\ovl{\Laz v}\cdot\Rot v+v\cdot\Rot\ovl{\Laz v}.
	\end{align*}
	Hence, by using Fubini's theorem, we see
	\begin{align*}
		\scp{\Phi\hspace*{0.03cm}\Xi\hspace*{0.03cm}\solu}{\Laz\solu}
	 	  _{\lt(\G(\rhat,\rtil))}
	 	&=\int_{\rhat}^{\rtil}\phi(r)\,\Big(\,\scp{\Rot v}
	 	  {\Laz v}_{\lt(\U(r))}
		  +\scp{ v}{\Rot\Laz v}_{\lt(\U(r))}\,\Big)\,dr\\
		&=\int_{\rhat}^{\rtil}\phi(r)\int_{0}^{r}
		  \Big(\,\scp{\Rot v}{\Laz v}_{\lt(\Sp(\sigma))}
		  +\scp{ v}{\Rot\Laz v}_{\lt(\Sp(\sigma))}\,\Big)\,d\sigma\,dr\\
		&=\int_{0}^{\rtil}\int_{\max\{\rhat,\sigma\}}^{\rtil}\phi(r)\,
		  \Big(\,\scp{\Rot v}{\Laz v}_{\lt(\Sp(\sigma))}
		  +\scp{ v}{\Rot\Laz v}_{\lt(\Sp(\sigma))}\,\Big)\,dr\,d\sigma\\
		&=\int_{0}^{\rtil}\psi(\sigma)\,
		  \Big(\,\scp{\Rot v}{\Laz v}_{\lt(\Sp(\sigma))}
		  +\scp{ v}{\Rot\Laz v}_{\lt(\Sp(\sigma))}\,\Big)\,d\sigma\\
		&=\scp{\Psi\hspace*{0.03cm}\Rot v}{\Laz v}_{\lt(\U(\rtil))}
		  +\scp{\Psi v}{\Rot\Laz v}_{\lt(\U(\rtil))}\\
		&=\scp{\Psi\hspace*{0.03cm}\Rot v}{\Laz v}_{\lt(\Omrtil)}
		  +\scp{\Psi v}{\Rot\Laz v}_{\lt(\Omrtil)}\\
		&=\scp{\Psi\hspace*{0.03cm}\Rot\solu}{\Laz\solu}_{\lt(\Omrtil)}
		  +\scp{\Psi\solu}{\Rot\Laz\solu}_{\lt(\Omrtil)}\,,
	\end{align*}
where the last line follows by construction of $v$.
\end{proof}
\noindent
We end this section with a Lemma, which will be needed to prove the polynomial 
decay and a-priori-estimate for the Helmholtz equation and can be shown by 
elementary partial 
integration.
\begin{lem}\label{lem:app-tech_Int-rules}
	Let $\ssolu\in\Htwoloc(\reals^{n})$, $0\notin\supp\ssolu$, $m\in\reals$ and 
	$\rtil>0$. Then
	\begin{enumerate}[label=$(\arabic*)$,itemsep=8pt]
		\item $\dsp\Re\int_{\U(\rtil)}
			   r^{m+1}\Delta\ssolu\hspace*{0.06cm}\pr\bssolu$
		\item[] $\dsp\;\quad=\frac{1}{2}\int_{\U(\rtil)}r^m
				\Big(\,(n+m-2)\,|\nabla\ssolu|^{2}-2m\,|\hspace*{0.03cm}\pr
				\ssolu\hspace*{0.04cm}|^{2}\,\Big)
				+\int_{\Sp(\rtil)}r^{m+1}\left(\,|\hspace*{0.03cm}\pr
				\ssolu\hspace*{0.04cm}|^{2}
				-\frac{1}{2}\,|\nabla\ssolu|^{2}\,\right)\,,$
		\item $\dsp\Re\int_{\U(\rtil)}r^{m}\Delta\ssolu\hspace*{0.06cm}\bssolu$
		\item[] $\dsp\;\quad=-\int_{\U(\rtil)}r^{m}\Big(\,|\nabla\ssolu|^{2}
				-\frac{m}{2}\,(n+m-2)\,r^{-2}\,|\ssolu|^{2}\,\Big)
				+\int_{\Sp(\rtil)}r^{m}\Big(\,\Re\big(\pr\ssolu
				\hspace*{0.06cm}\bssolu\hspace*{0.04cm}\big)
				-\frac{m}{2}\,r^{-1}\,|\ssolu|^{2}\,\Big)\,,$
		\item $\dsp\Im\int_{\U(\rtil)}r^{m}\Delta\ssolu\hspace*{0.06cm}\bssolu
		       =-m\int_{\U(\rtil)}r^{m-1}\Im\big(\pr\ssolu
		       \hspace*{0.06cm}\bssolu\hspace*{0.04cm}\big)
		       +\frac{1}{2}\int_{\Sp(\rtil)}r^{m+1}\,|\ssolu|^{2}\,,$
		\item $\dsp\Re\int_{\U(\rtil)}r^{m+1}\hspace*{0.03cm}	
			   \ssolu\hspace*{-0.006cm}\p_{r}\hspace*{-0.07cm}\bssolu
			   =-\frac{1}{2}\int_{\U(\rtil)}r^{m}\,(n+m)\,|\ssolu|^{2}
			   +\frac{1}{2}\int_{\Sp(\rtil)}r^{m+1}\,|\ssolu|^{2}\,,$\\[2pt]
	\end{enumerate}
	where $\p_{r}:=\xi\cdot\nabla$.
\end{lem}
%
%
\section{Polynomial Decay and A-Priori Estimate for the Helmholtz Equation}
\label{sec:pol-est_hh-theory}
%
%
In this section we present well known results for the Helmholtz equation, which 
we will use to achieve similar results for Maxwell's equations. We start with a 
regularity result \big(\,cf. \cite[Lemma 4]{weck_generalized1_1997}\,\big) and 
the polynomial decay \big(\,cf. \cite[Lemma 5]{weck_generalized1_1997}\,\big).
\begin{lem}\label{lem:app-hh_reg}
	Let $t\in\reals$. If $\ssolu\in\lttrn$ and $\Delta\ssolu\in\lttrn$, it holds
	$\ssolu\in\Htwotrn$ and 
	\begin{align*}
		\normHtwotrn{\ssolu}
			\leq c\cdot\Big(\normlttrn{\Delta\ssolu}+\normlttrn{\ssolu}\Big)
	\end{align*}
	with $c\in(0,\infty)$ independent of $\ssolu$\;and $\Delta\ssolu$.
\end{lem}
\begin{proof}
	For $t=0$ we have $\ssolu,\Delta\ssolu\in\ltrn$ and 
	using Fourier transformation we obtain\\[-7pt]
	\begin{align}
	\label{equ:app-hh_htwo-ww-est}
	\begin{split}
		\normltrn{\Delta\ssolu}^{2}+\normltrn{\ssolu}^{2}
		&=\normltrn{r^{2}\calF(\ssolu)}^{2}+\normltrn{\calF(\ssolu)}^{2}\\
		&=\int_{\rn}(r^4+1)\,|\calF(\ssolu)|^{2}
		 \geq\frac{1}{2}\cdot\normltrn{(1+r^{2})\,\calF(\ssolu)}^{2}\,,
	\end{split}
	\end{align}
	yielding $\ssolu\in\Htworn$ and the desired estimate. So let us switch to  
	$t\neq0$. Then, using a well known result concerning inner 
	regularity \big(\,e.g., \cite[Chapter VII, $\S\hspace*{0.2mm}3.2$, Theorem 1]
	{dautray_mathematical_2000}\,\big), we already have $\ssolu\in\Htwolocrn$. 
	Now let $\rtil>1$ and define $\eta_{\rtil}\in\dsp\cicrn$ by 
	$\eta_{\rtil}(x):=\rho^t\eta(r(x)/\rtil)$. Then 
	$\eta_{\rtil}\ssolu\in\Htworn$,
	\begin{align*}
		|\nabla\eta_{\rtil}|\leq c\cdot\rho^{t-1}\;\,\text{with}\;\,c=c(t)>0\,,
	\end{align*}
	and 
	\begin{align*}
		\scp{\nabla\big(\eta_{\rtil}\ssolu\big)}
		 {\nabla\big(\eta_{\rtil}\ssolu\big)}_{\ltrn}
		&=\Re\scpltrn{\nabla\ssolu}{\nabla\big(\eta_{\rtil}^{2}\ssolu\big)}
		  +\norm{\big(\nabla\eta_{\rtil}\big)\ssolu}_{\ltrn}^{2}\\
		&\leq c\cdot\Big(\normltrn{\eta_{\rtil}\Delta\ssolu}
		 \normltrn{\eta_{\rtil}\ssolu}+\normlttmorn{\ssolu}^{2}\Big)\\
		&\leq c\cdot\Big(\normlttrn{\Delta\ssolu}^{2}
		 +\normlttrn{\ssolu}^{2}\Big)\,,
	\end{align*}
	with $c=c(n,t)\in(0,\infty)$, hence
	\begin{align*}
		\norm{\nabla\ssolu}_{\ltt(\B(\rtil))}
		\leq
		 \norm{\nabla(\eta_{\rtil}\ssolu)-(\nabla\eta_{\rtil})\ssolu}_{\ltrn}
		 \leq c(n,t)\cdot\Big(\normlttrn{\Delta\ssolu}
		 +\normlttrn{\ssolu}\Big)\,.	
	\end{align*}
	Sending $\rtil\To\infty$ (\,monotone convergence\,) shows 
	$\ssolu\in\Hotrn$ and
	\begin{align}\label{equ:app-hh_ho-est}
		\normHotrn{\ssolu}
			\leq c(n,t)\cdot\Big(\normlttrn{\Delta\ssolu}
			+\normlttrn{\ssolu}\Big)\,.	
	\end{align}
	Moreover,
	\begin{align*}
		\Delta\big(\rho^{t}\ssolu\big)
			=t\Big(n+(t-2)\frac{r^{2}}{1+r^{2}}\Big)\rho^{t-2}\ssolu
			 +2r\rho^{t-2}\pr\ssolu+\rho^{t}\Delta\ssolu\,,
	\end{align*}
	such that with \eqref{equ:app-hh_ho-est} we obtain
	\begin{align}\label{equ:app-hh_Delta-est}
			\normltrn{\Delta\big(\rho^{t}\ssolu\big)}
				\leq c\cdot\Big(\normlttrn{\Delta\ssolu}
				+\normlttrn{\ssolu}\Big)\,,
	\end{align}
	with $c\in(0,\infty)$ independent of $\ssolu$ and $\Delta\ssolu$. Hence 
	$\Delta\big(\rho^{t}\ssolu\big)\in\ltrn$ and we may apply the first case. 
	This shows $\rho^t\ssolu\in\Htworn$ and using \eqref{equ:app-hh_htwo-ww-est}, 
	\eqref{equ:app-hh_ho-est} and\;\eqref{equ:app-hh_Delta-est}, we obtain 
	(\,uniformly w.r.t. $\ssolu$ and $\Delta\ssolu$\,)
	\begin{align*}
		\normHtwotrn{\ssolu}
		  &\leq c\cdot\Big(\normHtworn{\rho^{t}\ssolu}
		   +\normltrn{\big(\nabla\hspace*{-0.03cm}\rho^{t}\big)
		    \hspace*{0.04cm}\nabla\ssolu}
		   +\normltrn{\big(\nabla\hspace*{-0.03cm}\rho^{t}\big)\ssolu}
		   +\sum_{|\alpha|=2}\normltrn{\big(\p^{\alpha}
		    \hspace*{-0.09cm}\rho^{t}\hspace*{0.03cm}\big)\ssolu}\Big)\\
		  &\leq c\cdot\Big(\normltrn{\Delta\big(\rho^{t}\ssolu\big)}
		   +\normltrn{\rho^{t}\ssolu}+\normlttmorn{\nabla\ssolu}
		   +\normlttmorn{\ssolu}\Big)\\[8pt]
		  &\leq c\cdot\Big(\normlttrn{\Delta\ssolu}
		   +\normlttrn{\ssolu}\Big)
	\end{align*}
	yielding $\ssolu\in\Htwotrn$ and the required estimate.
\end{proof}
\begin{lem}[Polynomial decay]\label{lem:app-hh_pol-dec}
	Let $J\Subset\reals\sm(0)$ be some interval, $\ga\in J$ and $s,t\in\reals$ 
	with $t>-1/2$ and $t\leq s$. If $\ssolu\in\lttrn$ and 
	$\srhs:=\big(\,\Delta+\ga^{2}\,\big)\,\ssolu\in\ltsporn$ it holds 
	\begin{align*}
		\ssolu\in\Htwosrn
		\qquad\text{and}\qquad
		\normHtwosrn{\ssolu}\leq c\cdot\Big(\normltsporn{g}
		 +\normltsmorn{\ssolu}\Big)
	\end{align*}
	with $c=c(n,s,J)\in (0,\infty)$ not depending on $\ga,\srhs$\;or $\ssolu$.
\end{lem}
\begin{proof}
	The assertion follows directly from Lemma \ref{lem:app-hh_reg}, if we can 
	show 
	\begin{align*}
		\ssolu\in\ltsrn
		\qquad\text{with}\qquad		
		\normltsrn{\ssolu}\leq c\cdot\Big(\normltsporn{g}
		 +\normltsmorn{\ssolu}\Big)\,.
	\end{align*}
	Therefore let $v:=\check{\chi}\ssolu$, where $\check{\chi}\in\cirn$ 
	with $\check{\chi}=1$ on $\cU(1)$ and vanishing in a neighbourhood of the 
	origin. By assumption we already have $\ssolu\in\Htwotrn$ 
	(\,cf. Lemma \ref{lem:app-hh_reg}\,), hence $v\in\Htwolocrn$ and we 
	may apply the partial integration rules from Lemma 
	\ref{lem:app-tech_Int-rules} to
	\begin{align*}
		\Re\int_{\Gap}\big(\,\Delta\ssolu+\ga^{2}\ssolu\,\big)
		   \big(\,r^{2t+1}\pr\bssolu
		   +\beta r^{2t}\hspace*{0.01cm}\bssolu\,\big)
		   =\Re\int_{\Gap}\big(\,\Delta v+\ga^{2} v\,\big)
		   \big(\,r^{2t+1}\pr\bar{v}
		   +\beta r^{2t}\hspace*{0.01cm}\bar{v}\,\big)=\ldots,
	\end{align*}
	with $\rtil>\rhat\geq1$ and
	\begin{align*}
		\beta:=\max\big\{\,(n-1)/2\hspace*{0.05cm},
		\hspace*{0.01cm}t+(n-1)/2\,\big\}\,.	
	\end{align*}
	After some rearrangements this leads to
	\begin{align}\label{equ:app-hh_poly-dec-I1}
		\begin{split}
			&\hspace*{-1.2cm}
			 \int_{\Gap}r^{2t}\hspace*{0.03cm}
			 \Big(\,\big(\hspace*{0.05cm}\beta-(n+2t-2)/2\hspace*{0.05cm}\big)
			 \hspace*{0.03cm}|\nabla \ssolu|^{2}
		 	 +\big(\hspace*{0.05cm}(n+2t)/2-\beta\hspace*{0.05cm}\big)
		 	 \hspace*{0.03cm}\ga^{2}|\ssolu|^{2}\;\Big)\\
			&\hspace*{-0.6cm}+2t\int_{\Gap}r^{2t}\hspace*{0.03cm}|\pr\ssolu|^{2}
		 	 +\int_{\Sp(\rtil)}\rtil^{2t+1}\hspace*{0.03cm}|\nabla \ssolu|^{2}\\
			&=-\Re\int_{\Gap}\big(\,\Delta\ssolu+\ga^{2}\ssolu\,\big)
		 	 \big(\,r^{2t+1}\pr\bssolu+\beta r^{2t}\hspace*{0.01cm}\bssolu\,\big)
		 	 +t(n+2t-2)\beta\int_{\Gap}r^{2t-2}\hspace*{0.03cm}|\ssolu|^{2}\\
			&\hspace*{0.6cm}
			 +\int_{\Sp(\rhat)}\rhat^{2t+1}\Big(\,\beta t\rhat^{-2}
			 \hspace*{0.03cm}|\ssolu|^{2}-\beta\rhat^{-1}\Re\big(\pr\ssolu
			 \hspace*{0.06cm}\bssolu\hspace*{0.04cm}\big)
			 -|\pr\ssolu|^{2}\,\Big)\\
			&\hspace*{1.2cm}+\int_{\Sp(\rtil)}\rtil^{2t+1}\Big(\,|\pr\ssolu|^{2}
		 	 +\beta\rtil^{-1}\Re\big(\pr\ssolu\hspace*{0.06cm}\bssolu
			 \hspace*{0.04cm}\big)
			 -\beta t\rtil^{-2}\hspace*{0.03cm}|\ssolu|^{2}\,\Big)\\
			&\hspace*{1.8cm}+\frac{1}{2}\int_{\Sp(\rtil)}\rtil^{2t+1}
			 \Big(\,|\nabla\ssolu|^{2}+\ga^{2}|\ssolu|^{2}\,\Big)
			 +\frac{1}{2}\int_{\Sp(\rhat)}\rhat^{2t+1}\Big(\,|\nabla\ssolu|^{2}
		 	 -\ga^{2}|\ssolu|^{2}\,\Big)\,.
		\end{split}
	\end{align}
	Let us first have a look on the left hand side of this equation. For 
	$t\geq 0$ (\,i.e., $\beta=t+(n-1)/2$\,) we skip the second and third integral 
	to obtain
	\begin{align*}
		&\int_{\Gap}r^{2t}\hspace*{0.03cm}\Big(\,\big(\hspace*{0.05cm}\beta
		 -(n+2t-2)/2\hspace*{0.05cm}\big)\hspace*{0.03cm}|\nabla \ssolu|^{2}
		 +\big(\hspace*{0.05cm}(n+2t)/2-\beta\hspace*{0.05cm}\big)
		 \hspace*{0.03cm}\ga^{2}|\ssolu|^{2}\;\Big)\\
		&\hspace*{0.6cm}+2t\int_{\Gap}r^{2t}\hspace*{0.03cm}|\pr\ssolu|^{2}
		 +\int_{\Sp(\rtil)}\rtil^{2t+1}\hspace*{0.03cm}|\nabla \ssolu|^{2}\\
		&\hspace*{1.2cm}\geq\frac{1}{2}\int_{\Gap}r^{2t}\hspace*{0.03cm}
		 \Big(\,\big(\hspace*{0.03cm}2\beta-(n+2t-2)\hspace*{0.03cm}\big)
		 \hspace*{0.03cm}|\nabla \ssolu|^{2}+\big(\hspace*{0.03cm}(n+2t)-2\beta
		 \hspace*{0.03cm}\big(\hspace*{0.03cm}\ga^{2}|\ssolu|^{2}\;\Big)\\
		&\hspace*{1.2cm}=\frac{1}{2}\int_{\Gap}r^{2t}\hspace*{0.03cm}
		 \Big(\,|\nabla\ssolu|^{2}+\ga^{2}|\ssolu|^{2}\,\Big)\,,
	\end{align*}
	while in the case of $t<0$ (\,i.e., $\beta=(n-1)/2$\,) we just skip the third 
	integral and end up with
	\begin{align*}
		&\int_{\Gap}r^{2t}\hspace*{0.03cm}\Big(\,\big(\hspace*{0.05cm}\beta
		 -(n+2t-2)/2\hspace*{0.05cm}\big)\hspace*{0.03cm}|\nabla \ssolu|^{2}
		 +\big(\hspace*{0.05cm}(n+2t)/2-\beta\hspace*{0.05cm}\big)
		 \hspace*{0.03cm}\ga^{2}|\ssolu|^{2}\;\Big)\\
		&\hspace*{0.6cm}+2t\int_{\Gap}r^{2t}\hspace*{0.03cm}|\pr\ssolu|^{2}
		 +\int_{\Sp(\rtil)}\rtil^{2t+1}\hspace*{0.03cm}|\nabla \ssolu|^{2}\\
		&\hspace*{1.2cm}\geq\int_{\Gap}r^{2t}\hspace*{0.03cm}
		 \Big(\,\big(\hspace*{0.03cm}\beta-(n+2t-2)/2+2t\hspace*{0.03cm}\big)
		 \hspace*{0.03cm}|\nabla \ssolu|^{2}+\big(\hspace*{0.03cm}(n+2t)/2-\beta
		 \hspace*{0.03cm}\big)\hspace*{0.03cm}\ga^{2}|\ssolu|^{2}\;\Big)\\
		&\hspace*{1.2cm}
		 =\bigg(\frac{1}{2}+t\bigg)\int_{\Gap}r^{2t}\hspace*{0.03cm}
		 \Big(\,|\nabla\ssolu|^{2}+\ga^{2}|\ssolu|^{2}\,\Big)\,,
	\end{align*}
	since $|\pr\ssolu|\leq|\nabla\ssolu|$. Thus for arbitrary $t\in\reals$ the 
	left hand side of \eqref{equ:app-hh_poly-dec-I1} can be estimated from below 
	by
	\begin{align*}
		\min\bigg\{\frac{1}{2},\frac{1}{2}+t\bigg\}\int_{\Gap}r^{2t}
		\hspace*{0.03cm}
		\Big(\,|\nabla\ssolu|^{2}+\ga^{2}|\ssolu|^{2}\,\Big)\,.
	\end{align*}
	For the right hand side we have (\,$\rtil>1$\,)
	\begin{align*}
		&\int_{\Sp(\rtil)}\rtil^{2t+1}\Big(\,|\pr\ssolu|^{2}+\beta\rtil^{-1}
		 \Re\big(\pr\ssolu\hspace*{0.06cm}\bssolu\hspace*{0.04cm}\big)-\beta t 
		 \rtil^{-2}|\ssolu|^{2}\,\Big)\\
		&\hspace*{0.6cm}
		 \leq\int_{\Sp(\rtil)}\rtil^{2t+1}\Big(\,|\pr\ssolu|^{2}+\beta
		 |\pr\ssolu\hspace*{0.06cm}\bssolu\hspace*{0.02cm}|
		 +\beta|t||\ssolu|^{2}\,\Big)
		 \leq c\cdot\int_{\Sp(\rtil)}\rtil^{2t+1}
		 \Big(\,|\nabla\ssolu|^{2}+|\ssolu|^{2}\,\Big)\,,
	\end{align*}
	as well as
	\begin{align*}
		&\int_{\Sp(\rhat)}\rhat^{2t+1}\Big(\,\beta t\rhat^{-2}
		 \hspace*{0.03cm}|\ssolu|^{2}-\beta\rhat^{-1}\Re\big(\pr\ssolu
		\hspace*{0.06cm}\bssolu\hspace*{0.04cm}\big)-|\pr\ssolu|^{2}\,\Big)\\
		&\hspace*{0.6cm}\leq\int_{\Sp(\rtil)}\rtil^{2t+1}
		 \Big(\,\beta|t|\rhat^{-2}\hspace*{0.03cm}|\ssolu|^{2}+\beta\rhat^{-1}
		 |\pr\ssolu\hspace*{0.06cm}\bssolu\hspace*{0.02cm}|\,\Big)
		 \leq c\cdot\int_{\Sp(\rtil)}\rtil^{2t+1}
		 \Big(\,\rhat^{-2}|\ssolu|^{2}+|\nabla\ssolu|^{2}\,\Big)\,,
	\end{align*}
	such that equation \eqref{equ:app-hh_poly-dec-I1} becomes
	\begin{align*}
		&\min\bigg\{\frac{1}{2},\frac{1}{2}+t\bigg\}\int_{\Gap}r^{2t}
		 \hspace*{0.03cm}
		 \Big(\,|\nabla\ssolu|^{2}+\ga^{2}|\ssolu|^{2}\,\Big)\\
		&\hspace*{0.6cm}\leq \int_{\Gap}r^{t+1}|\hspace*{0.02cm}
		 \srhs\hspace*{0.02cm}|
		 \hspace*{0.03cm}\Big(\,r^{t}\hspace*{0.03cm}|\nabla\ssolu|
		 +\beta r^{t-1}\hspace*{0.03cm}|\ssolu|\,\Big)+\beta\hspace*{0.02cm}
		 \big|\hspace*{0.02cm}t(n+2t-2)\hspace*{0.03cm}\big|
		 \int_{\Gap}r^{2t-2}\hspace*{0.03cm}|\ssolu|^{2}\\
		&\hspace*{1.2cm}+c(n,t)\cdot\left(\,\int_{\Sp(\rhat)}
		 \rhat^{2t+1}\Big(\,\rhat^{-2}\hspace*{0.03cm}|\ssolu|^{2}
		 +|\nabla\ssolu|^{2}
		 -\ga^{2}|\ssolu|^{2}\,\Big)+\int_{\Sp(\rtil)}\rtil^{2t+1}
		 \Big(\,|\nabla\ssolu|^{2}+|\ssolu|^{2}\,\Big)\,\right)\,.
	\end{align*}
	By assumption we have $\ssolu\in\Htwotrn$, such that according to 
	Lemma \ref{lem:app-tech_liminf} the lower limit for $\rtil\To\infty$ of the 
	last boundary integral vanishes. Hence we may replace $\Gap$ by $\cU(\rhat)$ 
	and in addition use Young's inequality to obtain
	\begin{align}\label{eq:app-hh_pol-dec-first-est}
		&\hspace*{-0.6cm}
		 \norm{r^{t}\hspace*{0.03cm}\nabla\ssolu}_{\lt(\cU(\rhat))}^{2}
		 +\ga^{2}\norm{r^{t}\hspace*{0.03cm}\ssolu}_{\lt(\cU(\rhat))}^{2}\\
		&\leq c(n,t)\cdot
		 \Big(\norm{r^{t+1}\hspace*{0.03cm}\srhs}_{\lt(\cU(\rhat))}^{2}
		 +\norm{r^{t-1}\hspace*{0.03cm}\ssolu}_{\lt(\cU(\rhat))}^{2}
		 +\int_{\Sp(\rhat)}\rhat^{2t+1}\Big(\,|\nabla\ssolu|^{2}
		 -\ga^{2}|\ssolu|^{2}
		 +\rhat^{-2}\hspace*{0.03cm}|\ssolu|^{2}\,\Big)\Big)\notag\\
		&\leq c(n,t)\cdot\Big(\normlttporn{\srhs}^{2}
		 +\normlttmorn{\ssolu}^{2}+\int_{\Sp(\rhat)}
		 \rhat^{2t+1}\Big(\,|\nabla\ssolu|^{2}
		 -\ga^{2}|\ssolu|^{2}+\rhat^{-2}\hspace*{0.03cm}|\ssolu|^{2}\,\Big)\Big)
		 \notag\,.
	\end{align}
	Now suppose that $s=t$. Then the assertion simply follows by choosing 
	$\rhat:=1$ as the trace theorem bounds the surface integral by
	$\norm{\ssolu}_{\Htwo(\U(1))}^{2}$ and with
	Lemma \ref{lem:app-hh_reg}
	\begin{align*}
		\normHotrn{\ssolu}
		 &\leq c(n,s,J)\cdot\Big(\normlttporn{\srhs}+\normlttmorn{\ssolu}
		 +\norm{\ssolu}_{\Htwo(\U(2))}\Big)\\
		&\leq c(n,s,J)\cdot\Big(\normlttporn{\srhs}
		 +\normlttmorn{\ssolu}
		 +\norm{\ssolu}_{\Htwotmorn}\Big)\\
		&\leq c(n,s,J)\cdot\Big(\normlttporn{\srhs}
		 +\normlttmorn{\ssolu}
		 +\norm{\Delta\ssolu}_{\ltmorn}\Big)\\
		&\leq c(n,s,J)\cdot\Big(\normlttporn{\srhs}
		 +\normlttmorn{\ssolu}\Big)
	\end{align*}
	holds.
	For the case $\ssolu\not\in\ltsrn$ let 
	$\hat{s}:=\sup\setb{m\in\reals}{u\in\lgen{}{2}{m}(\rn)\hspace*{0.03mm}}$. 
	Then, w.l.o.g.\footnote{Otherwise we replace $s$ and $t$ by 
	$t_{k}:=t+k/4$ resp.\;$s_{k}:=t_{k+1}$, $k=0,1,2,\ldots$ and obtain the 
	assertion after finitely many steps of the type $t_{k}<s_{k}\leq t_{k}+1/2$ 
	(\,cf. Appendix, Section \ref{sec:proofs-max}, Proof of Lemma 
	\ref{lem:pol-est_decomp},\,).}, we 
	may assume 
	\begin{align*}
		\hat{s}-1/2<t<\hat{s}<s\leq t+1/2\,,
	\end{align*} 
	hence $\da:=1-2(s-t)\in(0,1)$. Multiplying 
	\eqref{eq:app-hh_pol-dec-first-est} with $\rhat^{-\da}$ and integrating 
	from $1$ to some $\rcheck>1$ leads to
	\begin{align}
	\label{equ:app-hh_pol-dec-sec-est}
	\begin{split}
		\int_{1}^{\rcheck}\rhat^{-\da}\int_{\cU(\rhat)}r^{2t}
		 \Big(\,|\nabla\ssolu|^{2}
		 +\ga^{2}|\ssolu|^{2}\,\Big)\,d\rhat
		&\leq c(n,t)\cdot
		 \bigg(\int_{1}^{\rcheck}\rhat^{-\da}\int_{\cU(\rhat)}
		 r^{2t+2}|\srhs|^{2}+r^{2t-2}|\ssolu|^{2}\,d\rhat\qquad\qquad\;\\
		&\qquad+\int_{\Gaprcheck}r^{2t+1-\da}\Big(\,|\nabla\ssolu|^{2}
		 -\ga^{2}|\ssolu|^{2}
		 +r^{-2}\hspace*{0.03cm}|\ssolu|^{2}\,\Big)\,\bigg).
	\end{split}
	\end{align}
	By Fubini's theorem we have for arbitrary $h\in\lorn$
	\begin{align*}
		\int_{1}^{\rcheck}\rhat^{-\da}\int_{\cU(\rhat)}h\,d\rhat
		&=\int_{1}^{\rcheck}\int_{\rhat}^{\infty}
		 \int_{\Sp(\sigma)}\rhat^{-\da}\,h\,d\sigma\,d\rhat
		=\int_{1}^{\infty}\int_{1}^{\min\{\sigma,\rcheck\}}
		 \rhat^{-\da}\int_{\Sp(\sigma)}\,h\,d\rhat\,d\sigma\\
		&=\int_{1}^{\infty}(1-\da)^{-1}
		 \min\{\sigma^{1-\da}-1,\rcheck^{1-\da}-1\}\,\int_{\Sp(\sigma)}h
		 \,d\sigma\\
		&=\int_{1}^{\infty}\int_{\Sp(\sigma)}
		 \underbrace{(1-\da)^{-1}
		 \min\{r^{1-\da}-1,\rcheck^{1-\da}-1\}}_{=:\theta_{\rcheck}}\,h
		 \,d\sigma=\int_{\cU(1)}\theta_{\rcheck}\,h
	\end{align*}
	such that \eqref{equ:app-hh_pol-dec-sec-est} becomes \big(\,note that 
	$\theta_{\rcheck}\leq (1-\da)^{-1}\cdot r^{1-\da}$ and $1-\da=2(s-t)$\,\big)
	\begin{align}
		&\qquad\int_{\cU(1)}\theta_{\rcheck}\,r^{2t}
		\Big(\,|\nabla\ssolu|^{2}+\ga^{2}|\ssolu|^{2}\,\Big)\label{equ:app-hh_mud-est}\\
		&\leq c(n,t)\cdot\bigg(\int_{\cU(1)}\theta_{\rcheck}
		 \Big(\,r^{2t+2}|\srhs|^{2}+r^{2t-2}|\ssolu|^{2}\,\Big)
		 +\int_{\Gaprcheck}r^{2t+1-\da}\Big(\,|\nabla\ssolu|^{2}
		 -\ga^{2}|\ssolu|^{2}
		 +r^{-2}\hspace*{0.03cm}|\ssolu|^{2}\,\Big)\,\bigg)\nonumber\\
		&\leq c(n,s)\cdot\bigg(\normltsporn{\srhs}^{2}
		 +\normltsmorn{\ssolu}^{2}
		 +\int_{\Gaprcheck}r^{2s}\Big(\,|\nabla\ssolu|^{2}
		 -\ga^{2}|\ssolu|^{2}\,\Big)\,\bigg).\nonumber
	\end{align}
	Finally, look at
	\begin{align*}
		\Re\int_{\Gaprcheck}r^{2s}\srhs\bssolu
		=\Re\int_{\Gaprcheck}r^{2s}\srhs\bar{v}\,.
	\end{align*}
	Applying Lemma \ref{lem:app-tech_Int-rules}, we obtain	 
	(after some rearrangements)
	\begin{align*}
		&\int_{\Gaprcheck}r^{2s}\hspace*{0.03cm}
		 \Big(\,|\nabla\ssolu|^{2}-\ga^{2}|\ssolu|^{2}\,\Big)\\
		&\hspace*{0.4cm}=-\Re\int_{\Gaprcheck}r^{2s}\srhs\bssolu
		 +s(n+2s-2)\int_{\Gaprcheck}r^{2s-2}|\ssolu|^{2}\\
		&\hspace*{0.8cm}+\int_{\Sp(\rcheck)}\rcheck^{2s}
		 \Big(\,\Re\big(\pr\ssolu\hspace*{0.06cm}\bssolu\hspace*{0.04cm}\big)
		 -s\rcheck^{-1}|\ssolu|^{2}\,\Big)
		 -\int_{S(1)}\Big(\,\Re\big(\pr\ssolu\hspace*{0.06cm}
		 \bssolu\hspace*{0.04cm}\big)
		 -s|\ssolu|^{2}\,\Big)\\
		&\hspace*{0.4cm}\leq c(n,s)\cdot
		 \bigg(\,\int_{\Gaprcheck}\Big(\,r^{2s+2}|\srhs|^{2}
		 +r^{2s-2}|\ssolu|^{2}\,\Big)+\int_{S(1)}\Big(\,|\nabla\ssolu|^{2}
		 +|\ssolu|^{2}\,\Big)+\int_{\Sp(\rcheck)}\rcheck^{2s}
		  \Big(\,|\nabla\ssolu|^{2}
		 +|\ssolu|^{2}\,\Big)\,\bigg)\,,
	\end{align*}
	hence (\,using the trace theorem and Lemma \ref{lem:app-hh_reg}\,)
	\begin{align}\label{equ:app-hh_bdy-int-est}
		\begin{split}
		&\int_{\Gaprcheck}r^{2s}\hspace*{0.03cm}
		 \Big(\,|\nabla\ssolu|^{2}-\ga^{2}|\ssolu|^{2}\,\Big)\\
		&\hspace*{0.6cm}\leq c(n,s)\cdot\bigg(\,\normltsporn{\srhs}^{2}
		 +\normltsmorn{\ssolu}^{2}+\norm{\ssolu}_{\Htwo(\U(1))}^{2}
		 +\int_{\Sp(\rcheck)}\rcheck^{2s}\Big(\,|\nabla\ssolu|^{2}
		 +|\ssolu|^{2}\,\Big)\,\bigg)\qquad\\
		&\hspace*{0.6cm}\leq c(n,s,J)\cdot\bigg(\,\normltsporn{\srhs}^{2}
		 +\normltsmorn{\ssolu}^{2}
		 +\int_{\Sp(\rcheck)}\rcheck^{2s}\Big(\,|\nabla\ssolu|^{2}
		 +|\ssolu|^{2}\,\Big)\,\bigg)
		\end{split}
	\end{align}
	and inserting \eqref{equ:app-hh_bdy-int-est} into \eqref{equ:app-hh_mud-est} 
	we end up with
	\begin{align*}
		\int_{\cU(1)}\theta_{\rcheck}\,r^{2t}\hspace*{0.03cm}
		 \Big(\,|\nabla\ssolu|^{2}+\ga^{2}|\ssolu|^{2}\,\Big)
		\leq c(n,s,J)\cdot\bigg(\normltsporn{\srhs}^{2}
		 +\normltsmorn{\ssolu}^{2}
		 +\int_{\Sp(\rcheck)}\rcheck^{2s}\Big(\,|\nabla\ssolu|^{2}
		 +|\ssolu|^{2}\,\Big)\,\bigg)\,.
	\end{align*}
	Again the lower limit for $\rcheck\To\infty$ of 
	the boundary integral vanishes (\,cf.\;Lemma \ref{lem:app-tech_liminf} and 
	observe that $\ssolu\in\Hgen{}{2}{s-\frac{1}{2}}(\rn)$, since $0<s-t\leq 1/2$ 
	by assumption\,), such that passing to the limit on a suitable subsequence we 
	obtain
	\begin{align*}
		\norm{\ssolu}_{\ltsrn}^{2}
		&\leq c(n,s,J)\cdot\bigg(\,		
		 \int_{\cU(1)}(1-\da)^{-1}\,r^{2t+1-\da}
		 \Big(\,|\nabla\ssolu|^{2}+\ga^{2}|\ssolu|^{2}\,\Big)
		 +\norm{\ssolu}_{\lts(\U(1))}^{2}\,\bigg)\\
		&\leq c(n,s,J)\cdot
		 \bigg(\,\lim_{\rcheck\rightarrow\infty}
		 \int_{\cU(1)}\theta_{\rcheck}\,r^{2t}\hspace*{0.03cm}
		 \Big(\,|\nabla\ssolu|^{2}+\ga^{2}|\ssolu|^{2}\,\Big)
		 +\norm{\ssolu}_{\ltsmo(\U(1))}^{2}\,\bigg)\\
		&\leq c(n,s,J)\cdot\Big(\,\normltsporn{\srhs}^{2}
		 +\normltsmorn{\ssolu}^{2}\,\Big)\,,
	\end{align*}
	showing $\ssolu\in\ltsrn$ and the required estimate.
\end{proof}
\begin{lem}[A-priori estimate]\label{lem:app-hh_a-priori}
	Let $n\in\N$, $t<-1/2$, $1/2<s<1$, and let $J\Subset\reals\sm(0)$ be an interval. 
	Then there exist $c,\da\in(0,\infty)$, such that for all $\beta\in\C_{+}$ 
	with $\beta^{2}=\nu^{2}+i\nu\tau$, $\nu\in J$, $\tau\in(0,1]$ and 
	$\srhs\in\ltsrn$
	\begin{align}\label{equ:app-hh_a-priori-est}
		\begin{split}
		&\normlttrn{(\,\Delta+\beta\,)^{-1}\srhs}
		 +\normhosmtworn{\exp(-i\nu r)(\,\Delta+\beta\,)^{-1}\srhs}\\
		&\hspace*{4.5cm}\leq c\cdot\Big(\normltsrn{\srhs}
		 +\normltomda{(\,\Delta+\beta\,)^{-1}\srhs}\Big)
		\end{split}
	\end{align}
	holds.
\end{lem}
Ikebe and Saito \cite{ikebe_limiting_1972} proved this estimate for the space 
dimension $n=3$ and with $t=-s$, which already shows the result also for any 
$t<-1/2$ as the norms depend monotonic on the parameters $s$ and $t$. For 
arbitrary space dimensions, we follow the proof of Vogelsang 
\cite[Satz 4]{vogelsang_absolute_1982}.
\begin{proof}
	First of all, observe that 
	\begin{align*}
		\maps{\Delta}{\Htworn\subset\ltrn}{\ltrn}{\ssolu}{\Delta\ssolu}	
	\end{align*}
	is self-adjoint and, therefore,
	$\ssolu:=(\,\Delta+\beta\,)^{-1}\srhs\in\Htworn$ is well 
	defined. Moreover, due to the monotone dependence of the 
	norms on the parameters $s$ and $t$, it is enough to concentrate on the case 
	$t=-s$. With $\tssolu:=\exp(-i\nu r)\ssolu$ and $\tsrhs:=\exp(-i\nu r)\srhs$, 
	we have $\tssolu\in\Htwo(\Om)$ and
	\begin{align*}
		\Delta\tssolu+i\nu\Big(\tau\tssolu
		+\frac{n-1}{r}\tssolu
		+2\hspace*{0.25mm}\pr\tssolu\Big)=\tsrhs\,.
	\end{align*}
	Applying Lemma \ref{lem:app-tech_Int-rules} to
	\begin{align*}
		&\Re\int_{\Gaprtil}\tsrhs\,\Big(\,r^{2s-1}\pr\btssolu
		 +\frac{n-1}{2}\,r^{2s-2}\btssolu
		 +\frac{\tau}{2}\,r^{2s-1}\btssolu\,\Big)=\ldots,
	\end{align*}
	with $\rtil>1$ and using the same techniques as in the proof of Lemma 
	\ref{lem:app-hh_pol-dec} we obtain 
	\begin{align*}
		&\frac{1}{2}\int_{\Gaprtil}r^{2s-2}
		 \Big(\,(4s-4)|\pr\tssolu|^{2}-(2s-3)|\nabla\tssolu|^{2}\,\Big)
		 +\frac{1}{2}\int_{\Gaprtil}r^{2s-1}\tau|\nabla\tssolu|^{2}\\
		&\hspace*{0.6cm}=-\Re\int_{\Gaprtil}\tsrhs\,\Big(\,r^{2s-1}\pr\btssolu
		 +\frac{n-1}{2}\,r^{2s-2}\btssolu
		 +\frac{\tau}{2}\,r^{2s-1}\btssolu\,\Big)\\
		&\hspace*{1.2cm}
		 +\frac{n-1}{2}(s-1)(n+2s-4)\int_{\Gaprtil}r^{2s-4}|\tssolu|^{2}
		 +\frac{\tau}{4}(2s-1)(n+2s-3)\int_{\Gaprtil}r^{2s-3}|\tssolu|^{2}\\
		&\hspace*{1.8cm}
		 +\frac{1}{2}\int_{\Sp(\rtil)}\rtil^{2s-1}\left(\,2\,|\pr\tssolu|^{2}
		 +\tau\,\Re\big(\pr\tssolu\hspace*{0.03cm}\btssolu\big)
		 -|\nabla \tssolu|^{2}
		 -\frac{(2s-1)}{\rtil}\,\tau\,|\tssolu|^{2}\,\right)\\
		&\hspace*{2.4cm}+\frac{(n-1)}{2}\int_{\Sp(\rtil)}\rtil^{2s-2}
		 \left(\Re\big(\pr\tssolu\hspace*{0.03cm}\btssolu\big)
		 -\frac{s-1}{\rtil}|\tssolu|^{2}\right)\\
		&\hspace*{3cm}-\frac{1}{2}\int_{\Sp(1)}\Big(\,2\,|\pr\tssolu|^{2}
		 +\tau\,\Re\big(\pr\tssolu\hspace*{0.03cm}\btssolu\big)
		 -|\nabla \tssolu|^{2}-(2s-1)\,\tau\,|\tssolu|^{2}\,\Big)\\
		&\hspace*{3.6cm}-\frac{(n-1)}{2}\int_{\Sp(1)}
		 \Big(\,\Re\big(\pr\tssolu\hspace*{0.03cm}\btssolu\big)
		 -(s-1)|\tssolu|^{2}\,\Big)\,.
	\end{align*}
	Since $4s-4<0$ and $|\pr\tssolu|\leq|\nabla\tssolu|$, the left hand side 
	can be estimated from below
	\begin{align*}
		&\frac{1}{2}\int_{\Gaprtil}r^{2s-2}
		 \Big(\,(4s-4)|\pr\tssolu|^{2}-(2s-3)|\nabla\tssolu|^{2}\,\Big)
		 +\frac{1}{2}\int_{\Gaprtil}r^{2s-1}\tau
		 \hspace*{0.02cm}|\nabla\tssolu|^{2}\\
		&\hspace*{0.6cm}\geq\frac{1}{2}\int_{\Gaprtil}r^{2s-2}
		 \Big(\,(4s-4)-(2s-3)\,\Big)|\nabla\tssolu|^{2}
		 =\left(s-\frac{1}{2}\right)\int_{\Gaprtil}r^{2s-2}|\nabla\tssolu|^{2}\,,
	\end{align*}
	while for the right hand side we obtain
	\begin{align*}
		&-\Re\int_{\Gaprtil}\tsrhs\,\Big(\,r^{2s-1}\pr\btssolu
		 +\frac{n-1}{2}\,r^{2s-2}\btssolu
		 +\frac{\tau}{2}\,r^{2s-1}\btssolu\,\Big)\;+\;\ldots\\
		&\hspace*{0.8cm}
		 \leq\int_{\Gaprtil}r^s|\tsrhs|\,\Big(\,r^{s-1}|\nabla \tssolu|
		 +\frac{n-1}{2}\,r^{s-2}|\tssolu|
		 +\frac{\tau}{2}\,r^{s-1}|\tssolu|\,\Big)\\
		&\hspace*{1.6cm}+c\cdot\left(\,\int_{\Gaprtil}r^{2s-4}|\tssolu|^{2}
		 +\tau\int_{\Gaprtil}r^{s-2}|\tssolu|\,r^{s-1}|\tssolu|\,\right.\\
		&\hspace*{2.4cm}+\left.\int_{\Sp(1)}\Big(\,|\nabla\tssolu|^{2}
		 +|\pr\tssolu\hspace*{0.03cm}\btssolu|+|\tssolu|^{2}\,\Big)
		 +\int_{\Sp(\rtil)}\rtil^{2s-1}\Big(\,|\nabla\tssolu|^{2}
		 +|\pr\tssolu\hspace*{0.03cm}\btssolu|+|\tssolu|^{2}\,\Big)\,\right)\,,
	\end{align*}
	yielding
	\begin{align*}
		&\left(s-\frac{1}{2}\right)\int_{\Gaprtil}r^{2s-2}|\nabla\tssolu|^{2}\\
		&\hspace*{0.6cm}
		 \leq\int_{\Gaprtil}r^s|\tsrhs|\,\Big(\,r^{s-1}|\nabla \tssolu|
		 +\frac{n-1}{2}\,r^{s-2}|\tssolu|+\frac{\tau}{2}\,r^{s-1}|\tssolu|\,\Big)
		 +c\cdot\left(\,\int_{\Gaprtil}r^{2s-4}|\tssolu|^{2}\right.\\
		&\hspace*{1.6cm}+\left.\tau\int_{\Gaprtil}r^{2s-2}|\tssolu|
		 +\int_{\Sp(1)}\Big(\,|\nabla\tssolu|^{2}
		 +|\tssolu|^{2}\,\Big)
		 +\int_{\Sp(\rtil)}\rtil^{2s-1}\Big(\,|\nabla\tssolu|^{2}
		 +|\tssolu|^{2}\,\Big)\,\right)\,.
	\end{align*}
	Here, as well as in the sequel, $c\in(0,\infty)$ denotes a generic 
	constant independent of $\nu$, $\tau$, $\ssolu$ and $\srhs$. According to 
	Lemma \ref{lem:app-tech_liminf} the lower limit for $\rtil\To\infty$ of the 
	last boundary integral vanishes. Thus we may omit it and replace $\Gaprtil$ 
	by $\cU(1)$, such that using Young's inequality we end up with 
	\begin{align*}		
		&\norm{r^{s-1}\,\nabla\tssolu}_{\lt(\cU(1))}^{2}\\
		&\hspace*{0.6cm}\leq c\cdot\Big(\,\norm{r^s\tsrhs}_{\lt(\cU(1))}^{2}
		 +\tau\norm{r^{s-1}\tssolu}_{\lt(\cU(1))}^{2}
		 +\norm{r^{s-2}\tssolu}_{\lt(\cU(1))}^{2}
		 +\int_{\Sp(1)}\Big(\,|\nabla\tssolu|^{2}+|\tssolu|^{2}\,\Big)\,\Big)\\
		&\hspace*{0.6cm}\leq c\cdot\Big(\,\normltrn{\tsrhs}^{2}
		 +\tau\normltsmorn{\tssolu}^{2}+\normltsmtworn{\tssolu}^{2}
		 +\int_{\Sp(1)}\Big(\,|\nabla\tssolu|^{2}+|\tssolu|^{2}\,\Big)\,\Big)\,.
	\end{align*}
	In addition the surface integral is bounded by 
	$\norm{\tssolu}_{\Htwo(\U(1))}^{2}$ (\,trace theorem\,) and 
	Lemma \ref{lem:app-hh_reg} yields,
	\begin{align*}
		\norm{\tssolu}_{\Htwo(\U(2))}
		\leq c\cdot \norm{\tssolu}_{\Htwoms(\rn)}
		\leq c\cdot\left(\,\norm{\tsrhs}_{\lts(\rn)}
		 +\norm{\tssolu}_{\ltms(\rn)}\,\right)\,,
	\end{align*}
	showing
	\begin{align*}		
		\normltsmorn{\nabla\tssolu}^{2}
		&\leq c\cdot\Big(\,\normltrn{\srhs}^{2}
		 +\tau\normltsmorn{\tssolu}^{2}+\normltsmtworn{\tssolu}^{2}
		 +\norm{\tssolu}_{\Htwo(\U(1))}^{2}\\	
		&\leq c\cdot\Big(\,\normltrn{\srhs}^{2}
		 +\tau\normltsmorn{\ssolu}^{2}+\normltsmtworn{\ssolu}^{2}
		 +\normltmsrn{\ssolu}^{2}\,\Big)\,.
	\end{align*}
	By the differential equation we see
	\begin{align*}
		\normltrn{\srhs}\normltrn{\ssolu}
		\geq\big|\,\Im\scpltrn{\srhs}{\ssolu}\,\big|
		=\big|\,i\nu\tau\,\scpltrn{\ssolu}{\ssolu}\,\big|
		=\tau\,|\,\nu\,|\normltrn{\ssolu}^2\,,
	\end{align*}
	hence ($-s>s-2$)
	\begin{align}\label{equ:app-hh_first-apr-est}
		\begin{split}
			\hspace*{-1.2cm}\norm{\exp(-i\nu r)\ssolu}_{\hosmtwo(\rn)}
				&\leq c\cdot\Big(\normltsmtworn{\tssolu}
				 +\normltsmorn{\nabla \tssolu}\Big)\\
				&\leq c\cdot\Big(\,\normltsrn{\srhs}^{2}
				 +\tau\,\normltsmorn{\ssolu}^{2}
				 +\normltmsrn{\ssolu}^{2}\,\Big)\\
				&\leq c\cdot\Big(\,\normltsrn{\srhs}
				 +\normltmsrn{\ssolu}\,\Big)\,,
		\end{split}
	\end{align}
	and it remains to estimate $\normltmsrn{\ssolu}$. For that we calculate
	\begin{align*}
		\Im &\int_{\Gaprtil}\tsrhs\btssolu
		 =\Im\int_{\Gaprtil}\Delta\tssolu\btssolu
	  	     +\int_{\Gaprtil}\nu\left(\,\tau\tssolu
	  	     +\frac{n-1}{r}\tssolu\,\right)\,\btssolu
		     +2\nu\,\Re\int_{\Gaprtil}\pr\tssolu\hspace*{0.03cm}\btssolu
		     =\,\ldots\,,
	\end{align*}
	using Lemma \ref{lem:app-tech_Int-rules} and obtain
	\begin{align*}
		&\nu\int_{\Gaprtil}r^{-2s}\Big(\,(2s-1)|\tssolu|^2
		 +\tau r|\tssolu|^2\,\Big)\\
		&\hspace*{0.8cm}=\Im\int_{\Gaprtil}r^{1-2s}\tsrhs\btssolu
		 -(2s-1)\int_{\Gaprtil}r^{-2s}\,
		 \Im\big(\pr\tssolu\hspace*{0.03cm}\btssolu\big)\\
		&\hspace*{1.6cm}+\int_{\Sp(\rtil)}r^{1-2s}\Big(\,\tau|\tssolu|^2
		 +\Im\big(\pr\tssolu\hspace*{0.03cm}\btssolu\big)\,\Big)
		 -\int_{\Sp(1)}\Big(\,\tau|\tssolu|^2
		 +\Im\big(\pr\tssolu\hspace*{0.03cm}\btssolu\big)\,\Big)\\
		&\hspace*{0.8cm}\leq\int_{\Gaprtil}r^{s}|\tsrhs|\,r^{1-3s}|\tssolu|
		 +(2s-1)\int_{\Gaprtil}r^{s-1}|\pr\tssolu|\,r^{1-3s}|\tssolu|\\
		&\hspace*{1.6cm}+c\cdot\left(\,\int_{\Sp(\rtil)}r^{1-2s}
		 \Big(\,|\tssolu|^2+|\pr\tssolu\hspace*{0.03cm}\btssolu|\,\Big)
		 +\int_{\Sp(1)}\Big(\,|\tssolu|^2
		 +|\pr\tssolu\hspace*{0.03cm}\btssolu|\,\Big)\,\right)\\
		&\hspace*{0.8cm}\leq\Big(\,\norm{r^{s}\tsrhs}_{\lt(\Gaprtil)}
		 +(2s-1)\norm{r^{s-1}\,\nabla\tssolu}_{\lt(\Gaprtil)}\,\Big)
		 \cdot\norm{r^{1-3s}\tssolu}_{\lt(\Gaprtil)}\\
		&\hspace*{1.6cm}+c\cdot\left(\,\int_{\Sp(1)}\Big(\,|\tssolu|^2
		 +|\nabla\tssolu|^2\,\Big)
		 +\int_{\Sp(\rtil)}\rtil^{1-2s}\Big(\,|\tssolu|^2
		 +|\nabla\tssolu|^2\,\Big)\right)\,.
	\end{align*}
	As before the lower limit for $\rtil\To\infty$ of the last boundary 
	integral vanishes (\,cf. Lemma \ref{lem:app-tech_liminf} and observe that 
	$\tssolu\in\Htwo(\Om)$, $s>0$\,), such that we may omit it and replace 
	$\Gaprtil$ by $\cU(1)$, yielding (\,with \eqref{equ:app-hh_first-apr-est}\,)
	\begin{align*}
		&\norm{r^{-s}\tssolu}_{\lt(\cU(1))}^2\\
		&\hspace*{0.6cm}
		 \leq c\cdot\bigg(\,\Big(\,\norm{r^{s}\tsrhs}_{\lt(\cU(1))}
		 +\norm{r^{s-1}\nabla\tssolu}_{\lt(\cU(1))}\,\Big)
		  \cdot\norm{r^{1-3s}\tssolu}_{\lt(\cU(1))}
		 +\int_{\Sp(1)}\Big(\,|\tssolu|^2+|\nabla\tssolu|^2\,\Big)\,\bigg)\\
		&\hspace*{0.6cm}\leq c\cdot\Big(\,\Big(\,\normltsrn{\tsrhs}
		 +\normltsmorn{\nabla\tssolu}\,\Big)\cdot\normltomthreesrn{\tssolu}
		 +\int_{\Sp(1)}\Big(\,|\tssolu|^2+|\nabla\tssolu|^2\,\Big)\,\Big)\,.
	\end{align*}
	As the surface integral is bounded by 
	$\norm{\tssolu}_{\Htwo(\U(1))}^2$ (\,trace theorem\,) and with 
	\eqref{equ:app-hh_first-apr-est} we obtain
	\begin{align*}
		\norm{\tssolu}_{\ltms(\rn)}^{2}
			&\leq c\cdot\left(\,\Big(\,\norm{\tsrhs}_{\lts(\rn))}
		 	 +\norm{\nabla\tssolu}_{\ltsmo(\rn)}\,\Big)
		 	 \cdot\normltomthreesrn{\tssolu}
		 	 +\norm{\tssolu}_{\Htwo(\U(2))}^2\,\right)\\
			&\leq c\cdot\left(\,\Big(\,\normltsrn{\tsrhs}
			 +\normltmsrn{\tssolu}\,\Big)
			 \cdot\normltomthreesrn{\tssolu}
			 +\norm{\tssolu}_{\Htwo(\U(2))}^2\,\right)\,,
	\end{align*}
	hence (\,Young's inequality\,)
	\begin{align*}
		\norm{\tssolu}_{\ltms(\rn)}^{2}
			\leq c\cdot\left(\,\normltsrn{\tsrhs}+\normltomthreesrn{\tssolu}
			+\norm{\tssolu}_{\Htwo(\U(2))}^2\,\right)\,,
	\end{align*}
	Finally, using once again Lemma \ref{lem:app-hh_reg}, we arrive at
	\begin{align*}
		\norm{\tssolu}_{\ltms(\rn)}^2
			\leq c\cdot\Big(\,\normltsrn{\tsrhs}
			+\normltomthreesrn{\tssolu}\,\Big)\,, 
	\end{align*}
	which together with \eqref{equ:app-hh_first-apr-est} and 
	Lemma \ref{lem:app-tech_weighted-est} implies
	\begin{align}\label{equ:app-hh_mud-apr-est}
		\normltmsrn{\ssolu}+\normhosmtworn{\exp(-i\nu r)\,\ssolu}
			\leq c\cdot\Big(\,\normltsrn{\srhs}+\normltomda{\ssolu}\,\Big)
	\end{align}
	with $c,\da>0$ independent of $\nu$, $\tau$, $\ssolu$ and $\srhs$. 
\end{proof}
%
%
\section{Proofs in the Case of the Time-Harmonic Maxwell Equations}
\label{sec:proofs-max}
%
%
This section deals with the proofs of the decomposition lemma, the polynomial 
decay, and the a-priori estimate, which we skipped in the main part.

\begin{proof}[\rm\bf Proof of Lemma \ref{lem:pol-est_decomp}]
	We start with $\solu=\eta\solu+\ceta\solu$, noting that $\ceta\solu\in\Rt$. 
	Moreover,
	\begin{align*}
		\Rot\ceta\solu
			=\mathrm{C}_{\Rot,\ceta}\solu+\ceta\Rot\solu
			=\mathrm{C}_{\Rot,\ceta}\solu-i\ceta\La\rhs-i\om\ceta\La\solu
	\end{align*}
	and we have
	\begin{align*}
		\big(\hspace*{0.03cm}\Rot+\,i\om\Laz\,\big)\hspace*{0.03cm}\ceta\solu
			=\big(\,\mathrm{C}_{\Rot,\ceta}
			 -i\om\ceta\hat{\La}\,\big)\hspace*{0.03cm}\solu
			 -i\ceta\La\rhs=\rhs_{1}\in\lts\,,
	\end{align*}
	since $\supp\nabla\ceta$ is compact and $t+\ka\geq s$. According to 
	\cite[Theorem 4]{weck_generalized_1994}
	\begin{align*}
		\rhs_{1}=\rhs_{\sfR}+\rhs_{\sfD}+\rhs_{\calS}
		 \in\zrs\dpl\zds\dpl\calS_{s}	
	\end{align*}
	holds and we obtain
	\begin{align*}
		i\om\ceta\Laz\solu
			=\rhs_{1}-\Rot\ceta\solu
			=\rhs_{\sfD}-\Rot\ceta\solu+\rhs_{\sfR}+\rhs_{\calS}	\,.
	\end{align*}
	Defining\\[-8pt]
	\begin{enumerate}[label=$\circ$,itemsep=3pt]
		\item $\dsp\solu_{1}
			   :=-\frac{i}{\om}\Laz^{-1}\big(\,\rhs_{\sfR}+\rhs_{\calS}\,\big)
			   \in\rs$,
		\item $\dsp\tsolu
		  	   :=\ceta\solu-\solu_{1}
		  	   =\frac{i}{\om}\Laz^{-1}
		  	   \big(\hspace*{0.03cm}\Rot\ceta\solu-\rhs_{\sfD}\,\big)
		  	   \in\Rt\cap\zdt$,\\[-7pt]
	\end{enumerate}
	\cite[Lemma 4.2]{kuhn_regularity_2010} shows $\tsolu\in\Hot$ 
	and we have
	\begin{align*}
		\big(\hspace*{0.03cm}\Rot+\,i\om\Laz\,\big)\hspace*{0.03cm}\tsolu
		 =\Rot\big(\,\ceta\solu-\solu_{1}\,\big)+i\om\Laz\tsolu
		 =\rhs_{\sfD}+\frac{i}{\om}\,{\tLaz}
		  \parbox[t]{0.35cm}{\vspace*{-0.4cm}${}^{-1}$}
		  \Rot\hspace*{0.01cm}\rhs_{\calS}
		 =\rhs_{2}\in\zds\,.
	\end{align*}
	Next we solve 
	$\big(\hspace*{0.03cm}\Rot+\,1\,\big)\hspace*{0.03cm}\solu_{2}=\rhs_{2}$. 
	Using the Fourier transformation we look at
	\begin{align*}
		\hsolu_{2}:=\big(\,1+r^{2}\,\big)^{-1}
		 \big(\,1-ir\,\Xi\,\big)\,\calF(\rhs_{2})
	\end{align*}
	Since $s>1/2$ and $\rhs_{2}\in\lts$, we obtain $\hsolu\in\lto$ and
	hence $\solu_{2}:={\dsp\calF}^{\hspace*{0.03cm}-1}(\hsolu_{2})\in\Ho$. 
	Moreover, we have $\calF\big(\calF(\rhs_{2})\big)=\calP(\rhs_{2})\in\lts$ 
	(\,$\calP$: parity operator\,) yielding $\calF(\rhs_{2})\in\Hs$ and 
	as product of an $\Hs$-field with bounded $\ci$-functions, 
	$\hsolu\in\Hs$ 
	\big(\,cf. \cite[Lemma 3.2]{wloka_partial_1987}\,\big), hence 
	$\solu_{2}\in\lts$. In addition a straight forward calculation shows 
	$\calF\big(\,\big(\hspace*{0.03cm}\Rot+\,1\,\big)\hspace*{0.03cm}\solu_{2}
	\,\big)=\calF(\rhs_{2})$, 
	which by \cite[Lemma 4.2]{kuhn_regularity_2010} implies 
	\begin{align*}
		\big(\hspace*{0.03cm}\Rot+\,1\,\big)\hspace*{0.03cm}\solu_{2}=\rhs_{2}
		\qquad\text{and}\qquad
		\solu_{2}\in\Hos\cap\zds\,.
	\end{align*}
	Then (\,$t\leq s$\,)
	\begin{align*}
		\solu_{3}:=\tsolu-\solu_{2}
		\in\Hot\cap\zdt
	\end{align*}
 	satisfies 
 	\begin{align*}
 		\big(\hspace*{0.03cm}\Rot +\,i\om\Laz\,\big)\hspace*{0.03cm}\solu_{3}
 		&=\big(\hspace*{0.03cm}\Rot+\,i\om\Laz\,\big)\hspace*{0.03cm}\tsolu
 		 -\big(\hspace*{0.03cm}\Rot+\,i\om\Laz\,\big)\hspace*{0.03cm}\solu_{2}\\
 		&=\rhs_{2}-\big(\hspace*{0.03cm}\Rot+\,1\,\big)\hspace*{0.03cm}\solu_{2}
 		 +\big(\,1-i\om\Laz\,\big)\hspace*{0.03cm}\solu_{2}
 		 =\big(\,1-i\om\Laz\,\big)\hspace*{0.03cm}\solu_{2}
 		 \in\Hos\cap\zds
 	\end{align*}
	and using once more \cite[Lemma 4.2]{kuhn_regularity_2010} we get
	\begin{align*}
		\solu_{3}\in\Htwot\cap\zdt\,.	
	\end{align*}
	Finally
	\begin{align*}
		\Delta\hspace*{0.03cm}\solu_{3}
		&=\Rot\big(\hspace*{0.03cm}\Rot\hspace*{0.03cm}\solu_{3}\,\big)
	 	 =\big(\,1-i\om\tLaz\,\big)\Rot\solu_{2}-i\om\tLaz\Rot\solu_{3}\\
		&=\big(\,1-i\om\tLaz\,\big)\big(\,\rhs_{2}-\solu_{2}\,\big)
		  -i\om\tLaz\big(\,\big(\,1-i\om\Laz\,\big)\hspace*{0.03cm}\solu_{2}
		  -i\om\Laz\hspace*{0.03cm}\solu_{3}\,\big)\\
		&=\big(\,1-i\om\tLaz\,\big)\,\rhs_{2}
		  -\big(\,1+\om^{2}\eps_{0}\mu_{0}\,\big)\hspace*{0.03cm}\solu_{2}
		  -\om^{2}\eps_{0}\mu_{0}\hspace*{0.03cm}\solu_{3}
	\end{align*}
	holds, and hence
	\begin{align*}
		\big(\,\Delta+\om^{2}\eps_0\mu_0\,\big)\hspace*{0.03cm}\solu_{3}
		 =\big(\,1-i\om\tLaz\,\big)\,\rhs_{2}
		  -\big(\,1+\om^{2}\eps_0\mu_0\,\big)\hspace*{0.03cm}\solu_{2}.
	\end{align*}
	The asserted estimates follow by straight forward calculations using 
	\cite[Lemma 4.2]{kuhn_regularity_2010} and the continuity of the projections
	from $\lts$ into $\zrs$,\;$\zds$ and $\calS_{s}$.
\end{proof}
\begin{proof}[\rm\bf Proof of Lemma \ref{lem:pol-est_polynomial-decay}]
	As for $t\geq s-1$ there is nothing to prove, we concentrate on
	\begin{align*}
		\solu\in\Rtom\quad\;\;\text{with}\quad-1/2<t<s-1\,.
	\end{align*} 
	Therefore assume first that in addition
	\begin{align*}
		s-\ka<t\;\;\Longrightarrow\;	\;t<s<t+\ka\,.
	\end{align*}
 	Then we may apply Lemma \ref{lem:pol-est_decomp} and decompose the field 
 	$u$ in
 	\begin{align*}
 		\solu=\eta \solu+\solu_{1}+\solu_{2}+\solu_{3}\,,
 	\end{align*} 
 	with $\eta \solu+\solu_{1}+\solu_{2}\in\Rsom$ and $\solu_{3}\in\Htwot$ 
 	satisfying 
 	$\big(\,\Delta+\om^2\eps_{0}\mu_{0}\,\big)\hspace*{0.03cm}\solu_{3}\in\lts$.
	Thus the polynomial decay for the Helmholtz equation (\,cf. Lemma \ref{lem:app-hh_pol-dec}\,) shows
 	\begin{align*}
 		\solu_{3}\in\Htwosmo
 		\qquad\text{and}\qquad
 		\normHtwosmows{\solu_{3}}
 		 \leq c\cdot\Big(\,\normltsws{\big(\,\Delta+\om^2\eps_{0}\mu_{0}\,\big)
 	     \hspace*{0.03cm}\solu_{3}}+\normltsmtwows{\solu_{3}}\,\Big)\,,
 	\end{align*}
	$c=c(s,J)>0$, yielding 
	$\solu=\eta \solu+\solu_{1}+\solu_{2}+\solu_{3}\in\Rsmoom$. 
	Moreover, using the estimates of Lemma \ref{lem:pol-est_decomp} we obtain
	uniformly with respect to\,\;$\om$,\;$\solu$, and $\rhs$
	\begin{align*}
		\normRsmoom{\solu}
		&\leq c\cdot\Big(\,\normltsom{\rhs}+\normltsmkaom{\solu}
		 +\normltsmows{\solu_{3}}\,\Big)\\
		&\leq c\cdot\Big(\,\normltsom{\rhs}+\normltsmkaom{\solu}
		 +\normltsws{\big(\,\Delta+\om^2\eps_{0}\mu_{0}\,\big)
 	     \hspace*{0.03cm}\solu_{3}}+\normltsmtwows{\solu_{3}}\,\Big)\\
		&\leq c\cdot\Big(\,\normltsom{\rhs}
		+\normltsmmom{\solu}\,\Big)\,,
	\end{align*}
	where $m:=\min\{\ka,2\}$ and applying Lemma \ref{lem:app-tech_weighted-est} 
	we end up with:
	\begin{align*}
		\normRsmoom{\solu}
		\leq c\cdot\Big(\,\normltsom{f}+\normltomda{\solu}\,\Big)\,,
	\end{align*}
	for $c,\da\in(0,\infty)$ independent	 of\,\;$\om$,\;$\solu$ and $\rhs$. 
	So let's switch to the case  
	\begin{align*}
		t\leq s-\ka\;\;\Longrightarrow\;\;t+\ka\leq s\,.
	\end{align*}
 	Here, the idea is to approach $s$ by overlapping intervals to which the first 
 	case is applicable. For that, we choose some $\hat{k}\in\N$, such that with 
 	$\ga:=(\ka-1)/2>0$ we have
	\begin{align*}
		t+\ka+\big(\,\hat{k}-1\,\big)\cdot\ga\leq s\leq t+\ka+\hat{k}\cdot\ga\,,
	\end{align*}
	and for $k=0,1,\ldots,\hat{k}$ we define
	\begin{align*}
		t_{k}:=t+k\cdot\ga
		\qquad\text{as well as}\qquad 
		s_{k}:=t_{k+1}+1=t_{k}+(\ka+1)/2\,.
	\end{align*}
	Then (\,as $\ka>1$\,)
	\begin{align*}
		t_{k+1}<s_{k}=t_{k+1}+1=t+\ka+(k-1)\cdot\ga\leq s
		\quad\,\text{and}\,\quad 
		t_{k}<t_{k+1}+1=s_{k}=t_{k}+(\ka+1)/2<t_{k}+\ka\,,	
	\end{align*}
	such that we can successively apply the first case, ending up with 
	$\solu\in\Rgen{}{}{s_{\hat{k}}-1}(\Om).$
	If $s=s_{\hat{k}}$ we are done. Otherwise we choose 
	$t_{\hat{k}+1}:=s_{\hat{k}}-1$ and apply the first case once more, since 
	\begin{align*}
		t_{\hat{k}+1}<s_{\hat{k}}<s
		\leq t+\ka+\hat{k}\cdot\ga=t_{\hat{k}+1}+\ka\,.
	\end{align*} 
	Either way we obtain $\solu\in\Rsmoom$ and now the estimate follows as in 
	the first case.
\end{proof}
\begin{proof}[\rm\bf Proof of Lemma \ref{lem:pol-est_a-priori}]
	Without loss of generality we assume $s\in(1/2,1)$. Then 
	$s\in\reals\sm\mathbb{I}$ with $0<s<\ka$ and we can apply Lemma 
	\ref{lem:pol-est_decomp} (\,with $t=0$\,) to decompose $\solu:=
	\sol\rhs\in\Rztom$ into
	\begin{align*}
		\solu=\eta\solu+\solu_{1}+\solu_{2}+\solu_{3}	
	\end{align*}
	with $\solu_{3}\in\Htwo$ solving
	\begin{align*}
		\big(\,\Delta+\om^2\eps_{0}\mu_{0}\,\big)\hspace*{0.03cm}\solu_{3}
 		 =\big(\,1-i\om\tLaz\,)\,\rhs_{2}
 		 -\big(\,1+\om^2\eps_{0}\mu_{0}\,\big)
 	     \hspace*{0.03cm}\solu_{2}=:\rhs_{3}\in\lts\,, 	
	\end{align*}
	where $\rhs_2$ is defined as in Lemma \ref{lem:pol-est_decomp}. Moreover, 
	the estimates from Lemma \ref{lem:pol-est_decomp} along with
	\begin{align*}
		\big(\hspace*{0.02cm}\Rot
		 -\,i\om\hspace*{0.03cm}\sqrt{\eps_{0}\mu_{0}}\;\Xi\;\big)
		 \hspace*{0.03cm}\solu
		=-i\La\rhs-i\om\hspace*{0.02cm}\big(\,\Laz
		 +\sqrt{\eps_{0}\mu_{0}}\;\Xi\;\big)
		 \hspace*{0.03cm}\solu-i\om\hat{\La}\solu
	\end{align*}
	yield
	\begin{align}\label{equ:proofs-max_a-priori-mud}
		\begin{split}
			&\hspace*{-0.6cm}\normRtom{\solu}
		 	 +\normltsmoom{\big(\,\Laz+\sqrt{\eps_0\mu_0}\;\Xi\,\big)
	     	 \hspace*{0.02cm}\solu}\\
	    	&\leq c\cdot\Big(\,\normRtom{\solu}
	    	 +\normltsmoom{\big(\hspace*{0.02cm}\Rot-\,i\om\hspace*{0.03cm}
	     	 \sqrt{\eps_{0}\mu_{0}}\;\Xi\;\big)\hspace*{0.03cm}\solu}
	      	 +\normltsom{\rhs}+\normltsmkaom{\solu}\,\Big)\\
	    	&\leq c\cdot\Big(\,\normlttws{\solu_{3}}
	         +\normltsmows{\big(\hspace*{0.02cm}\Rot-\,i\om\hspace*{0.03cm}
	      	 \sqrt{\eps_{0}\mu_{0}}\;\Xi\;\big)\hspace*{0.03cm}\solu_{3}}
	      	 +\normltsom{\rhs}+\normltsmkaom{\solu}\,\Big)\,,
		\end{split}
	\end{align}
	with $c=c(s,t,J)>0$. Due to the monotonicity of 
	the norms with respect to $t$ and $s$, we may assume $t$ and $s$ 
	to be close enough to $-1/2$\;resp.\;$1/2$ such that $1<s-t<\ka$ holds. 
	Hence, the assertion follows by \eqref{equ:proofs-max_a-priori-mud} and 
	Lemma \ref{lem:app-tech_weighted-est}, if we can show
	\begin{align*}
		\normlttws{\solu_{3}}
	    +\normltsmows{\big(\hspace*{0.02cm}\Rot-\,i\om\hspace*{0.03cm}
	    \sqrt{\eps_{0}\mu_{0}}\;\Xi\;\big)\hspace*{0.03cm}\solu_{3}}
	    \leq c\cdot\Big(\,\normltsom{\rhs}+\normltsmkaom{\solu}\,\Big)\,,
	\end{align*}
	with $c\in(0,\infty)$ independent of $\om$,\,$\solu$ and $\rhs$. Therefore 
	note thatthe self-adjointness of the laplacian 
	$\map{\Delta}{\Htwo\subset\lt}{\lt}$ yields 
	$\big(\,\Delta+\om^2\eps_{0}\mu_{0}\,\big)^{-1}\rhs_{3}=\solu_{3}$
	and applying Lemma \ref{lem:app-hh_a-priori} componentwise, we obtain
	\begin{align*}
		\normlttws{\solu_{3}}
		 +\normhosmtwows{\exp(\hspace*{0.02cm}-\,i\la\hspace*{0.02cm}
		  \sqrt{\eps_{0}\mu_{0}}\hspace*{0.03cm}r)\hspace*{0.03cm}\solu_{3}}
		  \leq c\cdot\Big(\normltsws{\rhs_{3}}+\normltomda{\solu_{3}}\Big)\,.
	\end{align*}
	With 
	$\Rot\big(\exp(\hspace*{0.02cm}-\,i\la\hspace*{0.02cm}\sqrt{\eps_{0}\mu_{0}}
	 \hspace*{0.03cm}r)\hspace*{0.03cm}\solu_{3}\,\big)
     =\exp(\hspace*{0.02cm}-\,i\la\hspace*{0.02cm}\sqrt{\eps_{0}\mu_{0}}
	 \hspace*{0.03cm}r\hspace*{0.02cm})\,\big(\,\Rot
	 -i\la\sqrt{\eps_{0}\mu_{0}}\;\Xi\;\big)\hspace*{0.03cm}\solu_{3}$
	this leads to
	\begin{align}\label{equ:proofs-max_a-priori-mud2}
		\begin{split}
			&\normlttws{\solu_{3}}
		     +\normltsmows{\big(\hspace*{0.02cm}\Rot-\,i\la\hspace*{0.03cm}
		     \sqrt{\eps_{0}\mu_{0}}\;\Xi\;\big)\hspace*{0.03cm}\solu_{3}}\\
		    &\hspace*{0.6cm}\leq \normlttws{\solu_{3}}
		     +\normltsmows{\Rot\big(\exp(\hspace*{0.02cm}-\,i\la\hspace*{0.02cm}
		     \sqrt{\eps_{0}\mu_{0}}\hspace*{0.03cm}r)
		     \hspace*{0.03cm}\solu_{3}\,\big)}\\
		    &\hspace*{0.6cm}\leq \normlttws{\solu_{3}}
			 +\normhosmtwows{\exp(\hspace*{0.02cm}-\,i\la\hspace*{0.02cm}
			 \sqrt{\eps_{0}\mu_{0}}\hspace*{0.03cm}r)\hspace*{0.03cm}\solu_{3}}
			 \leq c\cdot\Big(\normltsws{\rhs_{3}}
			 +\normltomda{\solu_{3}}\Big)\,,
		\end{split}
	\end{align}
	where $c>0$ is not depending on $\om$, $\solu_{3}$ and $\rhs_{3}$. But we 
	would like to estimate $\big(\hspace*{0.02cm}\Rot-\,i\om\hspace*{0.03cm}
	\sqrt{\eps_{0}\mu_{0}}\;\Xi\;\big)\hspace*{0.03cm}\solu_{3}$. For 
	that we need some additional arguments, starting with the observation that
	\begin{align*}
		\om=|\la|\,\big(\,1+(\sigma/\la)^2\,\big)^{1/4}\cdot
		\multif{\exp(i\varphi/2)}{\la>0}{\exp(i(\varphi/2+\pi))}{\la<0}
		\text{ with }\;
		\varphi:=\arctan(\sigma/\la)\in\Big(-\frac{\pi}{2},\frac{\pi}{2}\Big)\,,
	\end{align*}
	hence $|\Re(\om)|\geq\sqrt{2}/2\cdot|\la|$. Then
	$|\om+\la|\geq\sqrt{3/2}\cdot|\la|$ and we have 
	\begin{align*}
		|\om-\la|^2
		    =\left|\frac{\om^{2}-\la^{2}}{\om+\la}\right|^2
		    =\left|\frac{i\sigma\la}{\om+\la}\right|^2
		    \leq\frac{2}{3}\cdot\sigma^2.
	\end{align*}
	From this and the resolvent estimate 
	\begin{align*}
		\normltws{\rhs_{3}}=\normltws{(\,\Delta+\om^2\eps_0\mu_0\,)\,\solu_{3}}
		\geq|\Im(\om^2\eps_{0}\mu_{0})|\cdot\normltws{\solu_3}
		=\eps_{0}\mu_{0}\hspace*{0.02cm}
		 \sigma\hspace*{0.02cm}|\la|\cdot\normltws{\solu_3},
	\end{align*}
	we obtain (\,$s>1/2$\,)
	\begin{align*}
		\normltsmows{\big(\hspace*{0.02cm}\Rot-\,i\om\hspace*{0.03cm}
	     \sqrt{\eps_{0}\mu_{0}}\;\Xi\,\big)\hspace*{0.03cm}\solu_{3}}
	    &\leq\normltsmows{\big(\hspace*{0.02cm}\Rot-\,i\la
	     \hspace*{0.03cm}\sqrt{\eps_{0}\mu_{0}}\;\Xi\;\big)
	     \hspace*{0.03cm}\solu_{3}}+\normltsmows{(\hspace*{0.02cm}\om
	     -\la\hspace*{0.02cm})\hspace*{0.02cm}
	     \sqrt{\eps_{0}\mu_{0}}\;\Xi\,\solu_{3}}\\
	    &\leq \normltsmows{\big(\hspace*{0.02cm}\Rot-\,i\la
	     \hspace*{0.03cm}\sqrt{\eps_{0}\mu_{0}}\;\Xi\,\big)
	     \hspace*{0.03cm}\solu_{3}}+c\cdot|\la|^{-1}\normltsws{\rhs_{3}}	\,,  	
	\end{align*}
	such that with \eqref{equ:proofs-max_a-priori-mud2} and the estimates from 
	Lemma \ref{lem:pol-est_decomp} uniformly with respect to\,\;$\om$,\,$\solu$ 
	and $\rhs$
	\begin{align*}
		\normlttws{\solu_{3}}
	    +\normltsmows{\big(\hspace*{0.02cm}\Rot-\,i\om\hspace*{0.03cm}
	    \sqrt{\eps_{0}\mu_{0}}\;\Xi\,\big)\hspace*{0.02cm}\solu_{3}}
	    \leq c\cdot\Big(\,\normltsom{\rhs}+\normltsmkaom{\solu}\,\Big)\,.
	\end{align*}
\end{proof}
%
%
\end{document}